\documentclass[reqno]{amsart}

\usepackage{macros} 

\toggletrue{fact}

\numberwithin{equation}{subsection}

\makeatletter
\def\enumfix{%
\if@inlabel
 \noindent \par\nobreak\vskip-\topsep\hrule\@height\z@
\fi}

\let\olditemize\itemize
\def\itemize{\enumfix\olditemize}

\makeatother

\begin{document}

\title{Symmetries of the Cyclic Nerve}

\author{David Ayala, Aaron Mazel-Gee, and Nick Rozenblyum}

\date{\today}

\begin{abstract}
We undertake a systematic study of the Hochschild homology, i.e.\! (the geometric realization of) the cyclic nerve, of $(\infty,1)$-categories (and more generally of category-objects in an $\infty$-category), as a version of factorization homology.
In order to do this, we codify $(\infty,1)$-categories in terms of quiver representations in them.
By examining a universal instance of such Hochschild homology, we explicitly identify its natural symmetries, and construct a non-stable version of the cyclotomic trace map.
Along the way we give a unified account of the cyclic, paracyclic, and epicyclic categories.
We also prove that this gives a combinatorial description of the $n=1$ case of factorization homology as presented in~\cite{AFR2}, which parametrizes $(\infty,1)$-categories by solidly 1-framed stratified spaces.
\end{abstract}

\maketitle

\setcounter{tocdepth}{2}
\tableofcontents
\setcounter{tocdepth}{2}

\setcounter{section}{-1}

\section{Introduction}
\label{section.intro}

\startcontents[sections]


\subsection{Hochschild homology and its symmetries}

In this paper, we undertake a systematic study of (\bit{non-stable}) \bit{Hochschild homology}
\[
\sHH(\cC)
~\in~
\Spaces
\]
of an $(\infty,1)$-category $\cC$,
i.e.\! (the geometric realization of) its \bit{cyclic nerve} (see \Cref{dHH}).
Here we state those of our three main results that concern (possibly noninvertible) symmetries of $\sHH(\cC)$; we discuss our two other main results in \Cref{subsection.intro.quiv} and \Cref{subsection.connection.to.fact.hlgy}, respectively.
Throughout, we write
\[
\WW := \TT \rtimes \Nx
\]
for the \bit{Witt monoid}, the semidirect product monoid with respect to the isogenic action $\Nx \lacts \TT$ (see \Cref{d4}).
We write $\Mod_{\WW^\op}(\Spaces)$ for the $\infty$-category of (left) $\WW^\op$-modules (or equivalently right $\WW$-modules) in spaces.)

\begin{maintheorem}
[\Cref{t73}\Cref{theorem A}]
\label{mainthm.WWop.module.structure}
For any $(\infty,1)$-category $\cC$, the Hochschild homology of $\cC$ canonically admits the structure of a $\WW^\op$-module:
\[
\Bigl(
\WW^{\op}
~\lacts~
\sHH(\cC)
\Bigr)
~\in~
\Mod_{\WW^{\op}}(\Spaces)
~.
\]
\end{maintheorem}

We construct a canonical map
\[
\End_\cC
\xra{\trace}
\sHH(\cC)
\]
from the moduli space of endomorphisms in $\cC$ (see \Cref{t251}(1)).\footnote{This map can reasonably be regarded as a \bit{trace} map: $\sHH(\cC)$ is a universal receptacle of traces of endomorphisms in $\cC$.}
In particular, there results a composite map
\begin{equation}
\label{map.from.space.of.objs.to.HH}
\Obj(\cC)
\xra{~c~ \mapsto ~\id_c~}
\End_\cC
\xra{~\trace~}
\sHH(\cC)
\end{equation}
from the moduli space of objects in $\cC$ (see \Cref{t251}(2)).

\begin{maintheorem}
[\Cref{t73}\Cref{theorem B}]
\label{mainthm.WWinvariance.of.trace.map}
The map \Cref{map.from.space.of.objs.to.HH} is canonically invariant with respect to the $\WW^\op$-action of \Cref{mainthm.WWop.module.structure}: it canonically lifts to a map
\begin{equation}
\label{map.from.space.of.objs.to.WWop.invts.of.HH}
\Obj(\cC)
\longra
\sHH(\cC)^{\htpy \WW^\op}
~.
\end{equation}
\end{maintheorem}

The map \Cref{map.from.space.of.objs.to.WWop.invts.of.HH} in \Cref{mainthm.WWinvariance.of.trace.map} may be referred to as a \bit{non-stable cyclotomic trace} map (as discussed further in \Cref{subsection.motivation.Kthy.and.traces}). To make this connection precise, we give the following description of $\Mod_{\WW^\op}(\Spaces)$ in terms of equivariant homotopy theory.

\begin{maintheorem}
[\Cref{t44}]
\label{mainthm.WWop.modules.are.genuine}
The $\infty$-category $\Mod_{\WW^\op}(\Spaces)$ of $\WW^\op$-modules is equivalent to the $\infty$-category of $\Nx$-fixed proper-genuine $\TT$-modules:
\[
\Mod_{\WW^\op}(\Spaces)
~\simeq~
\left( \Mod^{\gen^<}_\TT(\Spaces) \right)^{\htpy \Nx}
~=:~
\Cyclo^\unst(\Spaces)
~.
\]
In particular, via \Cref{mainthm.WWop.module.structure}, the Hochschild homology of an $\infty$-category $\cC$ may be canonically regarded as an $\Nx$-fixed proper-genuine $\TT$-module:
\[
\sHH(\cC)
~\in~
\left( \Mod^{\gen^<}_\TT \right)^{\htpy \Nx}
~.
\]
\end{maintheorem}

\noindent Thus, the $\infty$-category $\Mod_{\WW^\op}(\Spaces)$ of $\WW^\op$-modules is a non-stable version of the $\infty$-category
\[
\Cyclo(\Spectra)
~:=~
\left( \Spectra^{\gen^<}_\TT \right)^{\htpy \Nx}
\]
of cyclotomic spectra (where the action $\Nx \lacts \Spectra^{\gen^<}_\TT$ is via geometric fixedpoints) \cite{BluMan}.

\subsection{Motivation: K-theory and traces}
\label{subsection.motivation.Kthy.and.traces}

Let us briefly contextualize this study.  
For an associative algebra $A$ over a commutative ring $\Bbbk$, a \bit{trace} (on $A$) is a $\Bbbk$-linear map 
\[
A \xra{~t~} V
\]
to a $\Bbbk$-vector space $V$ with the property that $t(ab) = t(ba)$ for all $a,b\in A$.
The universal (specifically, initial) trace is the canonical map 
\[
A \xra{~\rm quotient~} \frac{A}{[A,A]}
~=:~ \sHH_0(A)
\]
to the quotient by the commutator.  
This association $A\mapsto \sHH_0(A)$ fails to preserve projective resolutions; consequently, it fails to enjoy a host of desirable properties.
Such failure can be corrected by entertaining, not the quotient ${A}/{[A,A]}$, but a suitably {\it derived} quotient: 
\bit{Hochschild homology} is the universal {\it derived} trace:
\[
A
\xra{~\rm (derived)~quotient~}
\sHH(A)
~.
\]
Again, for computational advantage among other reasons, it is desirable for the input $A$ to Hochschild homology to be extended to {\it derived} associative algebras.
In fact, one can take $A$ to be an associative algebra in any symmetric monoidal $\infty$-category $\cX$ (satisfying some mild conditions).

For trace methods in algebraic K-theory (see~\cite{Madsen-trace}, for instance), it is desirable to extend the input of Hochschild homology yet further: one may wish to entertain the Hochschild homology $\sHH(\cC)$ of an $(\infty,1)$-category $\cC$ that is \textit{enriched} in a symmetric monoidal $\infty$-category $\cX$.
Provided suitable conditions on $\cX$, one may then expect the following features of such an extension.
\begin{enumerate}
\item\label{desideratum.for.HH.extension}

It is indeed an extension: in the case that $\cC = \fB A$ is the one-object $\cX$-enriched $\infty$-category corresponding to an algebra object $A$, we have $\sHH(\fB A) \simeq \sHH(A)$.

\item\label{desideratum.for.HH.cyclo.str}

There is a \bit{cyclotomic structure} on $\sHH(\cC)$.  
In other words, there is a proper-genuine $\TT$-module structure on $\sHH(\cC)$ that is invariant with respect to the isogenetic $\NN^\times$-action on such. 

\item\label{desideratum.for.HH.cyclo.trace}

There is a map $\Obj(\cC) \to \sHH(\cC)$ from the moduli space of objects in an enriched $\infty$-category to its Hochschild homology.
This map is invariant with respect to the cyclotomic structure of the previous point.

\end{enumerate}
Indeed, in the situation of spectral enrichment, such features immediately produce the \bit{cyclotomic trace map}
\[
\sK(\cC)
\xra{~\rm cyclotomic~trace~}
\sHH(\cC)^{\htpy \sf Cyc}~=:~{\sf TC}(\cC)
\]
from the algebraic K-theory of $\cC$ to its cyclotomic invariants (known as the \bit{topological cyclic homology} of $\cC$):
point~\Cref{desideratum.for.HH.extension} is achieved in the subsequent work~\cite{enriched1};
point~\Cref{desideratum.for.HH.cyclo.str} is achieved in its follow-up~\cite{cyclo}; and
point~\Cref{desideratum.for.HH.cyclo.trace} is achieved in its subsequent follow-up~\cite{trace}.  

In the present work, we achieve points~\Cref{desideratum.for.HH.extension}-\Cref{desideratum.for.HH.cyclo.trace} for $\cC$ a category-object \textit{in} an $\infty$-category $\cX$. (Note that $(\infty,1)$-categories define category-objects in the $\infty$-category $\Spaces$ of spaces.) We use the decorated terms ``\textit{non-stable} Hochschild homology'' and ``\textit{non-stable} cyclotomic'' to emphasize working with category-objects in an $\infty$-category $\cX$, leaving the undecorated terms for the general notions. As developed in~\cite{enriched1,cyclo,trace}, working in $\cX = \Cat_{(\infty,1)}$ affords the ``macrocosm'' framework whose ``microcosm'' application amounts to working with enriched $\infty$-categories; in that way, those works are founded on the present work.

\subsection{Parametrizing $(\infty,1)$-categories via quivers}
\label{subsection.intro.quiv}

As explained in \Cref{subsection.approach.to.proof.of.WWop.module.str}, our approach to proving Theorems \ref{mainthm.WWop.module.structure} \and \ref{mainthm.WWinvariance.of.trace.map} rests on an additional main theorem, stated as \Cref{tt1} below: a characterization of $(\infty,1)$-categories given by probing them with quivers.
Specifically, consider the category $\digraphs := \Fun(\bDelta^\op_{\leq 1} , \Fin)$ of finite directed graphs. 
One can contemplate the free $\infty$-category on a finite directed graph, and we define the $\infty$-category of \bit{quivers} to be the fully faithful image of this functor:
\[ \begin{tikzcd}
\digraphs
\arrow{rr}{\Free}
\arrow[dashed]{rd}
&
&
\Cat_{(\infty,1)}
\\
&
\Quiv
\arrow[hook]{ru}[sloped, swap]{\ff}
\end{tikzcd}
~.
\]
In fact, $\Quiv$ is an ordinary category (see \Cref{t26}), and relatedly each quiver is an ordinary category that admits an explicit description in terms of the underlying category generating it (see \Cref{t27}).
Now, given an $(\infty,1)$-category $\cC$, each quiver $\Gamma$ determines a moduli space $\Rep_\cC(\Gamma) := \hom_{\Cat_{(\infty,1)}}(\Gamma,\cC)$ of representations of $\Gamma$ in $\cC$. (See \Cref{d8}.)
This assembles as a functor
\[ \begin{tikzcd}[row sep=0cm]
\Cat_{(\infty,1)}
\arrow{r}{\Rep_{(-)}}
&
\PShv(\Quiv)
\\
\rotatebox{90}{$\in$}
&
\rotatebox{90}{$\in$}
\\
\cC
\arrow[maps to]{r}
&
\left(
\Gamma
\longmapsto
\Rep_\cC(\Gamma)
\right)
\end{tikzcd}~. \]
Moreover, for each $(\infty,1)$-category $\cC$, the functor $\Quiv^{\op} \xra{{\sf Rep}_\cC} \Spaces$ has the following local-to-global property.
\begin{itemize}
\item[]
{\bf Closed sheaf.}
For each finite directed graph $\Gamma$, and for each pair of subgraphs $\Gamma_- , \Gamma_+\subset \Gamma$ for which $\Gamma_- \cup \Gamma_+ = \Gamma$, the canonical square among spaces of representations of the quivers associated to these directed graphs,
\[
\begin{tikzcd}
\Rep_\cC(\Gamma)
\arrow{r}
\arrow{d}
&
\Rep_\cC(\Gamma_+)
\arrow{d}
\\
\Rep_\cC(\Gamma_-)
\arrow{r}
&
\Rep_\cC(\Gamma_- \cap \Gamma_+)
\end{tikzcd}
\]
is a pullback.
\end{itemize}

\begin{maintheorem}[\Cref{t2}]
\label{tt1}
The functor
\[
\Cat_{(\infty,1)}
\xra{\Rep_{(-)}}
\PShv(\Quiv)
\]
is fully faithful.
Moreover, its image consists of those functors $\Quiv^{\op} \ra \Spaces$ that satisfy the \bit{closed sheaf} condition as well as a \bit{univalence} condition.  
\end{maintheorem}

\noindent \Cref{tt1} may be seen as a version of Rezk's presentation of $(\infty,1)$-categories as complete Segal spaces \cite{Rezk1}: the inclusion $\bDelta \hookra \Quiv$ determines a restriction functor $\PShv(\Quiv) \ra \PShv(\bDelta)$, under which the closed sheaf condition corresponds to the Segal condition and the univalence conditions agree. More specifically, we view \Cref{tt1} as giving a ``coordinate-free'' presentation of $(\infty,1)$-categories, analogously to how reduced excisive functors give a ``coordinate-free'' presentation of spectra, as indicated in \Cref{figure-presenting-spectra-and-infty-one-cats}.

\begin{figure}[h]

\begin{tabular}{c||c|c}
&
spectra
&
$(\infty,1)$-categories
\\
\hline
\hline
coordinatized
&
$\{ E_n \in \Spaces_* \}_{n \geq 0}$ with $E_n \simeq \Omega E_{n+1}$
&
complete Segal spaces $\bDelta^\op \ra \Spaces$
\\
\hline
coordinate-free
&
reduced excisive functors $\Spaces^\fin_* \ra \Spaces$
&
univalent closed sheaves $\Quiv^\op \ra \Spaces$
\end{tabular}

\caption{Just as one can present spectra by probing them either by all finite pointed spaces or merely by spheres, one can present $(\infty,1)$-categories by probing them either by all quivers or merely by nonempty linear quivers.\label{figure-presenting-spectra-and-infty-one-cats}}
\end{figure}

\subsection{Universal Hochschild homology}
\label{subsection.approach.to.proof.of.WWop.module.str}

\Cref{tt1} provides a convenient context for defining the Hochschild homology of an $(\infty,1)$-category $\cC$. Namely, there is a standard functor
\begin{equation}
\label{e.first.chi}
\bDelta
\xlongra{\chi}
\Quiv
\end{equation}
that carries each object $[n]$ to the cyclically-directed quiver with $n+1$ vertices (see \Cref{notn.functor.chi.from.Delta.to.Quiv}).\footnote{Specifically, this functor $\bDelta \xra{\chi} \Quiv$ carries $[n] \in \bDelta$ to the cyclically-directed quiver given by adjoining a morphism $n \ra 0$ to the quiver associated to the linearly-directed graph $[n] = \{ 0 \ra \cdots \ra n \}$.}
The \bit{Hochschild homology} of an $(\infty,1)$-category $\cC$ is the colimit
\[
\sHH(\cC)
:=
\colim \left(
\bDelta^\op
\xra{\chi^\op}
\Quiv^\op
\xra{\Rep_\cC}
\Spaces
\right)
\]
(see \Cref{dHH} and \Cref{t250}).

Given this definition of Hochschild homology, we can now describe our approach to proving Theorems \ref{mainthm.WWop.module.structure} \and \ref{mainthm.WWinvariance.of.trace.map}. For this, we generalize from the context of $(\infty,1)$-categories to the slightly more general context of category-objects \textit{in} an $\infty$-category $\cX$ (satisfying mild conditions); note that an $(\infty,1)$-category is an example of a category-object in the $\infty$-category $\Spaces$ of spaces. We prove an analog of \Cref{tt1} in this context as \Cref{t65}. In fact, it is not hard to see that one can restrict to \textit{connected} quivers (see \Cref{d9}), so that a category-object $\cC$ in $\cX$ can be equivalently codified as a functor
\[
\Quiv^\con
\xra{\Rep_\cC}
\cX
\]
satisfying a closed sheaf condition.

Now, to prove Theorems \ref{mainthm.WWop.module.structure} \and \ref{mainthm.WWinvariance.of.trace.map}, we construct a universal $\infty$-category equipped with a category-object and its Hochschild homology: this is the pushout
\[ \begin{tikzcd}
\bDelta^\op
\arrow{r}{\chi^\op}
\arrow[hook]{d}
&
(\Quiv^\con)^\op
\arrow{d}
\\
(\bDelta^\op)^\rcone
\arrow{r}
&
\M^\con
\end{tikzcd}
\qquad
\text{( see \Cref{d7})~.}
\]
We suggestively write $\SS^1 \in \M^\con$ for the image of the cone point. (This notation will be elucidated in \Cref{subsection.connection.to.fact.hlgy}.)
So, given a category-object $\cC$ in $\cX$, we obtain a unique extension
\begin{equation}
\label{extn.of.RepC.from.Quivcon.to.X.to.Mcon}
\begin{tikzcd}[column sep=1.5cm]
(\Quiv^\con)^\op
\arrow{r}{\Rep_\cC}
\arrow[hook]{d}
&
\cX
\\
\M^\con
\arrow[dashed]{ru}[sloped, swap]{\w{\Rep}_\cC}
\end{tikzcd}
\end{equation}
that preserves Hochschild homology objects (namely the left Kan extension), i.e.\! such that the canonical morphism is an equivalence:
\[
\sHH(\cC)
\xlongra{\sim}
\w{\Rep}_\cC(\SS^1)
~.
\]

In general, pushouts among $\infty$-categories can be highly nontrivial, and the pushout defining $\M^\con$ is no exception. Nevertheless, we give a complete description of $\M^\con$ (see \Cref{t51'}). In particular, in view of the canonical extension \Cref{extn.of.RepC.from.Quivcon.to.X.to.Mcon}, \Cref{mainthm.WWop.module.structure} follows from the 
computation
\begin{equation}
\label{Endoms.of.circle.are.WWop}
\End_{\M^\con}(\SS^1)
\simeq
\WW^\op
\end{equation}
of the monoid of endomorphisms of $\SS^1$,\footnote{Nontrivial endomorphisms of $\SS^1$ arise from the fact that the functor $\bDelta^\op \ra (\Quiv^\con)^\op$ is not fully faithful. (By contrast, adjoining the colimit of a fully faithful inclusion results in a formally adjoined object with trivial endomorphism monoid.)} while \Cref{mainthm.WWinvariance.of.trace.map} follows from the computation that the hom-space $\hom_{\M^\con}(\pt, \SS^1)$ is contractible.

\begin{remark}
\label{rmk.factor.chi.through.paracyclic.etc}
The crux of the computation \Cref{Endoms.of.circle.are.WWop}  amounts to factoring the functor $\bDelta \xra{\chi} \Quiv^\con$ through the \bit{parasimplex}, \bit{cyclic}, and \bit{epicyclic} categories as in the commutative diagram
\begin{equation}
\label{e.classicals}
\begin{tikzcd}
\bDelta
\arrow{rd}
\arrow{rrrr}{\chi}
&
&
&
&
\Quiv^\con
\\
&
\copara
\arrow[two heads]{r}
\arrow{d}
&
\bLambda
\arrow[hook, two heads]{r}
\arrow{d}
&
\w{\bLambda}
\arrow{d}
\arrow[hook]{ru}
\\
&
\pt
\arrow[two heads]{r}
&
\BT
\arrow[hook, two heads]{r}
&
\BW
\end{tikzcd}
\end{equation}
and showing that the functor $\bDelta \ra \copara$ is initial, both squares are pullbacks, and the functor $\w{\bLambda} \ra \fB \WW$ is both a cartesian fibration and a localization (see \Cref{t36} and \Cref{t119}).
\end{remark}

\subsection{Connection with factorization homology}
\label{subsection.connection.to.fact.hlgy}

Although we defined the Witt monoid algebraically as the semidirect product $\WW := \TT \rtimes \Nx$, it also arises in differential topology, namely as the (topological) monoid of framed self-covering maps of the circle. Hence, \Cref{mainthm.WWop.module.structure} can be seen as articulating a contravariant functoriality of Hochschild homology for framed self-covers of the circle.

Of course, the connection between Hochschild homology and the circle is not new. Clasically, this connection is manifested e.g.\! by Connes' cyclic operator on Hochschild homology (see e.g. \cite{Loday}), which records the action of the (maximal) subgroup $\TT \subset \WW$. More recently, Hochschild homology has been recognized as \bit{factorization homology} over $S^1$: for $\EE_1$-algebras in \cite[Sect. 5.5.3]{LurieHA} (under the name ``topological chiral homology''), and for $(\infty,1)$-categories in \cite{AFR2}.

In more detail, the work \cite{AFR2} constructs for any dimension $n$ an $\infty$-category $\cMfd^{\sf sfr}_n$ of \textit{compact solidly $n$-framed stratified spaces}, whose morphisms accommodate a wide variety of notions in differential topology -- in particular, framed self-covers of compact framed $n$-manifolds. 
This $\infty$-category is specifically tailored towards the construction of factorization homology, which gives a way of ``integrating'' an $(\infty,n)$-category $\cC$ over compact framed $n$-manifolds (and more generally over compact solidly $n$-framed stratified spaces):
\[ \begin{tikzcd}[row sep=0cm]
\cMfd^{\sf sfr}_n
\arrow{r}
&
\Spaces
\\
\rotatebox{90}{$\in$}
&
\rotatebox{90}{$\in$}
\\
M
\arrow[maps to]{r}
&
{\displaystyle \int_M \cC}
\end{tikzcd}
~.\footnote{Actually, in general one must require that $\cC$ has adjoints: that is, for every $0 < k < n$, every $k$-morphism in $\cC$ admits both adjoints. 
Of course, this condition is vacuous when $n=1$.}
\]
The paper~\cite{AFR2} establishes this construction of factorization homology of $(\infty,n)$-categories for $n\leq 2$; for $n>2$, this construction has yet to be established.

To state our last main theorem, we introduce the notation $\M$ for the $\infty$-category obtained from $\M^\con$ by freely adjoining finite products.\footnote{In other words, the functor $\Quiv^\op \ra \M$ admits an analogous universal property in $\Cat^\times$ (the $\infty$-category of $\infty$-categories with finite products) to that of the functor $(\Quiv^\con)^\op \ra \M^\con$ in $\Cat$.}

\begin{maintheorem}
\label{mainthm.connection.to.fact.hlgy}
There is a canonical equivalence
\[
\M
\xlongra{\sim}
\cMfd^{\sf sfr}_1
~,
\]
which carries $\SS^1$ to $S^1$ (the framed circle). Under this equivalence, for any $(\infty,1)$-category $\cC$, the functor $\w{\Rep}_\cC$ is naturally equivalent to factorization homology: there is a canonical commutative triangle
\[ \begin{tikzcd}[column sep=2cm, row sep=0.5cm]
\M
\arrow{rd}[]{\w{\Rep}_\cC}
\arrow{dd}[sloped, anchor=north]{\sim}
\\
&
\Spaces
\\
\cMfd^{\sf sfr}_1
\arrow{ru}[swap]{\int \cC}
\end{tikzcd}
~.
\]
\end{maintheorem}

\begin{remark}
\label{rmk.circle.moduli}
From the point of view of manifolds afforded by \Cref{mainthm.connection.to.fact.hlgy}, the diagram \Cref{e.classicals} has a natural interpretation: 
the category $\w{\bLambda}$ can be regarded as a moduli category of (non-trivially) stratified framed circles, in which the non-invertible morphisms classify isogenies.
Namely, there is a fully faithful functor $\copara^{\rcone} \hookrightarrow \cMfd^{\sf sfr}_{1/S^1}$ whose image consists of refinement morphisms to $S^1$, and the (left lax) action by $\WW^{\op}$ arises through $\WW^\op \simeq \End_{\cMfd^{\sf sfr}_1}(S^1)$.
In particular, there is a fully faithful functor $\w{\bLambda} \hookrightarrow \cMfd^{\sf sfr}_1$ whose image consists of those solidly 1-framed stratified spaces with more than one stratum and whose underlying topological space is a circle.

\end{remark}

\begin{remark}
\label{rmk.curiosity}
\Cref{rmk.circle.moduli} gives a geometric explanation for why, among all diagrams in $(\Quiv^\con)^{\op}$, it is natural to formally adjoin the colimit of 
\[
\chi
\colon
\bDelta^{\op} \xra{\Cref{e.first.chi}} 
(\Quiv^\con)^{\op}
~.
\]
In that spirit, it would be interesting to give a characterization of the parsimplex, cyclic, or epicyclic category that is entirely intrinsic to the full subcategory $\Quiv \subset \Cat_{(\infty,1)}$ (ie, blind to the geometric context of \Cref{rmk.circle.moduli}).  

\end{remark}

In view of \Cref{rmk.curiosity}, \Cref{mainthm.connection.to.fact.hlgy} clarifies the manifold-theoretic origins of the parasimplex, cyclic, and epicyclic categories (which were originally defined combinatorially). 
On the other hand, \Cref{mainthm.connection.to.fact.hlgy} is 
a completely algebraic characterization of the $\infty$-category $\cMfd^{\sf sfr}_1$ despite its differential topology origins, in the spirit of the celebrated cobordism hypothesis \cite{Lurie-cobordism}. 
More broadly, we are inspired by the idea that some suitable $\infty$-category of ``$n$-quivers'' (ie, a suitable version of $n$-computads \cite{Street-2computads,Batanin})\footnote{Note that quivers are precisely 1-computads.}
analogously extends to give $\cMfd^{\sf sfr}_n$.

\subsection{Conventions}

We set the following conventions throughout this work.
\begin{itemize}

\item \catconventions

\item \spacescatsconventions

\item
Let $\cY$ be an $\infty$-category.
For $T$ a continuous monoid,
the $\infty$-category of (\bit{left} \bit{$T$-modules} (in $\cY$) is that of functors:
\[
\Mod_{T}(\cY)
~:=~ 
\Fun( \fB T , \cY)
~.
\]
Restricting along the unique functor $\fB T \xra{!} \ast$ defines a functor
$
\cY
=
\Mod_\ast(\cY)
\to
\Mod_T(\cY)
.
$
When they exist, the left- and right-adjoints to this functor are respectively the \bit{$T$-coinvariants} and the \bit{$T$-invariants}, and are respectively denoted
\[
(-)_{{\sf h}T}
\colon
\Mod_T(\cY)
\longrightarrow
\cY
\qquad
\text{ and }
\qquad
(-)^{{\sf h}T}
\colon
\Mod_T(\cY)
\longrightarrow
\cY
~.
\]

\end{itemize}

\subsection*{Acknowledgements}

DA was supported by the National Science Foundation under awards 1812055 and 1945639. AMG was supported by the National Science Foundation under award 2105031. 
This material is based upon work supported by the National Science Foundation under award 1440140 while the authors were in residence at the Mathematical Sciences Research Institute in Berkeley, California, during the Spring 2020 semester, and under award 1928930 while the authors participated in a program supported by the Mathematical Sciences Research Institute during the Summer of 2022 in partnership with the Universidad Nacional Aut\'{o}noma de M\'{e}xico.

\section{Parametrizing higher categories by quivers}

In this section, we characterize $(\infty,1)$-categories as copresheaves on a category of quivers.
Specifically, as \Cref{d12} we construct a category $\Quiv$ in which an object is a finite directed graph;
as \Cref{t2}, we construct a fully faithful functor between $\infty$-categories,
\[
\Cat_{(\infty,1)}
\longrightarrow
\PShv(\Quiv)
~,
\]
and characterize its image.

Informally, this is to say an $(\infty,1)$-category is characterized by quiver representations into it.  
The work of Rezk~(\cite{Rezk1}) implements such a characterization in terms of linear quivers.  
So, in this sense, the characterization presented here is more versatile.  
For instance, it accommodates representations of cyclically-directed quivers.

\subsection{Recollections of $\Cat_{(\infty,1)}$}\label{sec.cat.def}

\begin{definition}
\label{d14}
The \bit{simplex category} is the category $\bDelta$ in which an object is a finite nonempty linearly ordered set, and a morphism between two is a (weakly) order-preserving map, with composition given by composing maps.
Such a morphism is \bit{idle} if it is convex\footnote{A morphism $I \xra{f} J$ between linearly ordered sets is \bit{convex} if, for each $i_-, i_+\in I$ and each $j\in J$ for which $f(i_-) \leq  j \leq f(i_+)$, the preimage $f^{-1}(j)$ is nonempty.}; such a morphism is \bit{closed} if it is a convex inclusion; such a morphism is \emph{creation} if it is a convex surjection; such a morphism is \bit{active} if it preserves minima and maxima.
The subcategories 
\[
\bDelta^{\sf cls}
~,~
\bDelta^{\sf cr}
~,~
\bDelta^{\sf idl}
~,~
\bDelta^{\sf act}
~\subset~
\bDelta
\]
consist of all objects and those morphisms that are respectively closed, creation, idle, and active.

\end{definition}

\begin{remark}
An object in $\bDelta$ can be regarded as a linear quiver.  

\end{remark}

\begin{remark}
The morphisms in $\bDelta$ that we call ``closed'' are elsewhere called ``inert'' (e.g.\! in \cite{LurieHA}).  We choose our terminology to remain consistent with that surrounding the $\infty$-category $\cMfd_1^{\sfr}$, discussed in~\S\ref{sec.facts}.
\end{remark}

\begin{observation}
\label{t41}
The pair of subcategories $(\bDelta^{\sf act} , \bDelta^{\sf cls})$ is a factorization system on 
$\bDelta$.  
In other words, each morphism in $\bDelta$ uniquely factors as a composition of an active morphism followed by a closed morphism.  

\end{observation}

\begin{definition}
\label{dd1}
A \bit{basic closed cover (in $\bDelta$)} is a diagram in $\bDelta^{\sf cls}$ of the form
\[
\xymatrix{
\{1\}
\ar[rr]
\ar[d]
&&
\{1<\dots<p\}
\ar[d]
\\
\{0<1\}
\ar[rr]
&&
[p]
}
\]
for some $p>1$.
A presheaf $\bDelta^{\op} \xra{\cF} \Spaces$ is a \bit{closed sheaf (on $\bDelta$)} if it carries (the opposites of) each basic closed cover in $\bDelta$ to a limit diagram in $\Spaces$.
A \bit{closed cover (in $\bDelta$)} is a diagram $\cK^{\rcone} \to \bDelta^{\sf cls}$ for which, for each closed sheaf, $\bDelta^{\op} \xra{\cF} \Spaces$, the composite functor 
\[
(\cK^{\op})^{\lcone} = (\cK^{\rcone})^{\op} 
\to 
(\bDelta^{\sf cls})^{\op} 
\hookrightarrow 
\bDelta^{\op} 
\xra{\cF} 
\Spaces
\]
is a limit diagram.
\end{definition}

Recall the definition of the $\infty$-category $\Cat_{(\infty,1)}$ of \bit{$(\infty,1)$-categories} (see, for instance,~\cite{Rezk1}).
Recall from~\cite{flagged} the definition of the $\infty$-category $\fCat_{(\infty,1)}$ of \bit{flagged $(\infty,1)$-categories}, an object in which is a functor $\cC_0 \to \cC_1$ from an $(\infty,0)$-category (ie, an $\infty$-groupoid) to an $(\infty,1)$-category with the property that the functor is surjective on objects (ie, the resulting map between sets of isomorphism classes of objects, $\pi_0 \Obj(\cC_0) \to \pi_0 \Obj(\cC_1)$, is surjective).  
Consider the restricted Yoneda functor along the standard fully faithful functor
$\bDelta \hookrightarrow \Cat_{(\infty,1)} \subset \fCat_{(\infty,1)}$:
\begin{equation}\label{e3}
\Cat_{(\infty,1)}
~\subset~
\fCat_{(\infty,1)}
\longrightarrow
\PShv(\bDelta)
~,\qquad
\cC\mapsto 
\Bigl(
[p]^\circ
\mapsto 
\Hom_{\fCat_{(\infty,1)}}( [p] , \cC )
\Bigr)
~.
\end{equation}
The work of Rezk~(\cite{Rezk1}), followed by the work~\cite{flagged}, 
implies each of the functors~\Cref{e3} is fully faithful, 
and that the image of $\fCat_{(\infty,1)}$ consists of those presheaves $\bDelta^{\op}\xra{\cF} \Spaces$ that are closed sheaves (on $\bDelta$), 
while the image of $\Cat_{(\infty,1)}$ consists of those presheaves $\bDelta^{\op}\xra{\cF} \Spaces$ that further satisfy a univalent-completeness condition.  
Each of the fully faithful restricted Yoneda functors~\Cref{e3} is a right adjoint in a Bousfield localization between presentable $\infty$-categories:
\begin{equation}\label{e4}
\PShv(\bDelta)
~\rightleftarrows~
\fCat_{(\infty,1)}
~\rightleftarrows~
\Cat_{(\infty,1)}
~.
\end{equation}

\begin{remark}
One can take the definitions of the $\infty$-categories $\Cat_{(\infty,1)}$ and $\fCat_{(\infty,1)}$ as the images of these fully faithful functors~\Cref{e3}.
\end{remark}

\subsection{Quivers}

The main outcome of this subsection is \Cref{t27}, which gives an explicit identification of the free $(\infty,1)$-category on a quiver (also known as a finite directed graph).

In the next definition, the full subcategory $\bDelta_{\leq 1} \subset \bDelta$ consists of those finite nonempty linearly ordered sets with cardinality at most $2$.  
\begin{definition}\label{d3}
The category of \bit{finite directed graphs} is
\[
\digraphs
~:=~
\Fun(\bDelta_{\leq 1}^{\op} , \Fin)
~.
\]

\end{definition}

\begin{terminology}
\label{d26}
Let $\Gamma$ be a finite directed graph.
Its \bit{set of vertices} is $\Gamma^{(0)}:=\Gamma([0])$; its \bit{set of edges} is $\Gamma^{(1)} := \Gamma([1]) \setminus \Gamma([0])$; there is the evident span of sets 
\[
\Gamma^{(0)} \xlongla{s} \Gamma^{(1)} \xlongra{t} \Gamma^{(0)}
\]
arising from the cospan $\{0\} \to [1] \leftarrow \{1\}$ in $\bDelta$; for $v\in \Gamma^{(0)}$ a vertex, its \bit{set of exiting edges} is the preimage ${\sf Out}_\Gamma(v):=t^{-1}(v)\cap \Gamma^{(1)}$, its \bit{set of entering edges} is the preimage ${\sf In}_\Gamma(v):= s^{-1}(v)\cap \Gamma^{(1)}$, and its \bit{directed-valence} is the ordered pair of cardinalities $\Bigl( {\sf Card}\bigl( {\sf Out}_\Gamma(v) \bigr) , {\sf Card}\bigl( {\sf In}_\Gamma(v) \bigr)  \Bigr)$. 
Say $\Gamma$ is \bit{connected} if the quotient set $(\Gamma^{(0)}){/\sim}$ is a singleton, where the equivalence relation is generated by declaring $v\sim w$ if $(s,t)^{-1}(v,w) \neq \emptyset$.  
A finite directed graph is \bit{cyclically-directed} if it is connected and each vertex has directed-valence $(1,1)$ (i.e.\! each vertex has exactly one exiting edge and exactly one entering edge).
A finite directed graph is \bit{linearly-directed} if it is connected and each vertex has at most one entering edge and at most one exiting edge, and there is some vertex with no exiting edges (or, equivalently, no entering edges).

A morphism $\Gamma \xra{f} \Xi$ between finite directed graphs is \bit{non-degenerate} if it carries edges to edges, which is to say the diagram among sets
\[
\begin{tikzcd}
\Gamma([0])
\arrow{r}{f_{[0]}}
\arrow{d}[swap]{\Gamma(!)}
&
\Xi([0])
\arrow{d}{\Xi(!)}
\\
\Gamma([1])
\arrow{r}[swap]{f_{[1]}}
&
\Xi([1])
\end{tikzcd}
\]
is a pullback.

\end{terminology}

An object in $\digraphs$ can be regarded as a quiver, which can in turn be regarded as a category.  
We make this precise in what follows.

Consider the composite adjunction:
\begin{equation}\label{e11}
\PShv(\bDelta_{\leq 1})
\overset{\sf LKE}{ \underset{\rm restriction}{~\rightleftarrows~}}
\PShv(\bDelta)
\underset{\Cref{e4}}{~\rightleftarrows~}
\Cat_{(\infty,1)}
~.
\end{equation}
This results in a functor
\begin{equation}
\label{e1}
\Free
\colon
\digraphs
~:=~
\Fun(\bDelta_{\leq 1}^{\op} , \Fin)
~\subset~
\Fun(\bDelta_{\leq 1}^{\op} , \Spaces)
~=~
\PShv(\bDelta_{\leq 1})
\xra{~\Cref{e11}~}
\Cat_{(\infty,1)}
~,
\end{equation}
whose value on a finite directed graph $\Gamma$ is the \bit{free $(\infty,1)$-category} on $\Gamma$, hence the notation.

\begin{remark}
For $\Gamma$ a finite directed graph, the free category generated by $\Gamma$ may be familiar: an object is a vertex in $\Gamma$, a morphism is a finite sequence of directed edges, each consecutive pair of which that match up head-to-tail.
Unfortunately, the definition of ${\sf Free}(\Gamma)$ just above is a priori an $(\infty,1)$-category, and as so, it may not admit such a description.  (Necessarily, though, such describes the homotopy category of ${\sf Free}(\Gamma)$.)  
However, \Cref{t27} below ensures ${\sf Free}(\Gamma)$ is, in fact, an ordinary category, and it admits such a description.  
\Cref{t94}, further, supplies an explicit commensurable description of morphisms between values of ${\sf Free}$.

\end{remark}

\begin{definition}
\label{d12}
A \bit{quiver} is an $(\infty,1)$-category that is the free $(\infty,1)$-category on a finite directed graph.
The $\infty$-category $\Quiv$ is the full $(\infty,1)$-subcategory
\[
\Quiv
~\subset~
\Cat_{(\infty,1)}
\]
consisting of those $(\infty,1)$-categories that are quivers.\footnote{Note that some authors define a morphism of quivers to be a morphism of finite directed graphs. The notion that we study is strictly more general: in view of \Cref{t94}, \Cref{t26}, and \Cref{t34}, the $\infty$-category $\Quiv$ is actually an ordinary category and contains the former as a (1-full) subcategory containing the same (groupoid of) objects.}
\end{definition}

\begin{observation}
\label{t20}
By \Cref{d12}, 
there is a factorization, which is unique, in a diagram among $\infty$-categories:
\[
\begin{tikzcd}
&
\Quiv
\arrow[hook]{rd}[sloped]{\ff}
\\
\digraphs
\arrow[dashed]{ru}[sloped]{\Free}
\arrow{rr}[swap]{\Free}
&
&
\Cat_{(\infty,1)}
\end{tikzcd}
~.
\]
Furthermore, this factorizing functor,
\[
\Free \colon \digraphs
\longrightarrow
\Quiv
~,
\]
is surjective on objects.

\end{observation}

\begin{observation}
\label{t92}
The diagram among $\infty$-categories
\[
\begin{tikzcd}
\Spaces
&
\PShv(\bDelta_{\leq 1})
\arrow{l}[swap]{|-|}
\arrow{d}{\Cref{e11}}
\\
&
\Cat_{(\infty,1)}
\arrow{lu}[sloped, swap]{|-|}
\end{tikzcd}
\]
involving geometric realizations / $\infty$-groupoid-completions
canonically commutes, because the diagram among their right adjoints canonically commutes.
In particular, for each object $\Free(\Gamma) \in \Quiv$, the $\infty$-groupoid-completion 
\[
\bigl| \Free(\Gamma) \bigr|
~\simeq~
|\Gamma |
\] 
is the geometric realization of its generating finite directed graph, which canonically admits the structure of a 1-dimensional finite CW complex.  

\end{observation}

\begin{observation}
\label{t25}
By definition of $\Quiv$, both of the functors 
\[
\Free \colon \digraphs
\longrightarrow
\Quiv
\qquad
\text{ and }
\qquad
\Free \colon \digraphs
\longrightarrow 
\Cat_{(\infty,1)}
\]
preserve cobase-change along monomorphisms.
Indeed, the case of the former follows from the latter.
Next, the fully faithful functor
$\Fun(\bDelta_{\leq 1}^{\op} , \Fin) \hookrightarrow \Fun(\bDelta_{\leq 1}^{\op} , \Spaces)$ preserves cobase-change along monomorphisms.  
The assertion follows from the definition of $\Free$ as a composition~\Cref{e1}, using that both of the rightward functors in~\Cref{e11} are left adjoints and therefore preserve pushouts.
In particular, $\Quiv$ admits finite coproducts, which are given by disjoint unions of finite directed graphs.

\end{observation}

\begin{lemma}
\label{t6}
The simplex category 
is
the full $\infty$-subcategory
\[
\bDelta
~\subset~
\Cat_{(\infty,1)}
\]
consisting of those categories that are the values of $\Free$ on finite nonempty linearly-directed graphs (in the sense of \Cref{d26}).
In particular, there is a canonical fully faithful functor (between $\infty$-subcategories of $\Cat_{(\infty,1)}$),
\[
\bDelta
~\hookrightarrow~
\Quiv
~,
\]
the image of which consists of the finite nonempty linearly-directed graphs.  

\end{lemma}

\begin{notation}
\label{notation.brackets.functor.from.delta.op}
The fully faithful functor
\[
\rho
\colon
\bDelta
\xra{~\rm Lem~\ref{t6}~}
\Quiv
~,
\]
is defined so that the composite fully faithful functor $\bDelta \xra{\rho} \Quiv \hookrightarrow \fCat_{(\infty,1)}$ is the standard fully faithful functor.
The value of $\rho$ on an object $[p] \in \bDelta$ is denoted $\rho(p)\in \Quiv$ or simply $[p]\in \Quiv$.

\end{notation}

\begin{proof}[Proof of \Cref{t6}]
Let $p \geq 0$.
Consider the linearly ordered set $[p] \in \bDelta$, regarded as a category in the standard manner.  
Consider the linearly directed graph $A_p := A_{\{0,\dots,p\}} := (0 \to 1 \to \cdots \to p)$.
Consider the map between finite directed graphs
\[
A_p
\longrightarrow
[p]
\]
to the underlying directed graph of $[p]$ -- this is the unique map that is the identity map on sets of vertices.  
From the definition of ${\sf Free}$ in terms of a left adjoint, map between finite directed graphs determines a functor between $(\infty,1)$-categories
\begin{equation}
\label{k1}
{\sf Free}(A_p)
\longrightarrow
[p]
~.
\end{equation}
We next prove this functor~(\ref{k1}) is an equivalence, which implies the lemma.

We proceed by induction on $p\geq 0$.
Suppose $p\leq 1$.
Note that $A_p = [p]\in \bDelta_{\leq 1}$, in these cases.
Using that the left Kan extension $\PShv(\bDelta_{\leq 1}) \xra{\sf LKE} \PShv(\bDelta)$ restricts to representables as representables, then ${\sf LKE}([p]) \in \PShv(\bDelta)$ is the representable simplicial space on $[p]$.  
Generally, representable simplicial spaces satisfy the Segal and univalence-completeness conditions.  
We conclude that~(\ref{k1}) is an equivalence for $p\leq 1$.

Next, assume that $p>1$.
Observe the pushout diagram among directed graphs
\[
\xymatrix{
A_{\{1\}}
\ar[rr]
\ar[d]
&&
A_{\{1 , \dots , p\}}
\ar[d]
\\
A_{\{0 , 1\}}
\ar[rr]
&&
A_p
.
}
\]
The first statement of Observation~\ref{t25} implies the resulting diagram among $(\infty,1)$-categories
\[
\xymatrix{
{\sf Free}(A_{\{1\}})
\ar[rr]
\ar[d]
&&
{\sf Free}(A_{\{1 , \dots , p\}})
\ar[d]
\\
{\sf Free}(A_{\{0 , 1\}})
\ar[rr]
&&
{\sf Free}(A_p)
.
}
\]
is also a pushout.
Meanwhile, the Segal condition is just so that the diagram among $(\infty,1)$-categories
\[
\xymatrix{
\{1\}
\ar[rr]
\ar[d]
&&
\{1<\dots<p\}
\ar[d]
\\
\{0<1\}
\ar[rr]
&&
[p]
}
\]
is a pushout.  
The result follows by induction.

\end{proof}

The next technical result gives an explicit description of the values of $\Free$.
It is phrased in terms of the following notation.
For $[p]\in \bDelta$, and $\Gamma \in \digraphs$, 
an object in the under-over-category,
\[
\left(
\bDelta^{\sf idl}_{/\Gamma}
\right)^{[p]/}
~:=~
\bDelta^{[p]/}
\underset{\bDelta}\times
\bDelta^{\sf idl}
\underset{\digraphs}\times
\digraphs_{/\Gamma}
~,
\]
is a pair of morphisms: $[p]\xra{\sigma} [q]$ in $\bDelta$ and $[q]\xra{f}\Gamma$ in $\digraphs$;
the full subcategory
\[
\left(
\bDelta^{\sf idl}_{/^{\sf non.deg}\Gamma}
\right)^{[p]/^{\sf act}}
~\subset~
\left(
\bDelta^{\sf idl}_{/\Gamma}
\right)^{[p]/}
\]
consists of those objects $(\sigma,\Gamma)$ in which $\sigma$ is active (in the sense of \Cref{d14}) and $f$ is non-degenerate (in the sense of \Cref{d26}).  

\begin{lemma}
\label{t28}
Let $\Gamma \in \digraphs$ be a finite directed graph.  
For each $[p]\in \bDelta$, there is a canonical equivalence between spaces,
\[
\Hom_{\Cat_{(\infty,1)}}\left([p], \Free\left(\Gamma \right) \right)
~\simeq~
\Obj\left(
\left(
\bDelta^{\sf idl}_{/^{\sf non.deg}\Gamma}
\right)^{[p]/^{\sf act}}
\right)
~,
\]
involving the maximal $\infty$-subcategory of the (active) undercategory (in $\bDelta$) of the (non-degenerate) overcategory (in $\digraphs$).
Furthermore, through this composite identification, the unit morphism between directed graphs 
\[
{\sf unit}_\Gamma
\colon
\Gamma
\longrightarrow
\Free(\Gamma)
\]
evaluates on $[p]\in \bDelta_{\leq 1}$ as the monomorphism between spaces
\[
\Gamma([p])
~\simeq~
\Obj\left(
\left(
\bDelta^{\sf idl}_{/^{\sf non.deg}\Gamma}
\right)^{[p]/^{\sf cr}}
\right)
~\hookrightarrow~
\Obj\left(
\left(
\bDelta^{\sf idl}_{/^{\sf non.deg}\Gamma}
\right)^{[p]/^{\sf act}}
\right)
~\simeq~
\Hom_{\Cat_{(\infty,1)}}\left([p], \Free\left(\Gamma \right) \right)
~.
\]

\end{lemma}

\begin{proof}
We establish a sequence of equivalences between spaces:
\begin{align}
\hspace{-1cm}
\label{e30}
\Hom_{\Cat_{(\infty,1)}}\left([p], \Free\left(\Gamma \right) \right)
\xla{~\simeq~}
&
\colim
\left(
\left(
\left(
\bDelta^{\sf idl}
\right)^{[p]/}
\right)^{\op}
\xra{\rm forget}
\left(
\bDelta^{\sf idl}
\right)^{\op}
\xra{\Cref{e2}}
\left(
\digraphs
\right)^{\op}
\xra{\Hom_{\digraphs}(-,\Gamma)}
\Spaces
\right)
\\
\label{e29}
~\simeq~
&
\left|
\left(
\bDelta^{\sf idl}_{/\Gamma}
\right)^{[p]/}
\right|
\\
\label{e28}
\xla{~\simeq~}
&
\left|
\left(
\bDelta^{\sf idl}_{/^{\sf non.deg}\Gamma}
\right)^{[p]/^{\sf act}}
\right|
\\
\label{e10}
\xla{~\simeq~}
&
\Obj\left(
\left(
\bDelta^{\sf idl}_{/^{\sf non.deg}\Gamma}
\right)^{[p]/^{\sf act}}
\right)
~,
\end{align}
We first establish the identifications~(\ref{e29}),~(\ref{e28}), and~(\ref{e10}).
So let $[p] \in \bDelta$.
First, recall that, for $\cB \xra{F} \Spaces$ a functor from an $\infty$-category,
there is a canonical identification of its colimit as the $\infty$-groupoid-completion of its unstraightening:
\[
\Bigl|
{\sf Un}(F)
\Bigr|
~\simeq~
\colim(F)
~.
\]
Next, note that, for $\cK \xra{f} \cB$ a functor, the unstraightening of the composite $F\circ f$ is identical with the base-change along $f$ of the unstraightening of $F$:
\[
{\sf Un}(F \circ f)
~\simeq~
{\sf Un}(F)_{|\cK}
~.
\]
Next, note that, for $F = \Hom_{\cB}(-,b)$ representable, then the unstraightening of $F$ is the $\infty$-overcategory:
\[
\cB_{/b}
~\simeq~
{\sf Un}(F)
~.
\]
Putting these observations together reveals a canonical identification:
\begin{eqnarray*}
\label{e18}
\colim
\left(
\left(
\left(
\bDelta^{\sf idl}
\right)^{[p]/}
\right)^{\op}
\xra{\rm forget}
\left(
\bDelta^{\sf idl}
\right)^{\op}
\xra{(\ref{e2})}
\left(
\digraphs
\right)^{\op}
\xra{\Hom_{\digraphs}(-,\Gamma)}
\Spaces
\right)
&
~\simeq~
&
\left|
\left( \left( \bDelta^{\sf idl} \right)^{[p]/} \right)_{/\Gamma}
\right|
\\
&
~\simeq~
&
\left|
\left(
\bDelta^{\sf idl}_{/\Gamma}
\right)^{[p]/}
\right|
~,
\end{eqnarray*}
in which the second identification is direct from the definition of these over-under-categories.  
This establishes the identification~(\ref{e29}).

Next, the active-closed factorization system on $\bDelta$ determines a left adjoint localization
\begin{equation}
\label{e25}
\left(
\bDelta^{\sf idl}_{/\Gamma}
\right)^{[p]/}
\longrightarrow
\left(
\bDelta^{\sf idl}_{/\Gamma}
\right)^{[p]/^{\sf act}}
~.
\end{equation}
The surjective-injective factorization system on ${\sf Sets}$ determines a further right adjoint localization 
\begin{equation}
\label{e24}
\left(
\bDelta^{\sf idl}_{/\Gamma}
\right)^{[p]/^{\sf act}}
\longrightarrow
\left(
\bDelta^{\sf idl}_{/^{\sf non.deg}\Gamma}
\right)^{[p]/^{\sf act}}
~.
\end{equation}
The identification~(\ref{e28}) then follow because adjoint functors implement equivalences between $\infty$-groupoid-completions.

Next, by the injective-surjective factorization system on sets, 
the definition of \bit{non-degenerate} is such that the projection
\[
\bDelta^{\sf idl}_{/^{\sf non.deg}\Gamma}
\longrightarrow
\bDelta^{\sf idl}
\]
factors through the subcategory $\bDelta^{\sf cls} \subset \bDelta^{\sf idl}$ of closed morphisms.
Meanwhile, the definition of active is such that the projection
\[
(\bDelta^{\sf idl})^{[p]/^{\sf act}}
\longrightarrow
\bDelta^{\sf idl}
\]
factors through the subcategory $\bDelta^{\sf cr}\subset \bDelta$ of creation morphisms.
Therefore the canonical projection
\[
\left(
\bDelta^{\sf idl}_{/^{\sf non.deg}\Gamma}
\right)^{[p]/^{\sf act}}
\longrightarrow
\bDelta^{\sf idl}
\]
factors through $\Obj( \bDelta )
\xra{\simeq}
\bDelta^{\sf cls}\cap \bDelta^{\sf cr}
\subset
\bDelta^{\sf idl}$.
Now, both of these projections are full subcategories of the respective right and left fibrations, each with 0-type fibers:
\[
\bDelta^{\sf idl}_{/\Gamma}
\longrightarrow
\bDelta^{\sf idl}
\qquad
\text{ and }
\qquad
(\bDelta^{\sf idl})^{[p]/}
\longrightarrow
\bDelta^{\sf idl}
~.
\]
Because $\Obj( \bDelta )$ is a 0-type, we conclude that the $\infty$-category $\left(
\bDelta^{\sf idl}_{/^{\sf non.deg}\Gamma}
\right)^{[p]/^{\sf act}}$ is, in fact, a 0-type.
In particular, both of the functors
\begin{equation}
\label{e31}
\Obj\left(
\left(
\bDelta^{\sf idl}_{/^{\sf non.deg}\Gamma}
\right)^{[p]/^{\sf act}}
\right)
\xra{~\simeq~}
\left(
\bDelta^{\sf idl}_{/^{\sf non.deg}\Gamma}
\right)^{[p]/^{\sf act}}
\xra{~\simeq~}
\left|
\left(
\bDelta^{\sf idl}_{/^{\sf non.deg}\Gamma}
\right)^{[p]/^{\sf act}}
\right|
\end{equation}
are equivalences.
This establishes the identification~(\ref{e10}).

Now, 
consider the simplicial space
\begin{equation}
\label{e22}
\left|
\left(
\bDelta^{\sf idl}_{/\Gamma}
\right)^{[\bullet]/}
\right|
\colon
\bDelta^{\op}
\longrightarrow
\Spaces
~.
\end{equation}
Through the identification~(\ref{e29}), proved above, this simplicial space~(\ref{e22}) witnesses a left Kan extension:
\[ \begin{tikzcd}
\left(
\bDelta^{\sf idl}
\right)^{\op}
\arrow{r}[xshift=1.0cm, yshift=-1.0cm]{\rotatebox{225}{$\Longrightarrow$}}{(\ref{e2})}
\arrow{d}
&[0.5cm]
\left( \digraphs \right)^{\op}
\arrow{r}{\Hom_{\digraphs}(-,\Gamma)}
&[2cm]
\Spaces
\\
\bDelta^{\op}
\arrow[bend right=20]{rru}[swap, sloped]{
\left|
\left(
\bDelta^{\sf idl}_{/\Gamma}
\right)^{[\bullet]/}
\right|
}
\end{tikzcd}
~.
\]
In particular, there is a canonical morphism between simplicial spaces
\begin{equation}
\label{e23}
\left|
\left(
\bDelta^{\sf idl}_{/\Gamma}
\right)^{[\bullet]/}
\right|
\longrightarrow
\Hom_{\Cat_{(\infty,1)}}\left([\bullet],\Free\left(\Gamma\right)\right)
~,
\end{equation}
which extends the unit morphisms between directed graphs.
So the equivalence~(\ref{e30}) is a consequence of this morphism~(\ref{e23}) being an equivalence.
By definition of the functor $\Free$ in terms of left adjoints, this morphism~(\ref{e23}) is initial among all morphisms from~(\ref{e22}) to a simplicial space that satisfies the Segal and univalent-completeness conditions.  
So, the result follows upon showing the simplicial space~(\ref{e22}) satisfies the Segal and univalent-completeness conditions.

We first show~(\ref{e22}) satisfies the Segal condition.
So let $[p]\in \bDelta$.
We must show that the commutative square
\begin{equation}
\label{e27}
\begin{tikzcd}
\left(
\bDelta^{\sf idl}_{/\Gamma}
\right)^{[p]/}
\arrow{r}
\arrow{d}
&
\left(
\bDelta^{\sf idl}_{/\Gamma}
\right)^{\{1<\cdots<p\}/}
\arrow{d}
\\
\left(
\bDelta^{\sf idl}_{/\Gamma}
\right)^{\{0<1\}/}
\arrow{r}
&
\left(
\bDelta^{\sf idl}_{/\Gamma}
\right)^{\{1\}/}
\end{tikzcd}
\end{equation}
among categories induces a pullback diagram among $\infty$-groupoid-completions.
Note that, through the adjoint localizations~(\ref{e25}) and~(\ref{e24}), and the equivalence~(\ref{e31}), we obtain a commutative diagram among spaces (in fact, 0-types),
\[
\begin{tikzcd}
\Obj\left(
\left(
\bDelta^{\sf idl}_{/^{\sf non.deg}\Gamma}
\right)^{[p]/^{\sf act}}
\right)
\arrow{r}
\arrow{d}
&
\Obj\left(
\left(
\bDelta^{\sf idl}_{/^{\sf non.deg}\Gamma}
\right)^{\{1<\cdots<p\}/^{\sf act}}
\right)
\ar[d]
\\
\Obj\left(
\left(
\bDelta^{\sf idl}_{/^{\sf non.deg}\Gamma}
\right)^{\{0<1\}/^{\sf act}}
\right)
\arrow{r}
&
\Obj\left(
\left(
\bDelta^{\sf idl}_{/^{\sf non.deg}\Gamma}
\right)^{\{1\}/^{\sf act}}
\right)
\end{tikzcd}
~,
\]
that is identical with the diagram among $\infty$-groupoid-completions of~(\ref{e27}).
Now, by direct inspection, this commutative square among 0-types induces an equivalence between horizontal fibers.  
Therefore, this diagram among spaces is a pullback.

Finally, direct inspection reveals that the simplicial space~(\ref{e22}) satisfies the completeness condition.

\end{proof}

\begin{cor}
\label{t27}
For each finite directed graph $\Gamma$, the $(\infty,1)$-category
$
\Free(\Gamma)
$
has the following properties.
\begin{enumerate}

\item
Its space of objects is a finite $0$-type.
In fact, the canonical map between spaces,
\[
{\sf unit}_{\Gamma}([0]) 
\colon
\Gamma^{(0)} := \Gamma([0])
\longrightarrow
\Obj\left( \Free\left( \Gamma \right) \right)
~,
\]
is an equivalence.
In particular, for each $v \in \Obj\left( \Free\left(\Gamma \right) \right)$, the group $\Aut_{\Free(\Gamma)}(v) \simeq \{\id_v\}$ is trivial.

\item
It is an ordinary category: for each pair $v_s,v_t\in \Obj\left( \Free\left(\Gamma \right) \right)$, the unit map followed by the composition map for the $(\infty,1)$-category $\Free\left( \Gamma \right)$ defines an equivalence between spaces:
\[
\underset{q \geq 0}
\coprod
\Bigl\{
(e_1,\dots,e_q)\in (\Gamma^{(1)})^{\times q} \mid 
s(e_1) = v_s \text{ and } 
t(e_q)=v_t 
\text{ and }
\text{ for } 0<i<q,~s(e_i) = t(e_{i+1})
\Bigr\}
\xra{~\simeq~}
\Hom_{\Free(\Gamma)}(v_s,v_t)
~.
\]
In particular, 
$\Hom_{\Free(\Gamma)}(v_s,v_t)$ is a 0-type, 
a point in which is a sequence
of directed edges in $\Gamma$ from $v_s$ to $v_t$.

\item
It is gaunt.

\end{enumerate}

\end{cor}

\begin{proof}
Statement~(3) follows from Statements~(1) and~(2).

Statement~(1) follows immediately from \Cref{t28}, as the case $p=0$, using that the only active morphism $[0]\to [q]$ in $\bDelta$ is an isomorphism.

Statement~(2) follows from \Cref{t28}, through the case $p=1$, as we now explain.
Observe, through direct inspection, that the fiber over $[q]$ of the projection
$\Obj\left(
\left(
\bDelta^{\sf idl}_{/^{\sf non.deg}\Gamma}
\right)^{[1]/^{\sf act}}
\right)
\to 
\Obj(\bDelta^{\sf idl})
$
is canonically identified as the 0-type
\[
\Bigl\{
(e_1,\dots,e_q)\in ( \Gamma^{(1)})^{\times q} \mid 
s(e_1) = v_s \text{ and } 
t(e_q)=v_t 
\text{ and }
\text{ for } 0<i<q,~s(e_i) = t(e_{i+1})
\Bigr\}
~\subset~
\Gamma([1])^q
~.
\]

\end{proof}

\begin{remark}
\label{r13}
Let $\Gamma$ be a finite directed graph.
\Cref{t27}(1) states that the space of objects $\Obj\bigl(\Free(\Gamma) \bigr)$ is the 0-type of vertices $\Gamma^{(0)}$.
\Cref{t27}(2) states that, for $v_s,v_t\in \Gamma^{(0)}$ vertices, the space of morphism in $\Free(\Gamma)$ is the 0-type of \bit{directed paths} in $\Gamma$ from $v_s$ to $v_t$, a point in which is an $\ell\geq 0$ together with a sequence of $\ell$ directed edges in $\Gamma$:
\[
\Bigl(
v_s
\xra{a_1}
u_1
\xra{a_2}
u_2
\xra{a_3}
\cdots
\xra{a_{\ell-1}}
u_{\ell-1}
\xra{a_\ell}
v_t
\Bigr)
~.
\]

\end{remark}

\Cref{t27} lends the following.
\begin{cor}
\label{t94}
Let $\Gamma$ and $\Xi$ be finite directed graphs. 
The space of morphisms $\Hom_{\Quiv}\bigl( \Xi , \Gamma \bigr)$ is the 0-type consisting of the following data.
\begin{itemize}

\item
A map $f^{(0)}\colon \Xi^{(0)} \to \Gamma^{(0)}$ between sets of vertices.

\item
For each non-degenerate directed edge $(s(e) \xra{e} t(e)) \in \Xi^{(1)}$, a directed path in $\Gamma$,
\[
f^{(1)}(e)
~=~
\Bigl(
f^{(0)}(s(e))
\xra{b_1(e)}
y_1(e)
\xra{b_2(e)}
y_2(e)
\xra{b_3(e)}
\cdots
\xra{b_{\ell_e-1}(e)}
y_{\ell_e-1}(e)
\xra{b_{\ell_e}(e)}
f^{(0)}(t(e))
\Bigr)
~,
\]
from $f^{(0)}(s(e))$ to $f^{(0)}(t(e))$.

\end{itemize}

\end{cor}

\begin{lemma}
\label{t29}
There is a canonical pullback diagram among $\infty$-categories
\begin{equation}
\label{e2}
\begin{tikzcd}
\bDelta^{\sf idl}
\arrow{r}
\arrow{d}
&
\digraphs
\arrow{d}{\Free}
\\
\bDelta
\arrow{r}
&
\Quiv
\end{tikzcd}
~.
\end{equation}

\end{lemma}

\begin{proof}
Let $A_p$ and $A_q$ be the finite linearly-directed graphs for which ${\sf Free}(A_p) = [p]$ and ${\sf Free}(A_q) = [q]$.  
Let $[p] \xra{\sigma}[q]$ be a morphism in $\bDelta$.
The morphism $f$ is idle if and only if, for each generating morphism $(i-1) \xra{f_i} i$ in $[p]$ (ie, for each edge in $A_p$), the morphism $\sigma(i-1) \xra{\sigma(f_i)} \sigma(i)$ in $[q]$ is either a generating morphism or an identity morphism (ie, an edge in $A_q$).
The result then follows from \Cref{t94}.

\end{proof}

\begin{lemma}
\label{t26}
The $\infty$-category $\Quiv$ has the following features.
\begin{enumerate}
\item
$\Quiv$ is an ordinary category.

\item
Both of the functors 
\[
\Free \colon \digraphs
\longrightarrow
\Quiv
\qquad
\text{ and }
\qquad
\Free \colon \digraphs
\longrightarrow 
\Cat_{(\infty,1)}
\]
are monomorphisms.

\end{enumerate}

\end{lemma}

\begin{proof}
\Cref{t27}(3) implies the defining fully faithful inclusion $\Quiv \hookrightarrow \Cat_{(\infty,1)}$ factors through $\Cat_{(1,1)}\subset \Cat_{(\infty,1)}$.  
\Cref{t27} also implies that, for each $\Gamma,\Xi \in \Quiv$, the 1-groupoid $\Hom_{\Cat_{(1,1)}}\left( \Free(\Xi) , \Free(\Gamma) \right)$ is in fact a 0-type.  
It follows that $\Quiv$ is an ordinary category, which is statement~(1).

We now prove statement~(2).
Because $\Quiv \subset \Cat_{(\infty,1)}$ is a full $\infty$-subcategory,
it is sufficient to prove that $\Free\colon \digraphs \to \Cat_{(\infty,1)}$ is a monomorphism.  
For this, it's enough to show that the unit of the adjunction~(\ref{e11}) evaluates on each object $\Gamma \in \digraphs \subset \PShv(\bDelta_{\leq 1})$ as a monomorphism in $\PShv(\bDelta_{\leq 1})$:
\[
\Gamma
\longrightarrow
\Free(\Gamma)
\qquad
\text{ is a monomorphism }
~.
\]
This is implied by \Cref{t27}(1)\&(2).

\end{proof}

After \Cref{t20},
\Cref{t26} implies the following.
\begin{cor}
\label{t34}
The functor $\digraphs \xra{\Free} \Quiv$ restricts as an equivalence between moduli spaces of objects:
\[
\Obj(\Free)\colon
\Obj(\digraphs)
\xra{~\simeq~}
\Obj( \Quiv )
~.
\]

\end{cor}

\begin{notation}
\label{d25}
In light of \Cref{t34}, we do not distinguish in notation or terminology between an object in $\Quiv$ and its corresponding finite directed graph.

\end{notation}

\begin{observation}
\label{t1.10}
Regarding each finite set as a finite directed graph with no (non-degenerate) edges defines a functor
\[
\Fin
~\hookrightarrow~
\digraphs
\]
which is fully faithful. 
Inspecting the values of $\Free(\Gamma)$ of \Cref{t27} reveals that the composite functor
\[
\Fin
~\hookrightarrow~
\digraphs
\xra{~\Free~}
\Quiv
\]
is fully faithful.
\end{observation}

\begin{definition}
\label{d1}
The subcategories of \bit{idle}, \bit{closed}, and \bit{creation} morphisms,
\[
\Quiv
~\supset~
\Quiv^{\sf idl}
~\supset~
\Quiv^{\sf cls}
~,~
\Quiv^{\sf cr}
~,
\]
are the respective image under the monomorphism $\digraphs \xra{\Free}\Quiv$ and the images of the \bit{monomorphisms}, and of the \bit{epimorphisms}.\footnote{In other words, the functors $\digraphs \xra{\Free} \Quiv^{\sf idl}$ and ${\sf diGraphs^{\sf fin,mono}} \xra{\Free} \Quiv^{\sf cls}$ and ${\sf diGraphs^{\sf fin,epi}} \xra{\Free} \Quiv^{\sf cr}$ are equivalences between categories.}
A morphism $\Gamma \xra{F} \Xi$ is \bit{active} if for every edge $g \in \Xi$, there exists an edge $f \in \Gamma$ such that $g$ is a factor of $F(f)$.  A morphism $\Gamma \xra{F} \Xi$ is a \bit{generating refinement morphism} if $F$ is fully faithful, the complement of whose image consists of a single object in $\Xi$ that has directed-valence $(1,1)$.\footnote{In other words, $\Xi$ is obtained from $\Gamma$ by replacing a directed edge by two composable directed edges, as it is canonically equipped with a morphism from $\Gamma$.}
A morphism $\Gamma \xra{F} \Xi$ is a \bit{refinement morphism} if it is a composite of generating refinement morphisms.  
We denote by
\[
\Quiv
~\supset~
\Quiv^{\sf act}
~\supset~
\Quiv^{\sf ref} 
~,
\]
the subcategories of active and refinement morphisms.

\end{definition}

\begin{observation}
\label{t24.2}
Using \Cref{t29}, 
the fully faithful functor $\bDelta \xra{\rho} \Quiv$ respects the subcategories:
\[
\bDelta^{\sf cls}
=
\bDelta \cap \Quiv^{\sf cls}
\qquad\text{ and }\qquad
\bDelta^{\sf cr}
=
\bDelta \cap \Quiv^{\sf cr}
\qquad\text{ and }\qquad
\bDelta^{\sf idl}
=
\bDelta \cap \Quiv^{\sf idl}
~.
\]
Furthermore, 
\[
\bDelta^{\sf act}
=
\bDelta \cap \Quiv^{\sf act}
~.
\]

\end{observation}

\subsection{Closed covers and closed sheaves}

\begin{definition}
\label{dd2}
A \bit{basic closed cover (in $\Quiv$)} is a diagram in $(\digraphs)^{\sf mono} \simeq \Quiv^{\sf cls}$ of the form
\[
\xymatrix{
\Gamma_0
\ar[rr]
\ar[d]
&&
\Gamma_+
\ar[d]
\\
\Gamma_-
\ar[rr]
&&
\Gamma
}
\]
that witnesses a pushout in $\digraphs$.\footnote{Warning: it need not be a pushout in $(\digraphs)^{\sf mono}$}
A presheaf $\Quiv^{\op} \xra{\cF} \Spaces$ is a \bit{closed sheaf (on $\Quiv$)} if $\cF(\emptyset) = *$ and it carries (the opposites of) each basic closed cover in $\Quiv$ to a limit diagram in $\Spaces$.
A \bit{closed cover (in $\Quiv$)} is a diagram $\cK^{\rcone} \to \Quiv^{\sf cls}$ for which, for each closed sheaf, $\Quiv^{\op} \xra{\cF} \Spaces$, the composite functor 
\[
(\cK^{\op})^{\lcone} = (\cK^{\rcone})^{\op} 
\to 
(\Quiv^{\sf cls})^{\op} 
\hookrightarrow 
\Quiv^{\op} 
\xra{\cF} 
\Spaces
\]
is a limit diagram.
The $\infty$-category of closed sheaves (on $\Quiv$) is the full $\infty$-subcategory 
\[
\Shv^{\sf cls} ( \Quiv )
~\subset~
\PShv(\Quiv)
\]
consisting of the closed sheaves (on $\Quiv$).
\end{definition}

\begin{remark}
The data of a basic closed cover in $\Quiv$ is equivalently that of a pullback diagram in $\digraphs$ in which each morphism is a monomorphism.  
In particular, a basic closed cover in $\Quiv$ is the data of a finite directed graph $\Gamma$ together with a pair $\Gamma_- ,\Gamma_+ \subseteq \Gamma$ of subgraphs whose union $\Gamma_- \cup \Gamma_+ = \Gamma$ is entire.
\end{remark}

\begin{observation}
\label{t32}
The fully faithful functor $\bDelta \xra{\rho} \Quiv$ carries basic closed covers in $\bDelta$ to basic closed covers in $\Quiv$.

\end{observation}

\begin{definition}
\label{d8}
Let $\Gamma$ be a finite directed graph.
\begin{enumerate}
\item
The \bit{exit-path category (of $\Gamma$)} is the full subcategory 
\[
\sE(\Gamma)
~\subset~
(\bDelta_{\leq 1})^{\cls}_{/\Gamma}
~:=~
(\bDelta_{\leq 1})^{\cls} \underset{\digraphs}\times {\digraphs}_{/\Gamma}
\]
consisting of those objects, which are morphisms between directed graphs $[p]\xra{\sigma} \Gamma$, in which $\sigma$ is a monomorphism.\footnote{This terminology aligns with more common use of ``exit-path category'' (see, for example~\cite{Treumann}).  Specifically, one can regard the geometric realization $|\Gamma|$ of $\Gamma$ as a stratified space: each vertex and the interior of each edge is a stratum.  As so, $\sE(\Gamma)$ defined here agrees with the exit-path category of this stratified space.}\footnote{The exit-path category $\sE(\Gamma)$ is sometimes referred to as the subdivision $\sd(\Gamma)$.}

\item
For $\cC\colon \bDelta^{\op} \to \cX$ a simplicial object in an $\infty$-category $\cX$ that admits finite limits, the \bit{$\cC$-valued representations of $\Gamma$} is the limit
\[
{\sf Rep}_\cC(\Gamma)
~:=~
\lim
\Bigl(
\sE(\Gamma)^{\op}
\xra{\rm forget}
(\bDelta_{\leq 1}^{\cls})^{\op}
\hookrightarrow
\bDelta^{\op}
\xra{\cC}
\cX
\Bigr)
~\in
\cX
~.
\]

\end{enumerate}

\end{definition}

\begin{figure}
\centering
\begin{subfigure}{0.5\textwidth}
\centering
\begin{tikzpicture}
	\node (A) at (0,0) {\textbullet};
	\node (B) at (3,0) {\textbullet};

	\node (C) at (7.2,-0.75) {\textbullet};
	\node (D) at (7.2,0.75) {\textbullet};
    
    \draw[->, shorten <=1.5mm] (B) arc[start angle=180, end angle=520, radius=0.75];
	\draw[->] (A) -- (B);
	
	 \draw[->, shorten <=1.5mm, shorten >=1.5mm] (C) arc[start angle=-90, end angle=90, radius=0.75];
	 \draw[->, shorten <=1.5mm, shorten >=1.5mm] (C) arc[start angle=-90, end angle=-270, radius=0.75];
	

	
\end{tikzpicture}
\caption{$\Gamma$} \label{fig:1a}
\end{subfigure}%
  
\hspace*{\fill}

\begin{subfigure}{0.5\textwidth}
\centering
\begin{tikzpicture}
\node (A) at (0,0) {$[0]$};
\node (B) at (1.5,0) {$[1]$};
\node (C) at (3,0) {$[0]$};
\node (D) at (4.5,0) {$[1]$};

\node (E) at (6.4,0) {$[1]$};
\node (F) at (7.2,0.7) {$[0]$};
\node (G) at (8,0) {$[1]$};
\node (H) at (7.2,-0.7) {$[0]$};

\draw[->] (A) -- (B) node[midway, above] {\footnotesize{0}};
\draw[->] (C) -- (B) node[midway, above] {\footnotesize{1}};


\draw[->, shorten <=2mm, shorten >=2mm] (C) arc[start angle=180, end angle=0, radius=0.75] node[midway, above] {\footnotesize{0}};
\draw[->, shorten <=2mm, shorten >=2mm] (C) arc[start angle=180, end angle=360, radius=0.75] node[midway, below] {\footnotesize{1}};

\draw[->, shorten >=-0.5mm, shorten <=-0.5mm] (F) edge[bend right] node[midway, above]{\footnotesize{1}} (E);
\draw[->, shorten >=-0.5mm, shorten <=-0.5mm] (F) edge[bend left] node[midway, above]{\footnotesize{1}} (G);

\draw[->, shorten >=-0.5mm, shorten <=-0.5mm] (H) edge[bend left] node[midway, below]{\footnotesize{0}} (E);
\draw[->, shorten >=-0.5mm, shorten <=-0.5mm] (H) edge[bend right] node[midway, below]{\footnotesize{0}} (G);

\end{tikzpicture}
\caption{$\sE(\Gamma) \to \bDelta_{\leq 1}^{\cls}$}

\end{subfigure}%
\caption{A graph $\Gamma$ and its exit path category.}

\end{figure}

\begin{observation}
\label{t66}
For each finite directed graph $\Gamma$, its exit-path category $\sE(\Gamma)$ is a finite category.  
Specifically, 
$\sE(\Gamma)$ is a gaunt category with 
\[
\Obj\bigl(\sE(\Gamma) \bigr) =  \Gamma^{(0)} \amalg \Gamma^{(1)}
\]
and, for $x, y\in \Obj\bigl( \sE(\Gamma) \bigr)$, 
\[
\Hom_{\sE(\Gamma)}(x,y)
~=~
\begin{cases}
\Bigl(
s^{-1}(x) \underset{\Gamma^{(1)}}\cap \{y\}
\Bigr)
\coprod
\Bigl(
t^{-1}(x) \underset{\Gamma^{(1)}}\cap \{y\}
\Bigr)
&
,\qquad
\text{if $x\in \Gamma^{(0)}$ and $y\in \Gamma^{(1)}$}
\\
\ast
&
,\qquad
\text{if $x=y$}
\\
\emptyset
&
,\qquad
\text{if $x\neq y$ and either $x\in \Gamma^{(1)}$ or $y \in \Gamma^{(0)}$}
\end{cases}
~.
\]  
In particular, the non-identity morphisms in $\sE(\Gamma)$ are precisely those from a vertex $x\in \Gamma^{(0)}$ to an edge $y\in \Gamma^{(1)}$ for $x$ is either the source or the target of $y$.  
Furthermore, there are no non-trivial composites in $\sE(\Gamma)$.
Moreover, $\sE(\Gamma)$ is a poset if and only if $\Gamma$ has no {\it self-loops} (ie, for each edge in $\Gamma$, its source is distinct from its target). 

\end{observation}

\begin{observation}
\label{t31}
For $\Gamma$ a finite directed graph, the functor
\begin{equation}
\label{e35}
\bigl(
\sE(\Gamma)
\bigr)^{\rcone}
\longrightarrow
\Quiv^{\cls}
~,\qquad
\begin{cases}
\left(
[p]\xra{\sigma} \Gamma
\right)
&
\longmapsto
\rho(p)
\\
+\infty
&
\longmapsto
\Gamma
\end{cases}
\end{equation}
is a closed cover in $\Quiv$.

\end{observation}

\subsection{Parametrizing $(\infty,1)$-categories by finite directed graphs}

The main outcome of this subsection is \Cref{t2}, which characterizes an $(\infty,1)$-category in terms of quiver representations into it.

\begin{lemma}
\label{t30}
The diagram among $\infty$-categories, which involves right Kan extension along $\rho$ and restriction to $\Quiv$,
\begin{equation}
\label{e33}
\begin{tikzcd}
\fCat_{(\infty,1)}
\arrow{d}[swap]{\Cref{e3}}
\arrow{r}{\Yo}
&
\PShv(
\fCat_{(\infty,1)}
)
\arrow{d}{\rm restriction}
\\
\PShv(\bDelta)
\arrow{r}[swap]{\rho_\ast}
&
\PShv(\Quiv)
\end{tikzcd}
~,
\end{equation}
canonically commutes.
\end{lemma}

\begin{proof}
By definition of the functor~(\ref{e3}), it factors as a composite,
\[
(\ref{e3})
\colon
\fCat_{(\infty,1)}
\xra{~\Yo~}
\PShv(\fCat_{(\infty,1)})
\xra{~\rm restriction~}
\PShv(\bDelta)
~,
\]
where the restriction functor above is along the standard fully faithful inclusion $\bDelta \hookrightarrow \fCat_{(\infty,1)}$.  
Now, \Cref{t6} grants that this fully faithful inclusion factors: $\bDelta \xra{\rho} \Quiv \hookrightarrow \fCat_{(\infty,1)}$.  
So we have a commutative diagram among $\infty$-categories:
\[
\begin{tikzcd}
\fCat_{(\infty,1)}
\arrow{d}[swap]{\Cref{e3}}
\arrow{r}{\Yo}
&
\PShv(
\fCat_{(\infty,1)}
)
\arrow{d}{\rm restriction}
\\
\PShv(\bDelta)
&
\PShv(\Quiv)
\arrow{l}[swap]{\rho^\ast}
\end{tikzcd}
~.
\]
Evoking the adjunction $(\rho^\ast , \rho_\ast)$ supplies a canonical morphism in $\Fun\left( \fCat_{(\infty,1)} , \PShv\left( \Quiv \right) \right)$,
\[
{\rm restriction} \circ {\Yo}
\longrightarrow
\rho_\ast \circ~(\ref{e3})
~.
\]
The value of this morphism on a flagged $(\infty,1)$-category $\cC\in \fCat_{(\infty,1)}$ is the canonical morphism in $\PShv(\Quiv)$, whose value on $\Gamma \in \Quiv$ is the map between spaces:
\begin{eqnarray*}
{\rm restriction}\circ {\Yo}(\cC)(\Gamma)
&
~=~
&
\Hom_{\fCat_{(\infty,1)}}\left( \Free(\Gamma) , \cC \right)
\\
&
\longrightarrow
&
\lim \left(
(\bDelta_{/\Free(\Gamma)})^{\op}
\xra{\rm forget}
\bDelta^{\op}
\xra{
\Hom_{\fCat_{(\infty,1)}}\left( [\bullet] , \cC \right)
}
\Spaces
\right)
\\
&
~\simeq~
&
\lim \left(
(\bDelta^{\op})^{\Gamma/}
\xra{\rm forget}
\bDelta^{\op}
\xra{
\Hom_{\fCat_{(\infty,1)}}\left( [\bullet] , \cC \right)
}
\Spaces
\right)
\\
&
~\simeq~
&
\rho_\ast \circ (\ref{e3})(\Gamma)
~,
\end{eqnarray*}
in which the arrow is obtained by the functor $\Hom_{\fCat_{(\infty,1)}}( - , \cC)$ to the canonical morphism in $\fCat_{(\infty,1)}$:
\[
\colim\Bigl( \bDelta_{/\Free(\Gamma)} \xra{\rm forget} \bDelta \hookrightarrow \fCat_{(\infty,1)} \Bigr)
\longrightarrow
\Free(\Gamma)
~.
\]
This canonical morphism in $\fCat_{(\infty,1)}$ is an equivalence because the functor $\bDelta \hookrightarrow \fCat_{(\infty,1)}$ strongly generates.

\end{proof}

\begin{notation}
\label{d6}
The functor
\[
\rho_\ast
\colon 
\fCat_{(\infty,1)}
\longrightarrow
\PShv(\Quiv)
\]
is the unambiguous diagonal functor of~(\ref{e33}).

\end{notation}

\begin{observation}
\label{t67}
The functor $\rho_\ast \colon \fCat_{(\infty,1)} \xra{(\ref{e3})} \PShv(\bDelta) \xra{\rho_\ast} \PShv(\Quiv)$ of \Cref{d6} is fully faithful.
Indeed, the functor~(\ref{e3}) is fully faithful, and fully faithfulness of $\rho$ implies fully faithfulness of $\rho_\ast$.

\end{observation}

The values of the fully faithful functor $\fCat_{(\infty,1)} \xra{\rho_\ast} \PShv(\Quiv)$ are rather simple, as in the following.
\begin{observation}
\label{t22}
By definition of the functor $\Free$ in terms of a left adjoint, for each finite directed graph $\Gamma$, the space of functors is canonically identified as the space of maps between graphs:
\begin{eqnarray}
\nonumber
\rho_\ast \cC(\Gamma)
&
~\underset{\rm Lem~\ref{t30}}\simeq~
&
\Hom_{\fCat_{(\infty,1)}}\bigl(\Free(\Gamma) , \cC \bigr)
\\
\nonumber
&
~\simeq~
&
\Hom_{\PShv(\bDelta_{\leq 1})}\bigl( \Gamma , \cC_{|\bDelta_{\leq 1}} \bigr)
\\
\nonumber
&
~\simeq~
&
\lim
\left(
\begin{tikzcd}
&
&
\Mor(\cC)^{\Gamma^{(1)}}
\arrow{d}{(s,t)}
\\
\Obj(\cC)^{\Gamma^{(0)}}
\arrow{r}{\rm diagonal}
&
\Obj(\cC)^{\Gamma^{(0)}}
\times
\Obj(\cC)^{\Gamma^{(0)}}
\arrow{r}{s^\ast \times t^\ast}
&
\Obj(\cC)^{\Gamma^{(1)}}
\times
\Obj(\cC)^{\Gamma^{(1)}}
\end{tikzcd}
\right)
\\
\nonumber
&
~=~
&
\Obj(\cC)^{\Gamma^{(0)}}
\underset{
\Obj(\cC)^{\Gamma^{(1)}}
\times
\Obj(\cC)^{\Gamma^{(1)}}
}
\times
\Mor(\cC)^{\Gamma^{(1)}}
~.
\end{eqnarray}
In particular, the space $\rho_\ast \cC(\Gamma)$ is a finite limit in which each term is $\Obj(\cC)$ or $\Mor(\cC)$.  

\end{observation}

\begin{remark}
\Cref{t22} articulates that $\rho_\ast\cC(\Gamma)$ is the moduli space of \bit{$\cC$-labels of $\Gamma$}, a point in which is an object in $\cC$ for each vertex in $\Gamma$, a morphism in $\cC$ for each (non-degenerate) edge in $\Gamma$, together with source-target compatibilities.  

\end{remark}

For the next result, recall the \Cref{dd2} of a closed sheaf on $\Quiv$.
\begin{theorem}
\label{t2}
The functor 
\[
\fCat_{(\infty,1)} \xra{~\rho_\ast~} \PShv(\Quiv)
\]
is fully faithful;
the image consists of the closed sheaves on $\Quiv$:
\[
\rho_\ast
\colon
\fCat_{(\infty,1)}
~\simeq~
\Shv^{\sf cls}(\Quiv)
\colon
\rho^\ast
~.
\]

\end{theorem}

\begin{proof}
Restriction and right Kan extension define an adjunction
\[
\rho^\ast
\colon
\PShv(\Quiv)
~\rightleftarrows~
\PShv(\bDelta)
\colon
\rho_\ast
~.
\]
We first show that this adjunction restricts as an adjunction between the full $\infty$-subcategories:
\begin{equation}
\label{e20}
\rho^\ast
\colon
\Shv^{\cls}(\Quiv)
~\rightleftarrows~
\fCat_{(\infty,1)}
\colon
\rho_\ast
~.
\end{equation}
So let $\cF \in \Shv^{\cls}(\Quiv)$.
We must show the restriction 
$\rho^\ast \cF_{|(\bDelta^{\cls})^{\op}} \colon 
(\bDelta^{\cls})^{\op}
\hookrightarrow 
\bDelta^{\op} 
\xra{\rho^\ast \cF} 
\Spaces$ 
carries (the opposites of) basic closed covers in $\bDelta$ to limit diagrams in $\Spaces$.
Note that this restriction is identical with the composite functor:
\[
\rho^\ast \cF_{|(\bDelta^{\cls})^{\op}} \colon (\bDelta^{\cls})^{\op} \xra{\rho_{|(\bDelta^{\cls})^{\op}}} (\Quiv^{\cls})^{\op}  \xra{\cF_{|(\Quiv^{\cls})^{\op}}} \Spaces
~.
\]
By \Cref{t32}, the first functor carries basic closed covers in $\bDelta$ to basic closed covers in $\Quiv$.
By assumption the second functor preserves limits.

Next, let $\cC \in \fCat_{(\infty,1)}$ be a flagged $(\infty,1)$-category.  
Clearly, the value $\rho_\ast \cC(\emptyset) \simeq \ast$ is final.
We must show the restriction
\[ \rho_\ast \cC_{|(\Quiv^{\cls})^{\op}} \colon (\Quiv^{\cls})^{\op} \hookrightarrow \Quiv^{\op} \xra{\rho_\ast \cC} \Spaces\]
carries basic closed covers to pullbacks.  
By \Cref{t30}, it is sufficient to show that the composite inclusion
\[
\Quiv^{\cls}
\hookrightarrow
\Quiv
\longrightarrow
\fCat_{(\infty,1)}
\]
carries basic closed covers to pushout diagrams.  
By Definition~\ref{d1}, this is implied by \Cref{t25}.

Now, the proposition is implied by the adjunction~(\ref{e20}) being an equivalence,
which is implied by both its unit and counit being by equivalences.
So let $\cC\in \fCat_{(\infty,1)}$.
The counit evaluates as the morphism in $\fCat_{(\infty,1)}$:
\[
\rho^\ast \rho_\ast \cC
\xra{~\rm counit~} 
\cC
~.
\]
This morphism evaluates on $[p]\in \bDelta$ as the canonical equivalence between spaces:
\[
\rho^\ast \rho_\ast \cC ([p])
\simeq
\rho_\ast \cC ( \rho(p))
\underset{\rm Lem~\ref{t30}}\simeq
\Hom_{\fCat_{(\infty,1)}}\bigl( 
\rho(p) , \cC
\bigr)
\xra{~\simeq~}
\Hom_{\fCat_{(\infty,1)}}\bigl( 
[p] , \cC
\bigr)
\simeq
\cC([p])
~,
\]
obtained by applying $
\Hom_{\fCat_{(\infty,1)}}\bigl( 
, \cC
\bigr)
$
to the canonical identification $[p]\xra{\simeq} \rho ([p])$ that defines the cellular realization (see \Cref{notation.brackets.functor.from.delta.op}).  
Therefore, the counit of the adjunction~\ref{e20} is by equivalences.

Next, let $\cF \in \Shv^{\cls}(\Quiv)$.
We seek to show the unit of the adjunction~(\ref{t20}),
\begin{equation}
\label{e34}
\cF
\xra{~\rm unit~}
\rho_\ast \rho^\ast \cF
~,
\end{equation}
is by equivalences.
Let $\Gamma \in \Quiv$.
The unit morphism~(\ref{e34}), together with its naturality, evaluates on $\Gamma \in \Quiv$
as the top horizontal map in a diagram among spaces,
\begin{equation}
\label{e36}
\begin{tikzcd}[column sep=1.5cm]
\cF(\Gamma)
\arrow{r}{\rm unit}
\arrow{d}
&
\rho_\ast \rho^\ast \cF(\Gamma)
\arrow{d}
\\
\lim
\Bigl(
\bigl(
\sE(\Gamma)^{\op}
\bigr)^{\lcone}
\xra{(\ref{e35})}
(\Quiv^{\cls})^{\op}
\xra{\cF}
\Spaces
\Bigr)
\arrow{r}[swap]{\lim~ \rm unit}
&
\lim
\Bigl(
\bigl(
\sE(\Gamma)^{\op}
\bigr)^{\lcone}
\xra{(\ref{e35})}
(\Quiv^{\cls})^{\op}
\xra{ \rho_\ast \rho^\ast  \cF}
\Spaces
\Bigr)
\end{tikzcd}
~,
\end{equation}
in which the downward maps are given by applying $\cF$ to~(\ref{e35}).
Using the hypothesis that $\cF$ is a closed sheaf on $\Quiv$,
the first part of this proof ensures the functor $\Quiv^{\op} \xra{\rho_\ast \rho^\ast \cF} \Spaces$ is a closed sheaf on $\Quiv$.
So by \Cref{t31}, which ensures~(\ref{e35}) is a closed cover, 
the vertical maps in~(\ref{e36}) are both equivalences.
Therefore, the top horizontal map is an equivalence if and only if the bottom horizontal map is an equivalence.
The bottom horizontal map is an equivalence provided, for each $[p]\in \bDelta_{\leq 1}$, the unit morphism~(\ref{e34}) evaluated on $\rho(p)\in \Quiv$,
\[
\rho^\ast {\rm unit}
\colon
\rho^\ast \cF( [p])
~\simeq~
\cF\bigl(\rho(p) \bigr)
\xra{~\rm unit~}
\rho_\ast \rho^\ast \cF\bigl( ( \rho(p) \bigr)
~\simeq~
\rho^\ast \rho_\ast \rho^\ast \cF ( [p] )
~,
\]
is an equivalence.  
This unit morphism fits into a commutative diagram among spaces:
\[
\begin{tikzcd}
\rho^\ast \cF( [p])
\arrow{dr}[sloped, swap]{\rho^\ast ~{\rm unit}}
\arrow{rr}{\id}
&&
\rho^\ast \cF([p])
\\
&
\rho^\ast \rho_\ast \rho^\ast \cF ( [p] )
\arrow{ru}[sloped, swap]{{\rm counit}~ \rho^\ast}
\end{tikzcd}
~.
\]
The upward map was already shown to be an equivalence.
It follows that the downward map is an equivalence, as desired.

\end{proof}

\begin{remark}
\Cref{t2} allows us to regard an $(\infty,1)$-category $\cC$ as a functor $\cC\colon \Quiv^{\op} \to \Spaces$ satisfying certain descent conditions.  
In particular, an $(\infty,1)$-category $\cC$ is characterized by its spaces of ``$\Gamma$-points'' $\cC(\Gamma):= \rho_\ast\cC(\Gamma)$, as $\Gamma$ ranges through finite directed graphs.

\end{remark}

The following result presents the limit expression of \Cref{t22} in more conceptual terms.
\begin{prop}
\label{t35}
Let $\cC\in \fCat_{(\infty,1)} \underset{(\ref{e3})}\subset \PShv(\bDelta)$ be a flagged $(\infty,1)$-category.
For each finite directed graph $\Gamma$, 
there is a canonical equivalence between spaces:
\[
\rho_\ast(\cC)(\Gamma)
\xra{~\simeq~}
{\sf Rep}_\cC(\Gamma)
~.
\]

\end{prop}

\begin{proof}
We explain the following sequence of equivalences among spaces:
\begin{eqnarray}
\label{j1}
\rho_\ast(\cC)(\Gamma)
&
\xra{~\simeq~}
&
\lim
\Bigl(
\sE(\Gamma)^{\op}
\xra{\rm Obs~\ref{t31}}
\Quiv^{\op}
\xra{\rho_\ast(\cC)}
\Spaces
\Bigr)
\\
\label{j2}
&
\xla{~\simeq~}
&
\lim
\Bigl(
\sE(\Gamma)^{\op}
\xra{\rm forget}
(\bDelta_{\leq 1}^{\cls})^{\op}
\hookrightarrow
\bDelta^{\op}
\xra{\cC}
\Spaces
\Bigr)
~=:~
{\sf Rep}_\cC(\Gamma)
~.
\end{eqnarray}
Using the second statement of \Cref{t2},
the equivalence~(\ref{j1}) follows from \Cref{t31}.
By inspection, the functor $\sE(\Gamma) \to \Quiv^{\cls}$ of \Cref{t31} factors through the fully faithful functor $(\bDelta_{\leq 1}^{\sf cls})^{\op} \xra{\rho} \Quiv^{\cls}$.
Using this, the equivalence~(\ref{j2}) is the fact that the counit of the $(\rho^\ast , \rho_\ast)$-adjunction is an equivalence.

\end{proof}

\subsection{Category-objects in $\cX$}
In this subsection, we extend the main result \Cref{t2} of the earlier subsections from $(\infty,1)$-categories to category-objects in an ambient $\infty$-category.

Recall from~\S\ref{sec.cat.def} the full $\infty$-subcategories
\[
\Cat_{(\infty,1)}
~\subset~
\fCat_{(\infty,1)}
~\subset~
\PShv(\bDelta)
=
\Fun(\bDelta^{\op} , \Spaces)
~.
\]
An object in the smaller full $\infty$-subcategory is an \bit{$(\infty,1)$-category} also known as a \bit{complete Segal space}, while an object in the intermediate fully $\infty$-subcategory is a \bit{flagged $(\infty,1)$-category} also known as a \bit{Segal space}.  
In this section, we recall a simple generalization of these notions, as the following.

\begin{definition}
\label{d11}
Let $\cX$ be an $\infty$-category.
The $\infty$-category of \bit{category-objects in $\cX$} is the full $\infty$-subcategory
\[
\fCat_1[\cX]
~\subset~
\Fun(\bDelta^{\op},\cX)
\]
consisting of those functors
\[
\cC
\colon
\bDelta^{\op}
\longrightarrow
\cX
\]
for which the composite functor $(\bDelta^{\sf cls})^{\op} \hookrightarrow \bDelta^{\op} \xra{\cC} \cX$ carries (the opposites of) basic closed covers to pullbacks.

\end{definition}

\begin{remark}
We explain the notation of Definition~\ref{d11}.
A category-object in $\Spaces$ is a \bit{Segal space}:
\[
\fCat_{(\infty,1)}
~:=~
\PShv^{\sf Segal}(\bDelta)
~=~
\fCat_1[\Spaces]
~.
\]
The work \cite{flagged} proves that a Segal space is precisely the same data as a \bit{flagged $(\infty,1)$-category}, which is a functor $\cG \to \cC$ from an $\infty$-groupoid to an $(\infty,1)$-category that is surjective on isomorphism-classes of objects.  

\end{remark}

\begin{notation}
\label{d47}
Let $\cX$ be an $\infty$-category.  Consider the full $\infty$-subcategory
\[
\Shv_\cX^{\sf cls}(\Quiv)
~\subseteq~
\Fun(\Quiv^{\op},\cX)
\]
consisting of those functors $\Quiv^{\op} \xra{\cF} \cX$ for which the restricted functor $(\Quiv^{\cls})^{\op} \hookrightarrow \Quiv^{\op} \xra{\cF} \cX$ maps (the opposites of) basic closed covers to pullbacks.  

\end{notation}

\begin{observation}
\label{t59}
Let $\cX \xra{f} \cY$ be a functor that preserves finite limits.
Postcomposition with $f$ defines the upper horizontal functors in commutative squares
\[
\begin{tikzcd}[column sep=2cm]
\fCat_1[\cX]
\arrow[dashed]{r}{\fCat_1[f]}
\arrow[hook]{d}[swap]{\ff}
&
\fCat_1[\cY]
\arrow[hook]{d}{\ff}
\\
\Fun(\bDelta^\op , \cX)
\arrow{r}[swap]{\Fun(\bDelta^\op,f)}
&
\Fun(\bDelta^\op, \cY)
\end{tikzcd}
\qquad
\text{and}
\qquad
\begin{tikzcd}[column sep=2cm]
\Shv^\cls_\cX(\Quiv)
\arrow[dashed]{r}{\Shv^\cls_f(\Quiv)}
\arrow[hook]{d}[swap]{\ff}
&
\Shv^\cls_\cY(\Quiv)
\arrow[hook]{d}{\ff}
\\
\Fun(\Quiv^\op,\cX)
\arrow{r}[swap]{\Fun(\Quiv^\op,f)}
&
\Fun(\Quiv^\op,\cY)
\end{tikzcd}
~.
\]
Furthermore, if $f$ is fully faithful, then so are each of the horizontal functors in these diagrams.
\end{observation}

Recall from \Cref{notation.brackets.functor.from.delta.op} the functor $\bDelta \xra{\rho} \Quiv$.
\begin{theorem}
\label{t65}
Let $\cX$ be an $\infty$-category with finite limits.  Restriction along $\rho$ is an equivalence between $\infty$-categories:
\[
\rho^\ast
\colon
\Shv_\cX^{\sf cls}(\Quiv)
\xra{~\simeq~}
\fCat_1[\cX]
~.
\]
The value of its inverse on $\cC \in \fCat_1[\cX]$ evaluates as the finite limit in $\cX$:
\[
\rho_\ast(\cC)
\colon
\Gamma
\longmapsto
{\sf Rep}_\cC(\Gamma)
~.
\]

\end{theorem}

\begin{proof}
\Cref{t32} implies the restriction functor $\rho^\ast \colon \Fun(\Quiv^{\op},\cX) \to \Fun(\bDelta^{\op} , \cX)$ indeed restricts as a functor 
$
\rho^\ast
\colon
\Shv_\cX^{\sf cls}(\Quiv)
\to
\fCat_1[\cX]
$.
It remains to show this functor is an equivalence.
Applying \Cref{t59} to the Yoneda functor $\cX \to \PShv(\cX)$, which preserves finite limits, affords the commutative diagram among $\infty$-categories:
\[
\begin{tikzcd}[column sep=2.5cm, row sep=1.5cm]
\Shv_\cX^{\sf cls}(\Quiv)
\arrow{r}{\Shv_{\Yo}^{\sf cls}(\Quiv)}
\arrow{d}[swap]{\rho^\ast}
&
\Shv_{\PShv(\cX)}^{\sf cls}(\Quiv)
\arrow{r}{\simeq}
\arrow{d}{\Shv_{\rho^\ast}^{\sf cls}(\Quiv)}
&
\Fun\bigl(
\cX^{\op}
,
\fCat_{(\infty,1)}
\bigr)
\arrow{d}{\Fun\bigl(
\cX^{\op}
,
\rho^\ast
\bigr)}
\\
\fCat_1[\cX]
\arrow{r}[swap]{\fCat_1[{\Yo}]}
&
\fCat_1\bigl[\PShv(\cX)\bigr]
\arrow{r}[swap]{\simeq}
&
\Fun ( \cX^{\op} , \Shv^{\sf cls}(\Quiv) )
\end{tikzcd}
~,
\]
in which the right two horizontal functors are equivalences via the standard adjunction $\Fun(\cB , \Fun(\cA , \cC) ) \simeq \Fun(\cB \times \cA , \cC) \simeq \Fun(\cA\times \cB , \cC) \simeq \Fun(\cA , \Fun(\cB , \cC))$.
Because the Yoneda functor is fully faithful, all of the horizontal functors in this diagram are fully faithful.
By \Cref{t2}, the right vertical functor is an equivalence.
We conclude that the left vertical functor $\rho^\ast$ is fully faithful.

Now, the universal property of limits is such that the Yoneda functor preserve limits.
By definition of $\fCat_1[-]$ and $\Shv_-^{\sf cls}(\Quiv)$, it follows that the left two horizontal functors preserve limits, and thereafter all of the horizontal functors preserve limits.  
\Cref{t2} also gives that the inverse of the right vertical functor evaluates as finite limits, as asserted in the present proposition.  
Consequently, the left vertical functor $\rho^\ast$, which we already established is fully faithful, is an equivalence, with inverse as asserted in the present proposition.

\end{proof}

\section{Cyclically-directed graphs}
Here, we recall the epicyclic category, cyclic category, and paracyclic category, and collect some facts about them.

\subsection{Symmetries of a circle}
Here, we record the symmetries of a circle, and introduce the \bit{Witt monoid}.

\begin{notation}\label{d15}
\begin{itemize}
\item[]

\item
The monoid $\ZZ^\times$ is that of integers with multiplication.  
\\
The monoid $\NN^\times$ is the submonoid of natural numbers with multiplication.

\item
The \bit{circle group} is $\TT \subset \CC^\times$ is the unit complex numbers with multiplication, regarded as a group-object in topological spaces, thusly presenting a group-object in $\Spaces$.

\end{itemize}

\end{notation}

\begin{observation}
\label{t43}
The circle group $\TT$ presents the group-object in spaces $\sB \ZZ$, which is the deloop of the commutative group $\ZZ$ of integers:
\[
\TT
~\simeq~
\sB\ZZ
\qquad
\Bigl(
~
\text{ in }
\Alg(\Spaces)
~
\Bigr)
~.
\]

\end{observation}

Consider the natural action of the monoids $\NN^\times$ and $\ZZ^\times$ on the circle group:
\begin{equation}
\label{e45}
\NN^\times
~\hookrightarrow~
\ZZ^\times
\longrightarrow
\End_{\sf Gps}(\TT)
~,\qquad
r
\longmapsto
\Bigl(
u 
\longmapsto
u^r
\Bigr)
~.\footnote{
Through \Cref{t43}, this action agrees with the functor $\sB$ applied to the action 
$\NN^\times \xra{~r\mapsto (i\mapsto  ri)~} \End(\ZZ)$.
}
\end{equation}
With respect to the action $\ZZ^\times \underset{(\ref{e45})}\lacts \TT$,
consider the semi-direct product monoid-object in $\cS$:
\[
\TT \rtimes \ZZ^\times
~.
\]
It is presented by the topological monoid whose underlying topological space is $\TT \times \ZZ^\times$, and whose multiplication rule is the continuous map
\[
\Bigl(
\TT \times \ZZ
\Bigr)
\times
\Bigl(
\TT \times \ZZ
\Bigr)
\longrightarrow
\Bigl(
\TT \times \ZZ
\Bigr)
~,\qquad
\bigl(
(w,r)
,
(z,s)
\bigr)
\longmapsto
(
w z^r , rs
)
~.
\]

\begin{prop}
The map
\begin{equation}
\label{e103}
\TT
\rtimes
\ZZ^\times
\xra{~\simeq~}
\End_{\Spaces}(\SS^1)
~,\qquad
(z,r)
\longmapsto
\Bigl(
u
\mapsto
z u^r
\Bigr)
~,
\end{equation}
canonically defines an equivalence between monoid-objects in $\Spaces$.

\end{prop}

\begin{proof}
Evidently, the indicated map canonically lifts along the forgetful morphism $\End_{\sf Top}(\SS^1) \xra{\rm forget} \End_{\Spaces}(\SS^1)$, where ${\sf Top}$ is a convenient category of topological spaces.  
It follows that the indicated map is a morphism between monoid-objects in $\Spaces$.

Next, using the topological group structure of $\TT = \SS^1$, the map
\[
\End_{\Spaces}(\SS^1)
\xra{~\simeq~}
\SS^1 
\times
\Omega \SS^1
~,\qquad
( \SS^1 \xra{f} \SS^1 )
\longmapsto
\bigl(
~
f(1)
~,~
f(1)^{-1} f
~
\bigr)
~,
\]
is an equivalence between spaces.
Direct inspection reveals that the composite map $\TT \rtimes \ZZ^\times \xra{\Cref{e103}} \End_{\Spaces}(\SS^1) \xra{\simeq} \SS^1 \times \Omega \SS^1$ is a product of equivalences, and is therefore an equivalence.

\end{proof}

\begin{observation}
\label{t100}
The diagram
\[
\begin{tikzcd}[column sep=2cm]
\TT \rtimes \ZZ^\times
\arrow{r}{\Cref{e103}}
\arrow{d}[swap]{\pr}
&
\End_{\Spaces}(\SS^1)
\arrow{d}{\deg}
\\
\ZZ^\times
\arrow{r}[swap]{r \longmapsto (d \mapsto rd) }
&
\End_{\sf Ab}(\ZZ)
\end{tikzcd}
\]
among monoid-objects in $\Spaces$ commutes, where the right downward morphism is given by applying $\sH_1$, $1^{\rm st}$ integral homology.

\end{observation}

\begin{definition}\label{d4}
The \bit{Witt monoid}\footnote{The notation $\WW$ stems from the fact that this will keep track of Frobenius and Verschiebung operators, along the lines of \cite{HessMad-Witt}.} is the semi-direct product monoid-object in $\Spaces$,
\[
\WW 
~:=~
\TT
\rtimes
\NN^\times
~,
\]
with respect to the action $\NN^\times \underset{(\ref{e45})}\lacts \TT$.
Specifically, $\WW$ is presented by the topological monoid whose underlying topological space is $\TT \times \NN$, and whose multiplication rule is the continuous map
\[
\Bigl(
\TT
\times
\NN
\Bigr)
\times
\Bigl(
\TT
\times
\NN
\Bigr)
\longrightarrow
\Bigl(
\TT
\times
\NN
\Bigr)
~,\qquad
\bigl(
(w,r)
,
(z,s)
\bigr)
\longmapsto
(
w z^r
,
rs
)
~.
\]

\end{definition}

\begin{observation}
\label{t96}
Using that both of the monoids $\TT$ and $\NN^\times$ are commutative, there is a canonical identification between monoids:
\[
\WW^{\op}
~\simeq~
\NN^\times \ltimes \TT
~.
\]

\end{observation}

\begin{observation}
The canonical projection
\[
\BW
\xra{~\pr~}
\BN
\]
is a left fibration.
It straightens as the composite functor $\BN \xra{ \bigl\langle \NN^\times \underset{\Cref{e45}} \lacts \TT \bigr \rangle} {\sf Groups} \xra{\sB} \Spaces$.

\end{observation}

\begin{observation}
\label{t101}
In light of \Cref{t100}, there is a canonical pullback square among monoid-objects in $\Spaces$:
\[
\begin{tikzcd}[column sep=2cm]
\WW
\arrow{r}
\arrow{d}[swap]{\pr}
&
\End_{\Spaces}(\SS^1)
\arrow{d}{\deg}
\\
\NN^\times
\arrow{r}[swap]{r \longmapsto (d \mapsto rd)}
&
\End_{\sf Ab}(\ZZ)
\end{tikzcd}
~.
\]
Furthermore, because the bottom horizontal morphism is a monomorphism, so is the top horizontal morphism.

\end{observation}

\begin{remark}
In fact, the Witt monoid can be recognized as the monoid of framed self-coverings of the circle:
\[
{\sf Imm}^{\sf fr}(\SS^1)
~\simeq~
\WW
~.
\]

\end{remark}

\subsection{The epicyclic category}
Here, we recall the epicyclic category $\w{\bLambda}$, and construct a natural functor from it to $\BW$.

Recall from \Cref{d26} the definition of a cyclically-directed graph.
\begin{definition}
\label{d23}
The \bit{epicyclic} category is the $\infty$-subcategory
\[
\w{\bLambda}
~\subset~
\Cat_{(\infty,1)}
\]
consisting of those $(\infty,1)$-categories that are the values of $\Free$ on finite directed graphs that are cyclically-directed, and non-constant functors between such.
\footnote{
In the small handful of places that it has appeared in the literature, $\w{\bLambda}$ is actually defined as the opposite of what we have indicated here.  (See e.g.\! \cite{BFG-epi} for a foundational survey, which attributes the definition to an unpublished letter from Goodwillie to Waldhausen dating back to 1987.)  We have chosen our convention in the interest of uniformity.}

\end{definition}

\begin{observation}
\label{t6'}
There is a
canonical monomorphism (between $\infty$-subcategories of $\Cat_{(\infty,1)}$):
\[
\w{\bLambda}
~\hookrightarrow~
\Quiv
~.
\]

\end{observation}

\begin{observation}
\label{t97}
The finite directed graph consisting of a single vertex and no (non-degenerate) edges is the final object $\ast\in \Quiv$.
Consequently, the monomorphism $\w{\bLambda} \overset{\rm Obs~\ref{t6'}}\hookrightarrow \Quiv$ admits a final extension 
\[
\w{\bLambda}^{\rcone}
~\hookrightarrow~ 
\Quiv
~, 
\]
which is a monomorhpism,
whose value on the cone point is $\ast\in \Quiv$.
\end{observation}

Through a series of observations, we make the category $\w{\bLambda}$ more explicit.

\begin{observation}
\label{t115}
Let $\Gamma$ be a finite cyclically-directed graph.
Both of the maps from the set of non-degenerate edges to the set of vertices,
\[
\Gamma^{(0)}
\xla{~s~}
\Gamma^{(1)}
\xra{~t~}
\Gamma^{(0)}
~,
\]
are bijections between finite sets.  

\end{observation}

Recall from \Cref{r13} the notion and notation of directed paths in a directed graph.
\begin{observation}
\label{t116}
Let $\Gamma$ be a finite cyclically-directed graph.
Let $u,v \in \Gamma^{(0)}$ be vertices.
A directed path from $u$ to $v$ determines, and is determined by, the number of times it passes through $u$.  
More precisely, 
the map from the set of directed paths in $\Gamma$ from $u$ to $v$ to the set of non-negative integers,
\[
\underset{ k \geq 0}
\coprod
\Bigl\{
u
\to 
x_1
\to
\dots
\to
x_k
\to 
v
\Bigr\}
\longrightarrow
\ZZ_{\geq 0}
~,\qquad
\bigl(
u
\to 
x_1
\to
\dots
\to
x_k
\to 
v
\bigr)
\longmapsto
{\sf Card}\bigl\{
i
\mid 
x_i = u
\bigr\}
~,
\]
is a bijection.

\end{observation}

\begin{notation}
\label{d30}
Let $\Gamma$ be a cyclically-directed graph.
Let $v \in \Gamma^{(0)}$ be a vertex.
Denote by $e_v \in \Gamma^{(1)}$ the unique non-degenerate edge in $\Gamma$ whose source is $v$.
\end{notation}

\begin{observation}
\label{t113}
Let $\Gamma$ and $\Xi$ be finite cyclically-directed graphs.
By \Cref{t94}, and using \Cref{t115} and \Cref{t116}, the set $\Hom_{\Cat}\bigl( \Free(\Gamma) , \Free(\Xi) \bigr)$ is the 0-type in which a point is the following data:
\begin{itemize}

\item
a map $f^{(0)}\colon \Gamma^{(0)} \to \Xi^{(0)}$ between sets of vertices;

\item
a map $d\colon \Gamma^{(0)} \xra{\ell\mapsto d_\ell} \ZZ_{\geq 0}$~.

\end{itemize}
Indeed, given $\Free(\Gamma)  \xra{f} \Free(\Xi)$, for $v\in \Gamma^{(0)}$, the value $d_v \in \ZZ_{\geq 0}$ is that corresponding through the bijection of \Cref{t116} to the directed path $f^{(1)}(e_v)$ in $\Xi$, where $e_v$ is as in \Cref{d30}.

\end{observation}

\begin{definition}
\label{d101}
Let $\Gamma$ and $\Xi$ be finite cyclically-directed graphs.
The \bit{degree} map is defined through \Cref{t113} as
\[
\deg
\colon
\Hom_{\Cat}\bigl( \Free(\Gamma) , \Free(\Xi) \bigr)
\longrightarrow
\NN
~,\qquad
(f^{(0)} , d )
\longmapsto
1 
+
\underset{v \in \Gamma^{(0)}} \sum 
d_v
~.
\]

\end{definition}

\begin{observation}
\label{t117}
The degree map is multiplicative.  
Specifically, for $\Gamma$, $\Xi$, and $\chi$ finite cyclically-directed graphs, and for $\Free(\Gamma) \xra{ f } \Free(\Xi) \xra{ g } \Free(\chi)$, there is an equality between numbers 
\[
\deg(g \circ f)
=
\deg(g) \deg(f)
~.
\]

\end{observation}

\Cref{t117} enables the following.

\begin{definition}
\label{d100}
The \bit{degree} functor
\[
\deg
\colon
\w{\bLambda}
\longrightarrow
\BN
\]
is given by, for each pair of objects $\Gamma , \Xi \in \w{\bLambda}$, 
the degree map on spaces of morphisms:
$
\Hom_{\w{\bLambda}} \bigl( \Gamma , \Xi \bigr)
\xra{\deg}
\NN
$
.

\end{definition}

\begin{observation}
\label{t95}
\begin{enumerate}

\item[]

\item
By \Cref{d23}, the space $\Obj(\w{\bLambda})$ is that of finite cyclically-directed graphs.
Consequently, by \Cref{t92}, for each object $\Gamma\in \w{\bLambda}\subset \Cat$, its $\infty$-groupoid-completion $|\Gamma| \simeq \SS^1$ is non-canonically equivalent with a circle.

\item
The previous point implies there is a unique filler in the diagram among $\infty$-categories,
\begin{equation}
\label{e105}
\begin{tikzcd}
\w{\bLambda}
\arrow[hook]{r}
\arrow[dashed]{d}[swap]{|-|}
&
\Cat
\arrow{d}{|-|}
\\
\fB \End_{\Spaces}(\SS^1)
\arrow[hook]{r}
&
\Spaces
\end{tikzcd}
~,
\end{equation}
in which the downward functor is given by taking $\infty$-groupoid-completion.

\item
The diagram among $\infty$-categories
\[
\begin{tikzcd}
\w{\bLambda}
\arrow{r}{|-|}
\arrow{d}[swap]{\deg}
&
\fB \End_{\Spaces}(\SS^1)
\arrow{d}{\deg}
\\
\BN
\arrow[hook]{r}
&
\fB \ZZ^\times
\end{tikzcd}
\]
commutes.

\item
After \Cref{t101}, the previous point gives that the factorization \Cref{e105} uniquely factors further:
\begin{equation}
\label{e106}
\w{\bLambda} \xra{|-|} \BW
~.
\end{equation}

\end{enumerate}

\end{observation}

\begin{construction}
\label{d31}
Let $\Gamma$ be a finite cyclically-directed graph.
Let $C \xra{g} |\Gamma|$ be a morphism in $\BW \subset \fB \End_{\Spaces}(\SS^1)$ to the geometric realization of $|\Gamma|$.
By definition, this geometric realization fits into a coequalizer diagram among spaces:
\[
\Gamma^{(1)}
\overset{s}{\underset{t}\rightrightarrows}
\Gamma^{(0)}
\longrightarrow
|\Gamma|
~.
\]
Base change along $C \xra{g} |\Gamma|$ results in a coequalizer diagram among spaces:
\[
(g^\ast \Gamma)^{(1)}
\overset{s}{\underset{t}\rightrightarrows}
(g^\ast \Gamma)^{(0)}
\longrightarrow
C
~.
\]
Both of the spaces $(g^\ast \Gamma)^{(0)}$ and $(g^\ast \Gamma)^{(1)}$ are finite 0-types, and therefore $(g^\ast \Gamma)^{(1)}
\overset{s}{\underset{t}\rightrightarrows}
(g^\ast \Gamma)^{(0)}$
is a finite directed graph $g^\ast \Gamma$.
By its construction, $g^\ast \Gamma$ is cyclically-directed, and it is equipped with a morphism $g^\ast \Xi \to \Xi$ between directed graphs for whose geometric realization is identical with $C \xra{g} |\Xi|$.  

\end{construction}

\begin{lemma}
\label{t120}
The functor $\w{\bLambda} \xra{|-|} \BW$ is a Cartesian fibration.
A Cartesian morphism with target $\Gamma$, over a morphism $C \xra{g} |\Gamma|$ in $\BW$, is $g^\ast \Gamma \to \Gamma$ of \Cref{d31}.  

\end{lemma}

\begin{proof}
Let $\Gamma$ be a finite cyclically-directed graph.
Let $D \xra{h}C \xra{g} |\Gamma|$ be a pair of composable morphisms in $\BW$.
Because base change composes, there is a canonical identification $(g \circ h)^\ast \Gamma \simeq h^\ast ( g^\ast \Gamma)$ between directed graphs.  
We are therefore reduced to showing $\w{\bLambda} \xra{|-|} \BW$ is a locally Cartesian fibration, with locally Cartesian morphisms as stated.  

Let $\Xi \xra{f} \Gamma$ be a morphism in $\w{\bLambda}$.
Consider the resulting morphism $c_1 \xra{\langle |\Xi| \xra{|f|} |\Gamma| \rangle} \BW$.
We seek to show there is a unique morphism $\Xi \xra{\w{f}} |f|^\ast \Gamma$ in $\w{\bLambda}_{|c_1}$.
Well, the morphism $f$ determines a map $\Xi^{(0)} \simeq \Obj\left(\Free(\Xi)\right) \xra{\Obj(f)} \Obj\left( \Free(\Gamma) \right) \simeq \Gamma^{(0)}$ that fits into the diagram among spaces:
\[
\begin{tikzcd}
\Xi^{(0)}
\arrow{d}[swap]{\Obj(f)}
\arrow{r}
&
{|\Xi|}
\arrow{d}{|f|}
\\
\Gamma^{(0)}
\arrow{r}
&
{|\Gamma|}
\end{tikzcd}
~.
\]
By the universal property of pullbacks, this commutative diagram supplies a unique map $\Obj\left( \Free(\Xi) \right) \simeq \Xi^{(0)} \xra{\Obj(\w{f})} \Gamma^{(0)} \simeq \Obj\left( \Free(\Gamma) \right)$.
By \Cref{t113}, the space of morphisms $\Xi \xra{\w{f}} |f|^\ast \Gamma$ in $\w{\bLambda}$ is therefore the space of maps $\Xi^{(0)} \xra{d} \ZZ_{\geq 0}$.
Now, because ${\sf deg}$ is multiplicative (\Cref{t117}), $\deg(\w{f}) =1$.
Consequently, the space of morphisms $\Xi \xra{\w{f}} |f|^\ast \Gamma$ in $\w{\bLambda}_{|c_1}$ is contractible, since there is a unique map $\Xi^{(0)} \xra{d} \ZZ_{\geq 0}$ for which $1+\underset{x\in \Xi^{(0)}}\sum d_x = 1$.

\end{proof}

\subsection{The cyclic category}

Here, we define the cyclic category, originally defined by Connes, in terms of the epicyclic category.

\Cref{t117} enables the following.

\begin{definition}[\cite{connes}]
\label{d24}
The \bit{cyclic} category is the subcategory
\[
\bLambda
~\subset~
\w{\bLambda}
\]
consisting of all of the objects, which are finite cyclically-directed graphs, and those morphisms $f$ for which $\deg(f)=1$.  

\end{definition}

\begin{observation}
\label{t93}
\Cref{t113} implies each morphism $\Gamma \to \Xi$ in $\bLambda$ is determined by its values on vertices.
In other words, the functor
\[
\Obj
\colon
\bLambda
\longrightarrow
{\sf Sets}
\]
induces monomorphisms between spaces of morphisms.  

\end{observation}

\begin{observation}
\label{t118}
There is a canonical diagram among $\infty$-categories in which each square is a pullback:
\[
\begin{tikzcd}[column sep=1.5cm]
\bLambda
\arrow{r}{|-|}
\arrow{d}
&
\BT
\arrow{r}
\arrow{d}
&
\ast
\arrow{d}
\\
\w{\bLambda}
\arrow{r}{|-|}[swap]{\Cref{e106}}
&
\BW
\arrow{r}[swap]{\pr}
&
\BN
\end{tikzcd}
~.
\]
In other words, for $\Gamma$ and $\Xi$ finite cyclically-directed graphs, a morphism $\Free(\Gamma) \xra{f} \Free(\Xi)$ belongs to $\bLambda$ if and only if the induced map between $\infty$-groupoid-completions $|\Gamma| \xra{|f|} |\Xi|$ is an equivalence.

\end{observation}

\subsection{The paracyclic category}
We recall the definition of the paracyclic category and some of its properties and symmetries.

For $B$ and $P$ linearly ordered sets, the \bit{$P$-fold join of $B$} is the linearly ordered set $B^{\star P}$ whose underlying set is the product $P \times B$ of underlying sets of $P$ and of $B$, and whose linear order is the dictionary order: $(p,b) < (p',b')$ means either $p<p'$ or $p=p'$ and $b<b'$.

\begin{definition}[\cite{GetzlerJones-paracyclic}]\label{d5}
An object in the \bit{paracyclic category} $\copara$ is a nonempty linearly ordered set $I$ for which, for each $i<j$ in $I$ the interval $[i,j]\subset I$ is finite, equipped with an action by the additive group $\ZZ$
for which, for each $i\in I$, there is a relation $i < 1 \cdot i$ for all $i\in I$.
A morphism in $\copara$ is a $\ZZ$-equivariant (weakly) order-preserving map.   
Composition is composition of maps.
Identities are identity maps.

\end{definition}

\begin{example}
For each $\ell \in \NN$, consider the subset $\frac{1}{\ell} \ZZ \subset \QQ$.
This subset inherits a linear order from the standard linear order on $\QQ$.
This subset is evidently invariant under the translation action $\ZZ \lacts \QQ$, which is order-preserving.
In this way, $\frac{1}{\ell}\ZZ$ is an object in $\copara$.

\end{example}

\begin{lemma}
Every object in $\copara$ is non-canonically isomorphic with $\frac{1}{\ell}\ZZ$ for some $\ell\in \NN$.
\end{lemma}

\begin{proof}
Let $(\ZZ \lacts I) \in \copara$.
Using that $I$ is assumed nonempty, choose $i_0 \in I$.  
Take $\ell \in \NN$ such that $\bigl[~i_0 ~,~ 1 \cdot i_0~ \bigr] = \{i_0 =j_0< j_1 < \dots < j_\ell = 1\cdot i_0\} $.  
Consider the map 
\begin{equation}
\label{k2}
\frac{1}{\ell} \ZZ
\longrightarrow
I
~,\qquad
a
\longmapsto 
( \lfloor a \rfloor ) \cdot j_{\ell( a - \lfloor a \rfloor )}
\end{equation}
where $\lfloor a \rfloor$ is the floor of the rational number $a$.
The proof is complete upon showing this map~(\ref{k2}) is an isomorphism in $\copara$.

This map~(\ref{k2}) is evidently $\ZZ$-equivariant. 
Let $a,b \in \frac{1}{\ell} \ZZ$ such that $( \lfloor a \rfloor ) \cdot j_{\ell( a - \lfloor a \rfloor )} = ( \lfloor b \rfloor ) \cdot j_{\ell( b - \lfloor b \rfloor )}$.
Then $( \lfloor a \rfloor - \lfloor b \rfloor) \cdot j_{\ell( a - \lfloor a \rfloor )} =  j_{\ell( a - \lfloor a \rfloor )}$.
This is an element in both of the half-open intervals
\[
\bigl[ (\lfloor a \rfloor - \lfloor b \rfloor) \cdot i_0 ,  (\lfloor a \rfloor - \lfloor b \rfloor \cdot + 1) \cdot i_0 \bigr)
\qquad
\text{ and }
\qquad
\bigl[ i_0 , 1 \cdot i_0 \bigr)
~.
\]
Note that, for each $0<c \in \ZZ$, there are relations in $I$:
\[
(-c) \cdot i_0 < \cdots (-1) \cdot i_0 < i_0 < 1 \cdot i_0 < \cdots < c \cdot i_0
~.
\]
Consequently, because the intersection of the above half-open intervals is not empty, then $\lfloor a \rfloor - \lfloor b \rfloor = 0$.  
It then follows that $\ell ( a- \lfloor a \rfloor) = \ell ( b- \lfloor b \rfloor)$, which implies $a=b$.
We conclude that the map~(\ref{k2}) is injective.

We next prove that the map~(\ref{k2}) is surjective.
Let $i\in I$.   
Assume $i_0 \leq i$.  
Using that intervals in $I$ are assumed finite, enumerate the interval $[i_0 , i] = \{ i_0 = k_0 < k_1 < \dots < k_r = i]$.
The division algorithm in $\ZZ$ implies there exists integers $a,\mu\in \ZZ$ with $0\leq \mu < \ell$ such that $r = a \ell + \mu$.
Then $i=a\cdot j_{\mu}$.  
Establishing the case in which $i_0  \geq i$ is similar.

\end{proof}

\begin{observation}
\label{t61}
\begin{enumerate}
\item[]

\item
Let $\lambda = (\ZZ \lacts I) \in \copara$ be an object.
Consider the poset $\Fun^{\sf surj}(I,[1])$ of surjective functors from the poset $I$ to the poset $[1]$.
This poset is in fact linearly ordered.
Via pre-composition, the action $\ZZ \lacts I$ determines an action $\ZZ = \ZZ^{\op} \lacts \Fun^{\sf surj}(I,[1])$. 
In fact, $\lambda^{\vee}:= \bigl(\ZZ \lacts \Fun^{\sf surj}(I,[1])\bigr)$ is an object in $\copara$.

\item
The assignment $\lambda \mapsto \lambda^{\vee}$ defines a functor
\[
\copara
\longrightarrow
\para
~,\qquad
\lambda
\mapsto
\lambda^{\vee}
~.
\]
This functor is an equivalence between categories, with inverse given by $\mu^\circ \mapsto (\mu^{\vee})^{\circ}$.

\end{enumerate}

\end{observation}

Observe the functor
\begin{equation}\label{e15}
\bDelta
\longra
\copara
~,\qquad
[p]
\mapsto 
[p]^{\bigstar \ZZ}
~,
\end{equation}
in which $[p]^{\star \ZZ}$ is the $\ZZ$-fold join, regarded as a linearly ordered set with $\ZZ$-action given by shifting the joinands.  
We record the following well-known technical assertion.
\begin{lemma}[Proposition~4.2.8 of~\cite{LurieRot}]
\label{t36}
The functor~(\ref{e15}) is initial.
\end{lemma}

Using that the category $\bDelta^{\op}$ is sifted, \Cref{t36} and \Cref{t61}(2) imply the following.
\begin{cor}
\label{t62}
The category $\copara$ is both sifted and cosifted.  
In particular, the $\infty$-groupoid-completion $| \copara | \simeq \ast$ is final.

\end{cor}

\begin{observation}
\label{t111}
Let $\lambda = (\ZZ \underset{\alpha_\lambda}\lacts I) \in \copara$ be an object.
\begin{itemize}
\item
Using commutativity of the group $\ZZ$, the order-preserving automorphism $\alpha_\lambda$ of $I$ is canonically $\ZZ$-equivariant, which is to say it defines an automorphism
\[
\alpha_\lambda \in \Aut_{\copara}(\lambda)
~.
\]

\item
For each $r\in \NN^\times$, the consider the linearly ordered set $\{1<\cdots <r\}^{\star I}$ with order-preserving $\ZZ$-action given by 
\[
\alpha_{\varphi_r(\lambda)}
\colon
(i,k)
\mapsto 
\begin{cases}
(i,k+1)
&
,\qquad
\text{if $1\leq k<r$}
\\
( \alpha_\lambda(i) , k )
&
,\qquad
\text{if $k=r$}
\end{cases}
~,
\]
which we regard as an object $\varphi_r( \lambda ) \in \copara$.

\end{itemize}

\end{observation}

\begin{lemma}
\label{t10}
The associations of Observation~\ref{t111}
assemble as an action
\[
\NN^\times \ltimes \TT
\underset{\rm Obs~\ref{t96}}
{~\simeq~}
\WW^{\op}
~\lacts~
\copara
\]
of the opposite of the Witt monoid on the paracyclic category $\copara$.

\end{lemma}

\begin{proof}
The automorphism $\alpha_\lambda$, as well as, for each $r\in \NN^\times$, the object $\varphi_r(\lambda)$, are each functorial in $\lambda\in \copara$, which is to say they define lifts:
\[
\begin{tikzcd}
&
\Fun(\sB \ZZ , \copara)
\arrow{d}{\fgt}
\\
\copara
\arrow[dashed]{ru}[sloped]{\alpha}
\arrow{r}[swap]{=}
&
\copara
\end{tikzcd}
\qquad
\text{and}
\qquad
\begin{tikzcd}
&
\Fun(\Nx , \copara)
\arrow{d}{\ev_1}
\\
\copara
\arrow{r}[swap]{=}
\arrow[dashed]{ru}[sloped]{\varphi}
&
\copara
\end{tikzcd}
~.
\]
The lefthand lift canonically assembles as an action $\TT \lacts \copara$.
The righthand lift, together with the canonical identifications
$
\varphi_r \circ \varphi_s
\cong
\varphi_{rs}
$
supplied by the unique identification between linearly ordered sets $\{1<\cdots <r\}^{\star \{1<\cdots<s\}} \cong \{1<\cdots< sr\}$,
canonically assemble as an action $\NN^\times \lacts \copara$.
Furthermore, for each $r\in \NN^\times$, the identification between functors 
$
(\alpha_{\varphi_r})^{\circ r}
\cong
\varphi_r(\alpha)
$,
canonically extend the actions $\TT \lacts \copara$ and $\NN^\times \lacts \copara$ as an action $\WW^{\op} \underset{\rm Obs~\ref{t96}} \simeq \NN^\times \ltimes \TT \lacts \copara$.

\end{proof}

\subsection{Comparing the paracyclic, cyclic, and epicyclic categories}
Here, we observe an action of the Witt monoid on the paracyclic category, and identify the right-lax coinvariants of this action as the epicyclic category.

Observe the canonical functor
\begin{equation}\label{e37}
\copara
\longrightarrow
\w{\bLambda}
~,\qquad
(\ZZ \lacts I)
\mapsto
I_{{\sf h}\ZZ}
~,
\end{equation}
where $I_{{\sf h}\ZZ}$ is the finite cyclically-directed graph that is the $\ZZ$-quotient of the infinite linearly-directed graph whose set of vertices is the underlying set of $I$ and whose set of non-degenerate directed edges is that of consecutive relations in $I$.

\begin{lemma}
\label{t111''}
The square
\begin{equation}
\label{f22}
\begin{tikzcd}
\copara
\arrow{r}{\Cref{e37}}
\arrow{d}[swap]{!}
&
\w{\bLambda}
\arrow{d}{|-|}
\\
\ast
\arrow{r}
&
\BW
\end{tikzcd}
\end{equation}
among $\infty$-categories canonically commutes.
Furthermore, this square is a pullback.
\end{lemma}

\begin{proof}
We first establish commutativity of the square \Cref{f22}.
Let $( \ZZ \lacts I) \in \copara$.
Consider the canonical maps among spaces:
\[
\sB \ZZ
~\simeq~
\ast_{{\sf h}\ZZ}
\xla{~!_{{\sf h}\ZZ}~}
| I |_{{\sf h}\ZZ}
\longrightarrow
| I_{{\sf h}\ZZ}|
~,
\]
in which $I$ is regarded as the free category on the directed graph in which a vertex is an element in $I$ and a directed edge is a consecutive pair of elements in $I$.
As so, observe that its $\ZZ$-coinvariants is the free category on the cyclically-directed graph $I_{{\sf h}\ZZ}$.
Because $\infty$-groupoid-completion $|-|$ is a left adjoint, it preserves $\ZZ$-coinvariants.
Therefore, 
the rightward map is an equivalence.
Because $I$ is nonempty and linearly ordered, the $\infty$-groupoid-completion $|I| \simeq \ast$ is contractible.
Therefore, the leftward map is an equivalence. 
In summary, there is a canonical equivalence
\begin{equation}
\label{f20}
\sB \ZZ
~\simeq~
| I_{{\sf h}\ZZ}|
~.
\end{equation}
This equivalence between spaces is evidently functorial in $(\ZZ \lacts I) \in \copara$, thereby assembling as the sought commutativity of the square \Cref{f22}.

Using \Cref{t118}, the commutative square \Cref{f22} supplies the commutative square among $\infty$-categories:
\begin{equation}
\label{f23}
\begin{tikzcd}
\copara
\arrow{r}
\arrow{d}[swap]{!}
&
\bLambda
\arrow{d}{|-|}
\\
\ast
\arrow{r}
&
\sB \TT
\end{tikzcd}
~.
\end{equation}
Furthermore, because the base-change is associative, \Cref{t118} implies \Cref{f22} is a pullback provided \Cref{f23} is a pullback.
We now show that \Cref{f23} is a pullback.  

Tautologically, the functor~(\ref{e37}) canonically carries, for each $\lambda \in \copara$, the automorphism $\lambda \xra{\alpha_\lambda} \lambda$ in $\copara$ to the identity automorphism $\ol{\lambda} \xra{\id_{\ol{\lambda}}} \ol{\lambda}$ in $\bLambda \subset \w{\bLambda}$.  
Consequently, the functor \Cref{e37} is $\TT = \sB\ZZ$-invariant, 
thereby extending as a functor from the $\TT$-coinvariants:
\begin{equation}
\label{e38}
\bigl( 
\copara 
\bigr)_{{\sf h}\TT}
\longrightarrow
\bLambda
~.
\end{equation}
Moreover, the identification $\sB \ZZ \overset{\Cref{f20}}\simeq | \ol{\lambda}|$ is invariant with respect to the automorphism $|\ol{\alpha_\lambda}|$.  
Consequently, the commutative square of \Cref{f23} descends as a commutative square
\begin{equation}
\label{f24}
\begin{tikzcd}
(\copara)_{{\sf h} \sB \ZZ}
\arrow{r}{\Cref{e38}}
\arrow{d}[swap]{!_{{\sf h}\sB \ZZ}}
&
\bLambda
\arrow{d}{|-|}
\\
\ast_{{\sf h}\sB \ZZ}
\arrow{r}[swap]{\simeq}
&
\sB \TT
\end{tikzcd}
~.
\end{equation}
As the bottom horizontal functor is an equivalence between connected $\infty$-groupoids, \Cref{f23} is a pullback if and only if the functor \Cref{e38} is an equivalence.

By inspection, the functor \Cref{e37} is surjective on spaces of objects.
It follows that the functor \Cref{e38} is surjective on spaces of objects.
So it remains to show the functor~(\ref{e38}) is fully faithful.

By inspection, the functor $\copara \to \bLambda$ is surjective on spaces of morphisms.  
Therefore, the functor \Cref{e38} is surjective on spaces of morphisms.
So the functor~(\ref{e38}) is fully faithful provided it induces a monomorphism on spaces of morphisms.
Observe the canonically commutative diagram among categories:
\[
\begin{tikzcd}[column sep=1.5cm]
\copara
\arrow{r}{(\ref{f23})}
\arrow{d}[swap]{\fgt}
&
\bLambda 
\arrow{d}{\Obj}
\\
\Mod_{\ZZ}(\Set)
\arrow{r}[swap]{(-)_{{\sf h}\ZZ}}
&
\Set
\end{tikzcd}
~.
\]
Now let $\lambda = (\ZZ\lacts I)$ and $\mu = (\ZZ\lacts J)$ be objects in $\copara$.
Choose convex fundamental domains $C \subset I$ and $D \subset J$.
By definition of objects in $\copara$, there are canonical identifications between underlying $\ZZ$-sets: 
$\ZZ \times C \cong I$ and $\ZZ\times D \cong J$.  
Therefore, the forgetful functor $\copara \to \Mod_{\ZZ}({\sf Sets})$ induces monomorphisms between spaces of morphisms:
\[
\Hom_{\copara}(\lambda, \mu)
~\hookrightarrow~
\Hom_{\Mod_{\ZZ}({\sf Sets})} ( I , J )
~\simeq~
\Hom_{{\sf Sets}}(C , D)
\times \ZZ
~.
\]
It follows that the canonical map between spaces of morphisms,
\[
\Hom_{(\copara)_{{\sf h}\sB \ZZ}}([\lambda], [\mu])
~\simeq~
\Hom_{\copara}(\lambda, \mu)_{{\sf h}\ZZ}
~\hookrightarrow~
\Hom_{{\sf Sets}}(C , D)
~,
\]
is a monomorphism.
It follows from \Cref{t93} that $\Hom_{(\copara)_{{\sf h}\sB \ZZ}}([\lambda], [\mu]) \xra{(\ref{f24})} \Hom_{\bLambda}(\ol{\lambda} , \ol{\mu})$ is a monomorphism, as desired.

\end{proof}

\begin{cor}
\label{t119}
There is a diagram among $\infty$-categories,
\begin{equation}
\label{paracyclic.cyclic.epicyclic}
\begin{tikzcd}
\copara
\arrow[two heads]{r}
\arrow{d}{\sf loc}
&
\bLambda
\arrow[hook, two heads]{r}
\arrow{d}{\sf loc}
&
\w{\bLambda}
\arrow{d}{\sf loc}
\\
\ast
\arrow[two heads]{r}
&
\BT
\arrow[hook, two heads]{r}
&
\BW
\end{tikzcd}
~,
\end{equation}
with the following properties.
\begin{enumerate}

\item
Each square is a pullback.  

\item
The downward functors are Cartesian fibrations.

\item
The downward functors are localizations.

\item
The horizontal functors are surjective.

\item
The right horizontal functors are monomorphisms.

\end{enumerate}

\end{cor}

\begin{proof}
\Cref{t111''} implies the outer square is a pullback; \Cref{t118} implies the right square is a pullback.  
By the universal property of pullbacks, it follows that the left square is a pullback as well, thereby establishing property~(1).
Consequently, because Cartesian fibrations are closed under the formation of base-change, \Cref{t120} implies all of the downward functors are Cartesian fibrations, thereby establishing property~(2).

\Cref{t36} gives that the $\infty$-groupoid-completion of the paracyclic category $| \para | \xla{\simeq} | \bDelta^{\op} | \simeq \ast$ is contractible.
This is to say that the left vertical functor is a localization.
Properties~(1)-(2) thereafter imply all of the vertical functors are localizations, thereby establishing property~(3).

Because the lower horizontal functors are surjective, with the right functor a monomorphism, properties~(1)-(2) imply the upper horizontal functors are surjective, and the upper right horizontal functor is a monomorphism, thereby establishing properties~(4)-(5).

\end{proof}

\begin{observation}
\label{t98}
The actions $\TT \lacts \copara$ and $\WW^{\op} \lacts \copara$ codified by \Cref{t119}
agree with the actions of \Cref{t10}:
\[
(\copara)_{{\sf h} \TT}
\xra{~\simeq~}
\bLambda
\qquad
\text{ and }
\qquad
(\copara)_{{\sf r.lax}\WW}
\xra{~\simeq~}
\w{\bLambda}
~.
\]
Furthermore, the latter equivalence is implemented by, for each $r\in \NN^\times$, the natural transformation 
\[
\pi_r
\colon
\ol{\varphi_r}
\longrightarrow
(\ref{e37})
~.
\]
that evaluates on each $\lambda = (\ZZ \underset{\alpha_\lambda}\lacts I) \in \copara$, the functor between categories
\[
\pi_r(\lambda)
\colon
\ol{ \varphi_r(\lambda) }
:=
\bigl(
\{1<\dots<r\}^{\star I}
\bigr)_{{\sf h}\ZZ}
\xra{~ ( ! ^{\star I} )_{{\sf h}\ZZ}~}
(\ast^{\star I})_{{\sf h}\ZZ} = I_{{\sf h}\ZZ}
=:
\ol{ \lambda }
~.
\]

\end{observation}

\subsection{Directed cycles in directed graphs}
In this technical subsection, we identify, for each finite directed graph $\Gamma$, the $\infty$-groupoid-completion of the overcategroy $\copara_{/\Gamma}$ in terms of the Witt monoid and a set of directed cycles in $\Gamma$.

Recall from \Cref{d26} the notion of a finite cyclically-directed graph.
\begin{terminology}
\label{d19}
Let $\Gamma\in \digraphs$ be a finite directed graph.
\begin{enumerate}

\item
An \bit{directed cycle} (in $\Gamma$) is a morphism $\chi \xra{\gamma} \Gamma$ in $\digraphs$ with the following properties.
\begin{itemize}
\item
$\chi$ is either cyclically-directed or $\chi = \ast$ is final.

\item
For any factorization
\[
\begin{tikzcd}
\chi
\arrow{rr}{\gamma}
\arrow{rd}[sloped, swap]{q}
&&
\Gamma
\\
&
\chi'
\arrow{ru}[sloped, swap]{\ol{\gamma}}
\end{tikzcd}
\]
in $\digraphs$ in which $\chi'$ is either cyclically-directed or $\chi'= \ast$, 
the morphism $q$ is an isomorphism.

\end{itemize}

\item
The \bit{set of directed cycles (in $\Gamma$)} is
\[
\sZ^{\sf dir}(\Gamma)
~:=~
\Bigl\{
\text{ directed cycles in $\Gamma$ }
\Bigr\}_{/\sim}
~,
\]
where $\sim$ is the equivalence relation of isomorphism between cyclically-directed graphs over 
$\Gamma$.

\end{enumerate}

\end{terminology}

\begin{observation}
\label{t55}
For $\Gamma$ a finite directed graph, there is a monomorphism between sets
\[
\Gamma^{(0)}
\longrightarrow
\sZ^{\sf dir}(\Gamma)
\]
selecting the elements represented by those $\chi\to \Gamma$ in which $\chi = \ast$ has a single vertex and no (non-degenerate) edges.

\end{observation}

\begin{notation}
Let $\Gamma$ be a finite directed graph.
\begin{enumerate}
\item
Denote by
\[
\w{\bLambda}_{/^{\sf nc}\Gamma}
~\subset~
\w{\bLambda}_{/\Gamma}
\]
the full subcategory consisting of those $\chi \to \Gamma$ that are not constant.  

\item
Denote by
\[
(\copara)_{/^{\sf nc} \Gamma}
~:=~
\copara 
\underset{\w{\bLambda}} \times
\w{\bLambda}_{/^{\sf nc} \Gamma}
~\subset~
(\copara)_{/\Gamma}
\]
the full subcategory consisting of those $(\lambda , \ol{\lambda} \xra{f} \Gamma)$ in which $f$ is not constant.

\end{enumerate}

\end{notation}

\begin{observation}
\label{t56}
Let $\Gamma$ be a finite directed graph.
\begin{enumerate}
\item
By definition of a directed cycle, the full subcategory of $\digraphs_{/\Gamma}$ consisting of the directed cycles is, in fact, a groupoid.
Furthermore, the group of automorphisms of each object in this category is trivial.
As the isomorphism-classes this full subcategory are evidently indexed by the set $\sZ^{\sf dir}(\Gamma)$, there results a fully faithful functor
\[
\sZ^{\sf dir}(\Gamma)
\hookrightarrow
\digraphs_{/\Gamma}
\]
whose image consists of the directed cycles.

\item
The resulting composite monomorphism
$
\sZ^{\sf dir}(\Gamma)
\hookrightarrow
\digraphs_{/\Gamma}
\overset{\Free}{\underset{\rm Obs~\ref{t20}}\hookrightarrow}
{\Quiv}_{/\Gamma}
$
factors through ${\w{\bLambda}^{\rcone} }_{/\Gamma} \overset{\rm Obs~\ref{t97}}\subset {\Quiv}_{/\Gamma}$ and, in fact, does so fully faithfully,
\begin{equation}
\label{e67}
\sZ^{\sf dir}(\Gamma)
\xra{~\rm fully~faithful~}
{ \w{\bLambda}^{\rcone} }_{/\Gamma}
~,
\end{equation}
with image consisting of those $\chi \to \Gamma$ that are directed cycles.

\item
With respect to \Cref{t55}, the fully faithful functor~\Cref{e67} restricts as a fully faithful functor 
\begin{equation}
\label{f25}
\sZ^{\sf dir}(\Gamma) \setminus \Gamma^{(0)}
\xra{~\rm fully~faithful~}
\w{\bLambda}_{/^{\sf nc}\Gamma}
~,
\end{equation}
whose image consists of those {\it directed cycles} $\chi \to \Gamma$ that are not constant.

\end{enumerate}

\end{observation}

\begin{lemma}
\label{t54.1}
Let $\Gamma$ be a finite directed graph.
The functor \Cref{f25} is a right adjoint.

\end{lemma}

\begin{proof}

Denote the set $\cZ := \bigl(
\sZ^{\sf dir}(\Gamma) \setminus \Gamma^{(0)}
\bigr)$.
Let 
$
(\zeta \xra{f} \Gamma) \in 
\w{\bLambda}_{/^{\sf nc}\Gamma}
$.
We must show that the undercategory $\cZ ^{f/}$ has an initial object.

By \Cref{t56}, an object in this undercategory $\cZ^{f/}$ is a non-constant directed cycle $\chi \xra{\gamma} \Gamma$ together with a morphism $\zeta \xra{q} \chi$ in $\w{\bLambda}$ over $\Gamma$:
\begin{equation}
\label{e96}
\begin{tikzcd}
\zeta
\arrow{rr}{f}
\arrow{rd}[sloped, swap]{q}
&
&
\Gamma
\\
&
\chi
\arrow{ru}[sloped, swap]{\gamma}
\end{tikzcd}
~.
\end{equation}
By \Cref{t94}, the functor $f$ is the datum of a map $f^{(0)} \colon \zeta^{(0)} \to \Gamma^{(0)}$ and, for each pair of cyclically adjacent vertices $z,{\sf suc}(z)\in \zeta^{(0)}$, a linearly-directed graph $\{z\to y_1(z) \to \cdots \to y_{\ell_z-1}(z) \to {\sf suc}(z)\}$
together with a non-degenerate extension in $\digraphs$:
\[
\begin{tikzcd}[column sep=2cm]
\{z,{\sf suc}(z)\}
\arrow{r}{f^{(0)}}
\arrow[hook]{d}
&
\Gamma
\\
\{z \to y_1 \to \cdots \to y_{\ell_z-1} \to {\sf suc}(z)\}
\arrow[dashed, bend right=10]{ru}[sloped, swap]{f(z \to {\sf suc}(z))}
\end{tikzcd}
~.
\]
Consider the cyclically-directed graph $\w{\chi}$ obtained by cyclically gluing the linearly-directed graphs $\{z\to y_1 \to \cdots \to y_{\ell_z-1} \to {\sf suc}(z) \}$:
\[
\w{\chi}
~:=~
\zeta^{(0)}
\underset{ \zeta^{(0)} \coprod \zeta^{(0)} }
\coprod
\Bigl(
\underset{z\in \zeta^{(0)}} \coprod \{z \to y_1 \to \dots \to y_{\ell_z-1} \to {\sf suc}(z) \}
\Bigr)
~.
\]
By construction of $\w{\chi}$, 
there is a canonical morphism $\zeta \xra{q'}\w{\chi}$ in $\Cat$, and a canonical morphism $\w{\chi} \xra{\w{\gamma}} \Gamma$ in $\digraphs$, fitting into a commutative diagram among gaunt categories:
\[
f
\colon
\zeta
\xra{~q'~}
 \w{\chi} 
\xra{~\w{\gamma}~}
\Gamma 
\]
in which $\w{\gamma}$ is non-degenerate.
By construction, the morphism $q$ has degree 1.
With respect to the canonical homomorphism $G := \Aut_{/\Gamma}(\w{\chi}) \to \Aut(\w{\chi})$, there is a canonical factorization in $\digraphs$:
\[
\w{\gamma}
\colon
\w{\chi}
\xra{~\rm quotient~}
\w{\chi}_{/G}
=:
\chi
\xra{~\gamma~}
\Gamma
~.
\]
Because the group of automorphisms of a cyclically-directed graph is a finite cyclic group, then $G$ is a finite cyclic group.  
Because $\w{\gamma}$ is non-degenerate, the map $\gamma$ is also non-degenerate; because $f$ is not constant, then $\gamma$ is not constant.  
Furthermore, by construction, the map $\gamma$ is a directed cycle.  
We have a factorization $f\colon \zeta \xra{q:={\rm quotient} \circ q'} \chi \xra{\gamma} \Gamma$ as in diagram~(\ref{e96}).
In particular, the undercategory $\cZ^{f/}$ is nonempty.  
Lastly, the construction of this object is such that it is initial in the undercategory $\cZ^{f/}$, as desired.

\end{proof}

\begin{cor}
\label{t52.2}
Let $\Gamma$ be a finite directed graph.
There is a $\WW^{\op}$-equivariant functor
\[
\bigl(
\sZ^{\sf dir}(\Gamma)
\!\setminus\!
\Gamma^{(0)}
\bigr)
\times
\WW
\longrightarrow
(\copara)_{/^{\sf nc}\Gamma}
\]
that is a right adjoint.

\end{cor}

\begin{proof}
Consider the general set-up.
Let $\cE_0 \overset{\rho} \hookrightarrow \cE$ be a fully faithful right adjoint between $\infty$-categories.
Denote its left adjoint as $\cE_0 \xla{\lambda} \cE$ and its unit as $\id_\cE \xra{\eta} \rho\lambda$. 
Let $\cE \xra{\pi} \cB$ be a Cartesian fibration between $\infty$-categories.  
Let $b\in \cB$.
Consider the span among $\infty$-categories:
\[
\cE_0
\underset{\cB} \times
\cB^{b/}
\xla{~\pi~}
\cE_0
\underset{\cE} \times
\Ar(\cE)^{\cE_{|b}}
=:
\Ar(\cE)^{\cE_{|b}}_{|\cE_0}
\xra{~\ev_s~}
\cE_{|b}
~.
\]
Because $\pi$ is a Cartesian fibration, the leftward functor is a left adjoint localization, with right adjoint given by selecting the $\pi$-Cartesian morphisms.
Because $\cE_0 \xra{\rho} \cE$ is a right adjoint, the rightward functor is a right adjoint, with left adjoint given by $\eta$.  
Using that the composition of right adjoints is a right adjoint, we have a right adjoint functor among $\infty$-categories:
\begin{equation}
\label{f26}
\cE_0
\underset{\cB} \times
\cB^{b/}
\longrightarrow
\cE_{|b}
~.
\end{equation}
Furthermore, this functor is evidently $\End_\cB(b)^{\op}$-equivariant. 

Now, specialize these parameters as follows.
\begin{itemize}
\item
Take
$
\Bigl(
\cE_0 \xra{\rho} \cE
\Bigr) 
=
\Bigl(
\sZ^{\sf dir}(\Gamma)
\!\setminus\!
\Gamma^{(0)}
\xra{(\ref{f25})}
\w{\bLambda}_{/^{\sf nc} \Gamma}
\Bigr)
$.
\Cref{t56}(3)
ensures the named functor is indeed fully faithful;
\Cref{t54.1} ensures the named functor is indeed a right adjoint.

\item
Take
$
\Bigl(
\cE \xra{\pi} \cB
\Bigr) 
=
\Bigl(
\w{\bLambda}_{/^{\sf nc} \Gamma}
\xra{|-| \circ {\rm forget}} \BW
\Bigr)
$.
Using that, by definition, each morphism in $\w{\bLambda}$ is a non-constant functor, the forgetful functor $\w{\bLambda}_{/^{\sf nc} \Gamma} \xra{\rm forget} \w{\bLambda}$ is a right fibration.
\Cref{t120} states that the functor $\w{\bLambda} \xra{|-|} \BW$ is a Cartesian fibration.
Because the composition of Cartesian fibrations is a Cartesian fibration, the named functor is indeed a Cartesian fibration.

\item
Take $(b \in \cB) = (\ast \in \BW)$.

\end{itemize}
Using that $\sZ^{\sf dir}(\Gamma)
\!\setminus\!
\Gamma^{(0)}$ is a 0-type, there is a non-canonical identification between $\WW^{\op}$ spaces:
\[
\bigl(
\sZ^{\sf dir}(\Gamma)
\!\setminus\!
\Gamma^{(0)}
\bigr)
\times
\WW
~\simeq~
\bigl(
\sZ^{\sf dir}(\Gamma)
\!\setminus\!
\Gamma^{(0)}
\bigr)
\times
\End_{\BW}(\ast)
~\simeq~
\bigl(
\sZ^{\sf dir}(\Gamma)
\!\setminus\!
\Gamma^{(0)}
\bigr)
\underset{\BW}
\times
\BW^{\ast/}
~.
\]
Using \Cref{t111''}, there is a canonical identification $(\copara)_{/^{\sf nc}\Gamma} \xra{\simeq} (\w{\bLambda}_{/^{\sf nc}\Gamma})_{|\ast}$.
So the equivariant right adjoint \Cref{f26} can be identified as a $\WW^{\op}$-equivariant right adjoint
\[
\bigl(
\sZ^{\sf dir}(\Gamma)
\!\setminus\!
\Gamma^{(0)}
\bigr)
\times
\WW
\longrightarrow
(\copara)_{/^{\sf nc}\Gamma}
~,
\]
as desired.

\end{proof}

\begin{cor}
\label{t52}
Let $\Gamma$ be a finite directed graph.
The $\infty$-groupoid-completion of the overcategory $\copara_{/\Gamma} := \copara\underset{\Quiv} \times (\Quiv)_{/\Gamma}$, as it is equipped with the resulting $\WW^{\op}$-module structure, is canonically identified as the $\WW^{\op}$-space
\[
\Gamma^{(0)}
\coprod
\Bigl(
\bigl(
\sZ^{\sf dir}(\Gamma)
\!\setminus\!
\Gamma^{(0)}
\bigr)
\times
\WW
\Bigr)
~\simeq~
\bigl |
\copara_{/\Gamma}
\bigr |
~,
\]
where $\Gamma^{(0)}$ is the set of vertices of $\Gamma$, $\sZ^{\sf dir}(\Gamma)$ is the set of directed cycles in $\Gamma$, and $\WW$ is the underlying space of the Witt monoid.

\end{cor}

\begin{proof}
The proof is complete upon explaining the following sequence of equivalences among $\WW^{\op}$-spaces:
\begin{eqnarray*}
\bigl|
\copara_{/\Gamma}
\bigr|
&
~\simeq~
&
\bigl|
\copara_{/^{\sf c}\Gamma}
\amalg
\copara_{/^{\sf nc}\Gamma}
\bigr|
\\
&
~\simeq~
&
\bigl|
\copara_{/^{\sf c}\Gamma}
\bigr|
\amalg
\bigl|
\copara_{/^{\sf nc}\Gamma}
\bigr|
\\
&
~\simeq~
&
| 
\Gamma^{(0)}
\times
\copara 
|
\amalg
\bigl|
\copara_{/^{\sf nc}\Gamma}
\bigr|
\\
&
~\simeq~
&
\Gamma^{(0)}
\times
| \copara |
\amalg
\bigl|
\copara_{/^{\sf nc}\Gamma}
\bigr|
\\
&
~\simeq~
&
\Gamma^{(0)}
\amalg
\bigl|
\copara_{/^{\sf nc}\Gamma}
\bigr|
\\
&
~\simeq~
&
\Gamma^{(0)}
\coprod
\Bigl(
\bigl(
\sZ^{\sf dir}(\Gamma)
\!\setminus\!
\Gamma^{(0)}
\bigr)
\times
\WW
\Bigr)
~.
\end{eqnarray*}

Consider the full subcategories
\[
\copara_{/^{\sf c}\Gamma}
~,~
\copara_{/^{\sf nc}\Gamma}
~\subset~
\copara_{/\Gamma}
\]
consisting of those pairs $\bigl(\lambda , \ol{\lambda} \xra{f}   \Gamma \bigr)$ in which $f$ is respectively constant and not constant.  
Note, also, that every object in $\copara_{/\Gamma}$ belongs to one of these two full subcategories.
Note that there are no morphisms in $\copara_{/\Gamma}$ from an object in one of these full subcategories to an object in the other.
Therefore, the canonical functor
\begin{equation}
\label{e110}
\copara_{/^{\sf c}\Gamma}
~\amalg
\copara_{/^{\sf nc}\Gamma}
\longrightarrow
\copara_{/\Gamma}
\end{equation}
is an equivalence.
This implies the first equivalence.
The second equivalence follows from the fact that $\infty$-groupoid-completion preserves coproducts.

Observe that the canonical functor
\[
\copara
\times
\Gamma^{(0)}
\xra{~\simeq~}
\copara_{/^{\sf c}\Gamma}
~,\qquad
( \lambda , v )
\longmapsto
\bigl(\lambda , \ol{\lambda} \xra{{\sf const}_v}   \Gamma \bigr)
~,
\]
is an equivalence between categories.
This implies the third equivalence.
The fourth equivalence follows from the fact that $\infty$-groupoid-completion preserves products, and that $\Gamma^{(0)}$ is a set.  
\Cref{t62} implies the fifth equivalence.
\Cref{t52.2} implies the last equivalence, since adjunctions induce equivalences on $\infty$-groupoid-completions.

\end{proof}

\section{Universal Hochschild homology}

One might reasonably regard the functor $\bDelta^{\op}  \xra{\rho} \Quiv^{\op}$ as a category-object in $\Quiv^{\op}$.
Regarded as so, we contemplate its Hochschild homology $\sHH(\rho)$.
The category $\Quiv^{\op}$ admits very few colimits, and this Hochschild homology does not exist in $\Quiv^{\op}$.
In this section, we formally adjoin $\sHH(\rho)$ to $\Quiv^{\op}$, keeping that finite products exits, resulting in an $\infty$-category $\M$.
By construction, for $\cX$ an $\infty$-category that admits finite limits and geometric realizations such that products distribute over geometric realizations, then for $\cC$ a category-object in an $\infty$-category $\cX$ there is a unique extension
\[
\begin{tikzcd}[column sep=1.5cm]
(\Quiv)^\op
\arrow{r}{\Rep_\cC}
\arrow[hook]{d}
&
\cX
\\
\M
\arrow[dashed]{ru}[sloped, swap]{\w{\Rep}_\cC}
\end{tikzcd}
\]
such that $\w{\Rep}_\cC$ preserves finite products and $\w{\Rep}_\cC\colon \sHH(\rho) \mapsto \sHH(\cC)$.
Remarkably, we give an explicit ``object \& morphism'' description of this universal $\M$.
As so, the endomorphisms of $\sHH(\rho)$ in $\M$ codify universal (possibly non-invertible) symmetries of Hochschild homology of any $(\infty,1)$-category.

\subsection{Definition of Hochschild homology}
For this subsection, we fix an $\infty$-category $\cX$ that admits finite limits and geometric realizations.
We are now positioned to define Hochschild homology of an $(\infty,1)$-category, and more generally of a category-object in $\cX$.

\begin{notation}
\label{notn.functor.chi.from.Delta.to.Quiv}
Consider the composite functor
\begin{equation}
\label{g20}
\copara
\xra{~\Cref{e37}~}
\w{\bLambda}
\xra{~\rm Obs~\ref{t6'}~}
\Quiv
~,\qquad
(\ZZ \lacts I)
\longmapsto
I_{\htpy \ZZ}
~,
\end{equation}
whose image consists of the cyclically-directed quivers.
We write
\[
\chi
\colon
\bDelta
\xra{~\Cref{e15}~}
\copara
\xra{~\Cref{g20}~}
\Quiv
~,
\qquad
{[p]}
~\longmapsto~
\left\{
\begin{tikzcd}
&
1
\arrow{r}
&
\cdots
\arrow{r}
&
p-1
\arrow{rd}
\\
0
\arrow{ru}
&
&
&
&
p
\arrow{llll}
\end{tikzcd}
\right\}
\]
for the resulting composite functor.
\end{notation}

\begin{definition}
\label{dHH}
\bit{(Non-stable) Hochschild homology} is the functor
\[
\sHH
\colon
\fCat_1[\cX]
\longrightarrow
\cX
~,\qquad
\cC
\longmapsto 
\colim\left(
\para
\xra{\Cref{g20}}
\Quiv^{\op}
\xra{{\sf Rep}_\cC}
\cX
\right)
~.
\]
\end{definition}

\begin{observation}
\label{t250}
The functor $\sHH$ of \Cref{dHH} exists. 
Indeed, Lemma~\ref{t36} states that the functor $\bDelta\xra{\Cref{e15}} \copara$ is initial.
Therefore, for each category-object $\cC$ in $\cX$, its (non-stable) Hochschild homology can be computed as a geometric realization.
Specifically, the canonical morphism is an equivalence:
\[
\colim\left(
\bDelta^{\op}
\xra{\chi^\op}
\Quiv^{\op}
\xra{\Rep_\cC}
\cX
\right)
\xra{~\simeq~}
\sHH(\cC)
~.
\]
\end{observation}

\begin{observation}
\label{t251}
Let $\cC$ be a category-object in $\cX$.
\begin{enumerate}
\item
The value of the functor $\bDelta \xra{\chi} \Quiv$ on $[0]$ is the cyclically-directed quiver $\chi([0])$ with a single object.  
This quiver $\chi([0])$ corepresents endomorphisms in $\cC$:
\[
\Rep_\cC \left( \chi\left( \left[0\right] \right) \right) 
~\simeq~
\End_\cC
~.
\]
Consequently, there is a canonical morphism in $\cX$:
\begin{equation}
\label{g22}
\End_\cC
\longrightarrow
\sHH(\cC)
~.
\end{equation}

\item
The quiver $\ast$ with a single object and no non-identity morphisms is a final object in $\Quiv$.
This quiver $\ast$ corepresents objects in $\cC$:
\[
\Rep_\cC(\ast)
~\simeq~
\Obj(\cC)
~.
\]
As so, the unique morphism $\chi\left( \left[0 \right] \right) \xra{!} \ast$ in $\Quiv$ corepresents the morphism
\[
\Obj(\cC)
\xra{~\text{``$c\mapsto \id_c$''}~}
\End_\cC
\]
that selects identity endomorphisms.  
Consequently, there is a canonical composite morphism in $\cX$:
\[
\Obj(\cC)
\xra{~\text{``$c\mapsto \id_c$''}~}
\End_\cC
\xra{~\Cref{g22}~}
\sHH(\cC)
~.
\]

\end{enumerate}

\end{observation}

\subsection{Connected quivers}

\begin{notation}
\label{d9}
The full $\infty$-subcategory 
\[
\Quiv^{\sf con}
~\subset~ 
\Quiv
\]
consists of those quivers that are connected (ie, those finite directed graphs $\Gamma$ whose geometric realization $|\Gamma|$ is connected).

\end{notation}

\begin{prop}
\label{t45}
The full $\infty$-subcategory $\Quiv^{\sf con} \subset \Quiv$ freely generates $\Quiv$ via finite categorical coproducts.
More precisely, the following assertions are true.
\begin{enumerate}
\item
The $\infty$-category $\Quiv$ admits finite coproducts.

\item
Let $\Gamma$ be a quiver.
There is a finite set $A$ and an $A$-indexed sequence of connected quivers $(\Gamma_\alpha)_{\alpha \in A}$ together with an equivalence in $\Quiv$:
\[
\underset{\alpha \in A}
\coprod
\Gamma_\alpha
\xra{~\simeq~}
\Gamma
~.
\]

\item
Let $\Gamma$ and $\Gamma'$ be quivers.
Let $\Xi$ be connected quivers.
The canonical map between spaces
\[
\Hom_{\Quiv}(\Xi,\Gamma)
\coprod
\Hom_{\Quiv}(\Xi, \Gamma')
\xra{~\simeq~}
\Hom_{\Quiv}(\Xi , \Gamma \amalg \Gamma' )
\]
is an equivalence.

\end{enumerate}

\end{prop}

\begin{proof}

\Cref{t25} immediately implies $\Quiv$ admits finite coproducts.

We now show that $\Quiv^{\sf con} \subset \Quiv$ generates $\Quiv$ via finite coproducts.
Let $\Gamma \in \Quiv$ be an object.
Through Observations~\ref{t20}, $\Gamma$ is the datum of a finite directed graph.
As a finite directed graph, there is a unique identification as a coproduct in $\digraphs$,
\[
\underset{\Gamma_\alpha \in \pi_0(|\Gamma|)}
\coprod
\Gamma_\alpha
\xra{~\simeq~}
\Gamma 
~,
\]
in which each $\Gamma_\alpha$ is a connected finite directed graph. 
\Cref{t25} implies the canonical morphism in $\Quiv$ is an equivalence:
\[
\underset{\Gamma_\alpha \in \pi_0(|\Gamma|)}
\coprod
\Gamma_\alpha
\xra{~\simeq~}
\Gamma 
~.
\]
This shows that $\Quiv^{\sf con} \subset \Quiv$ generates $\Quiv$ via finite categorical coproducts.

We now show that $\Quiv^{\sf con} \subset \Quiv$ {\it freely} generates $\Quiv$ via finite categorical coproducts.
For this, it remains to show that the restricted Yoneda functor
\[
\Quiv
\longrightarrow
\PShv\bigl( \Quiv^{\sf con} \bigr)
~,\qquad
\Gamma
\mapsto 
\Bigl(
\Xi
\mapsto 
\Hom_{\Quiv}(\Xi, \Gamma )
\Bigr)
~,
\]
preserves finite coproducts.  
So let $\Gamma , \Gamma'\in \Quiv$ and let $\Xi \in \Quiv^{\sf con}$.
We must show the canonical map $\Hom_{\Quiv}( \Xi , \Gamma )
\coprod
\Hom_{\Quiv}( \Xi , \Gamma' )
\to
\Hom_{\Quiv}(\Xi , \Gamma  \amalg \Gamma' )
$ is an equivalence between spaces.
Well, this canonical map canonically factors as a composition of maps:
\begin{eqnarray*}
\Hom_{\Quiv}( \Xi , \Gamma  \amalg \Gamma' )
&
\xra{~\simeq~}
&
\Hom_{\Cat_{(\infty,1)}}\bigl(  \Free(\Xi) , \Free(\Gamma \amalg \Gamma')\bigr)
\\
&
\xla{~\simeq~}
&
\Hom_{\Cat_{(\infty,1)}}\bigl(  \Free(\Xi) , \Free(\Gamma) \amalg \Free(\Gamma') \bigr)
\\
&
\xla{~\simeq~}
&
\Hom_{\Cat_{(\infty,1)}}\bigl(  \Free(\Xi) , \Free(\Gamma) \bigr)
\coprod
\Hom_{\Cat_{(\infty,1)}}\bigl( \Free(\Xi) , \Free(\Gamma') \bigr)
\\
&
\xla{~\simeq~}
&
\Hom_{\Quiv}( \Xi , \Gamma )
\coprod
\Hom_{\Quiv}( \Xi , \Gamma' )
~.
\end{eqnarray*}
The first and last maps are equivalences because $\Quiv \subset \Cat_{(\infty,1)}$ is a full $\infty$-subcategory.
The second map is an equivalence because, as \Cref{t25} implies, the functor the inclusion $\Quiv \hookrightarrow \Cat_{(\infty,1)}$ preserves finite coproducts.
It is a feature of $\Cat_{(\infty,1)}$ that the third map is an equivalence, using that any two objects in $\Free(\Xi)$ are related by a zig-zag of morphisms in $\Free(\Xi)$.

\end{proof}

\begin{cor}
\label{t60}
The category $\Quiv$ has the following features.
\begin{enumerate}
\item
Each object in $\Quiv$ is canonically identified as a finite coproduct of connected quivers.
More precisely, there is a canonical equivalence between commutative monoids:
\[
\Obj(\Quiv)
~\simeq~
\Free_{\sf Com}\Bigl(
\Obj(\Quiv^{\sf con})
\Bigr)
\underset{\rm Cor~\ref{t34}}{~\simeq~}
\Free_{\sf Com}\Bigl(
\underset{
[\Gamma] \in \pi_0 \Obj({\sf diGraph^{fin.con}})
}
\coprod
\sB\Aut_{{\sf diGraph^{fin.con}}}(\Gamma)
\Bigr)
~.
\]

\item
Let $\Gamma, \Xi\in \Quiv$ be objects.
Through the previous point, there are canonical finite sets $A$ and $B$ together with an $A$- and a $B$-indexed sequence $(\Gamma_\alpha)_{\alpha \in A}$ and $(\Xi_\beta)_{\beta \in B}$ of connected quivers together with identifications $\underset{\alpha\in A} \coprod \Gamma_\alpha \simeq \Gamma $ and $\underset{\beta \in B}\coprod \Xi_\beta \simeq \Xi$ in $\Quiv$.
Through these identifications, there is a canonical identification of the space of morphisms in $\Quiv$ from $\Gamma$ to $\Xi$:
\[
\Hom_{\Quiv}\bigl(
\Gamma 
,
\Xi
\bigr)
~\simeq~
\underset{\beta \in A} \prod
~
\underset{\alpha \in B} \coprod
~
\Hom_{\Quiv^{\sf con}}(\Gamma_\alpha , \Xi_\beta)
~.
\]

\end{enumerate}

\end{cor}

\subsection{The category $\M^{\sf con}$}
\label{sec.M.con}

Recall from Notation~\ref{d9} the full $\infty$-subcategory $\Quiv^{\sf con} \subset \Quiv$ consisting of those quivers that are connected.
Note that the functor $\w{\bLambda} \xra{\rm Obs~\ref{t6'}} \Quiv$ factors through the full $\infty$-subcategory $\Quiv^{\sf con} \subset \Quiv$.
The functor~(\ref{e37}) and \Cref{t6'} supply a composite functor:
\begin{equation}
\label{e21}
\copara
\longrightarrow
\bLambda
\xra{\rm monomorphism}
\w{\bLambda}
\underset{\rm Obs~\ref{t6'}}
{\xra{\rm monomorphism}}
\Quiv^{\sf con}
~,\qquad
\lambda
\longmapsto 
\ol{\lambda}
~.
\end{equation}

\begin{definition}
\label{d7}
The $\infty$-category $\M^{\sf con}$, as it is equipped with a functor 
\[
(\Quiv^{\sf con})^{\op} \xra{~\delta~}\M^{\sf con}
~, 
\]
is initial among all such for which 
the colimit of the composite functor $\para \xra{~(\ref{e21})~} (\Quiv^{\sf con})^{\op}  \xra{~\delta~} \M^{\sf con}$ exists.  
The \bit{oriented circle} is the colimit
\[
\SS^1
~:=~
\colim
\Bigl(
\para \xra{ (\ref{e21}) } (\Quiv^{\sf con})^{\op} \xra{ \delta } \M^{\sf con}
\Bigr)
~{}~
\underset{\rm Notation}{~\simeq~}
~{}~
\underset{\mu^{\circ} \in \para}
\colim~
\ol{\mu}
~\in~
\M^{\sf con}
~.
\]

\end{definition}

\begin{remark}\label{r2}
After \Cref{t36}, Definition~\ref{d7} grants the existence of a colimit of the composite functor
\[
\bDelta^{\op}
\xra{~(\ref{e15})~}
\para 
\xra{(\ref{e21})} 
(\Quiv^{\sf con})^{\op}
\xra{~\delta~} 
\M^{\sf con}
~.
\]

\end{remark}

Definition~\ref{d7} immediately yields the following.
\begin{observation}\label{t15}
Let $(\Quiv^{\sf con})^{\op} \xra{F} \cX$ be a functor to an $\infty$-category with geometric realizations.  
\begin{enumerate}
\item
There is a canonical extension among $\infty$-categories, initial among all such:
\[
\begin{tikzcd}[column sep=2cm]
(\Quiv^{\sf con})^{\op}
\arrow{r}{\forall ~ F}
\arrow{d}[swap]{\delta}
&
\cX
\\
\M^{\sf con}
\arrow[dashed, bend right=5]{ru}[sloped, swap]{\exists \text{ initial } \w{F}}
\end{tikzcd}
~.
\]

\item
This functor $\w{F}$ has the property that it preserves the colimit of $\para \xra{(\ref{e21})} (\Quiv^{\sf con})^{\op}$, which is to say the canonical morphism in $\cX$,
\[
\colim
\Bigl(
\para \xra{(\ref{e21})} (\Quiv^{\sf con})^{\op} \xra{F} \cX
\Bigr)
\xra{~\simeq~}
\w{F}
\Bigl(
\colim\bigl( 
\para \xra{ (\ref{e21}) } (\Quiv^{\sf con})^{\op} \xra{ \delta } \M^{\sf con}
\bigr)
\Bigr)
\]
is an equivalence.

\end{enumerate}

\end{observation}

Definition~\ref{d7} lends the following.
\begin{observation}
\label{t40}
\begin{enumerate}
\item[]

\item
The restricted Yoneda functor 
\begin{equation}
\label{e63}
\M^{\sf con}
\xra{~M\mapsto ~\Hom_{\M^{\sf con}}\bigl(\delta(-),M\bigr)~} \PShv((\Quiv^{\sf con})^{\op})
\end{equation}
is the functor of \Cref{t15} applied to $(\Quiv^{\sf con})^{\op} \xra{\Yo} \PShv((\Quiv^{\sf con})^{\op})$.

\item
In particular, this restricted Yoneda functor~(\ref{e63}) preserves the colimit of the functor $\para \xra{ (\ref{e21}) } (\Quiv^{\sf con})^{\op} \xra{\Yo} \PShv((\Quiv^{\sf con})^{\op})$.

\item
Furthermore, this functor~(\ref{e63}) is fully faithful;
its image is the smallest full $\infty$-subcategory that contains $(\Quiv^{\sf con})^{\op} \underset{\Yo}\subset \PShv((\Quiv^{\sf con})^{\op})$ and that contains the colimit of $\para \xra{(\ref{e21})} (\Quiv^{\sf con})^{\op} \xra{\Yo} \PShv((\Quiv^{\sf con})^{\op})$.

\item
In particular, the canonical functor $(\Quiv^{\sf con})^{\op} \xra{\delta} \M^{\sf con}$ is fully faithful. 

\end{enumerate}

\end{observation}

Each cyclicly directed graph is, in particular, connected.
Therefore, the functor $\w{\bLambda} \xra{\rm Obs~\ref{t6'}} \Quiv$ factors through the full $\infty$-subcategory $\Quiv^{\sf con} \subset \Quiv$.
Recall the \Cref{d4} of the continuous monoid $\WW$, and the functor $\w{\bLambda} \to \BW$ from \Cref{t95}(4).
\begin{observation}
\label{t39}
After \Cref{t119}, 
Definition~\ref{d7} immediately grants the existence of a left Kan extension:
\begin{equation}
\label{e66}
\begin{tikzcd}[column sep=1.5cm]
\epicyc
\arrow{d}[swap]{{\sf loc}}
\arrow{r}{{\rm Obs~\ref{t6'}}}[swap, xshift=0cm, yshift=-0.3cm]{\rotatebox{45}{$\Leftarrow$}}
&
(\Quiv^{\sf con})^{\op}
\arrow{r}{\delta}
&
\M^{\sf con}
\\
\BW^{\op}
\arrow[dashed, bend right=5]{rru}
\end{tikzcd}
\end{equation}
\end{observation}

\begin{definition}
\label{d200}
For $\cA \la \cX \ra \cB$ a span among $\infty$-categories, their \bit{parametrized join} is the $\infty$-category:
\[
\cA
\underset{\cX}
\bigstar
\cB
~:=~
\cA \underset{\cX\times\{s\}} \coprod \cX \times c_1 \underset{\cX \times \{t\}} \coprod \cB
~.
\]

\end{definition}

\begin{observation}
\label{t201}
Let $\cA \la \cX \ra \cB$ be a span among $\infty$-categories.
Let $\cZ$ be an $\infty$-category.  
Unwinding the universal property of pushouts defining the parametrized join 
$\cA
\underset{\cX}
\bigstar
\cB$, 
the data of a functor 
\[
\cA
\underset{\cX}
\bigstar
\cB
\longrightarrow
\cZ
\]
is the data of a lax-commutative diagram:
\[
\begin{tikzcd}
\cX
\arrow{r}[swap, yshift=-0.4cm]{\rotatebox{45}{$\Leftarrow$}}
\arrow{d}
&
\cA
\arrow{d}
\\
\cB
\arrow{r}
&
\cZ
\end{tikzcd}
~.
\]
\end{observation}

\begin{observation}
\label{r9}
Through \Cref{t201}, the lax-commutative diagram~\Cref{e66} is precisely the data of a functor between $\infty$-categories:
\begin{equation}
\label{f51}
(\Quiv^{\sf con})^{\op}
\underset{\epicyc}
\bigstar
\BW^{\op}
\longrightarrow
\M^{\sf con}
~.\footnote{
Equivalently, this is a functor
\[
\BW \underset{\w{\bLambda}} \bigstar \Quiv^{\sf con}
\longrightarrow
(\M^{\sf con})^{\op}
~.
\]
}
\end{equation}

\end{observation}

\begin{prop}
\label{t51}
The functor between $\infty$-categories~\Cref{f51} is an equivalence.

\end{prop}

\begin{lemma}
\label{t200}
Let $\cA \la \cX \ra \cB$ be a span among $\infty$-categories.
The $\infty$-category 
$
\cA
\underset{\cX}
\bigstar
\cB
$
is characterized by the following properties:
\begin{enumerate}
\item
There are fully faithful functors,
\[
\cA
~\hookrightarrow~ 
\cA
\underset{\cX}
\bigstar
\cB
~\hookleftarrow~
\cB
~,
\]
which are jointly surjective on objects.

\item
Let $a\in \cA$ and $b\in \cB$, regarded as objects in $\cA
\underset{\cX}
\bigstar
\cB$.
The space of morphisms
\[
\Hom_{\cA
\underset{\cX}
\bigstar
\cB}
(b,a)
~=~
\emptyset
\]
while the space of morphisms
\[
\Hom_{\cA
\underset{\cX}
\bigstar
\cB}
(a,b)
~=~
\Bigl|
\cX^{a/}_{/b}
~:=~
\cA^{a/}
\underset{\cA}
\times
\cX
\underset{\cB}
\times
\cB_{/b}
\Bigr|
\]
is the $\infty$-groupoid-completion of the over-under $\infty$-category, with the evident $(\cA,\cB)$-bimodule structure.

\item
If, for each $b\in \cB$, the functor $\cX_{/b} \to \cA$ is a Cartesian fibration, then
\[
\Hom_{\cA
\underset{\cX}
\bigstar
\cB}
(a,b)
~=~
\Bigl|
\cX^{|a}_{/b}
~:=~
\{a\}
\underset{\cA}
\times
\cX
\underset{\cB}
\times
\cB_{/b}
\Bigr|
~.
\]
If, for each $a\in \cA$, the functor $\cX^{a/} \to \cB$ is a coCartesian fibration, then
\[
\Hom_{\cA
\underset{\cX}
\bigstar
\cB}
(a,b)
~=~
\Bigl|
\cX^{a/}_{|b}
~:=~
\{a\}
\underset{\cA}
\times
\cX
\underset{\cB}
\times
\{b\}
\Bigr|
~.
\]
If $\cA \la \cX \ra \cB$ is a bifibration, then
\[
\Hom_{\cA
\underset{\cX}
\bigstar
\cB}
(a,b)
~=~
\Bigl|
\cX^{|a}_{|b}
~:=~
\{a\}
\underset{\cA}
\times
\cX
\underset{\cB}
\times
\{b\}
\Bigr|
~.
\]

\end{enumerate}

\end{lemma}

\begin{proof}
Statement~(3) follows from~(2), using that the respective fully faithful functors $\cX^{|a}_{/b} \hookrightarrow \cX^{a/}_{/b}$ and $\cX^{a/}_{|b} \hookrightarrow \cX^{a/}_{/b}$ are adjoints if the functors $\cX \to \cA$ and $\cX \to \cB$ respectively satisfy the named conditions.

Consider the Cartesian fibration
\[
{\sf Cylr}( \cA \la \cX)
~:=~
\cA
\underset{\cX \times \{0\}} \coprod
\cX \times \{0<1\}
\xra{~\pr~}
\{0<1\}
\]
that is the unstraightening of the functor
$\{0<1\}^{\op} \xra{ \langle \cA \la \cX \rangle } \Cat$.
Consider the coCartesian fibration
\[
{\sf Cyl}(\cX \to \cB)
~:=~
\cX \times \{1<2\}
\underset{\cX \times \{2\}} \coprod \cB
\xra{~\pr~}
c_1
=
\{1<2\}
\]
that is the unstraightening of the functor
$\{1<2\} \xra{ \langle \cX \to \cB \rangle } \Cat$.
Consider the functor between pushouts:
\[
\cE
~:=~
{\sf Cylr}( \cA \la \cX)
\underset{\cX \times \{1\}}
\coprod
{\sf Cyl}(\cX \to \cB)
\longrightarrow
\{0<1\}
\underset{\{1\}}
\coprod
\{1<2\}
~=~
[2]
~.
\]
By definition of exponentiable fibrations (see~\cite[Sect. 5.3]{AFR2}), 
this functor $\cE \to [2]$ is an exponentiable fibration. 
By construction, there is a canonical functor filling a diagram among $\infty$-categories,
\[
\begin{tikzcd}
\cA
\coprod
\cB
\arrow{r}
\arrow{d}
&
\cA
\underset{\cX}
\bigstar
\cB
\arrow[dashed]{r}
\arrow{d}
&
\cE
\arrow{d}
\\
\{s\} \coprod \{t\}
\arrow{r}
&
c_1
\arrow{r}[swap]{\brax{0<2}}
&
{[2]}
\end{tikzcd}
~,
\]
in which both squares are pullbacks. 
Statement~(1) follows from the left square being pullback.  
Statement~(2) follows from Lemma 5.16 of~\cite{AFR2}.  
\end{proof}

\begin{proof}[Proof of Proposition~\ref{t51}]

We use the description of parametrized joins given in Lemma~\ref{t200}.
Namely, we show that the canonical functors $(\Quiv^{\sf con})^{\op} \to \M^{\sf con}$ and $\BW^{\op} \to \M^{\sf con}$ are fully faithful, then identify the spaces of morphisms in $\M^{\sf con}$ between objects in $(\Quiv^{\sf con})^{\op}$ and $\BW^{\op}$.  
\Cref{t40} implies the canonical functor $(\Quiv^{\sf con})^{\op} \to \M^{\sf con}$ is fully faithful.  
We next show the functor $\BW^{\op} \to \M^{\sf con}$ of \Cref{t39} is fully faithful.  
Through Observations~\ref{t39} and~\ref{t40}, this functor is the left Kan extension along the localization of \Cref{t119}:
\[
\begin{tikzcd}[column sep=1.5cm]
\epicyc
\arrow{d}[swap]{{\sf loc}}
\arrow{r}{{\rm Obs~\ref{t6'}}}[swap, xshift=0cm, yshift=-0.3cm]{\rotatebox{45}{$\Leftarrow$}}
&
(\Quiv^{\sf con})^{\op}
\arrow{r}{\Yo}
&
\PShv((\Quiv^{\sf con})^\op)
\\
\BW^{\op}
\arrow[dashed, bend right=5]{rru}
\end{tikzcd}
~.
\]
By the universal property of left Kan extensions, the above lax-commutative diagram factors as the lax-commutative diagram
\begin{equation}
\label{e56}
\begin{tikzcd}[column sep=2.5cm, row sep=1.5cm]
\epicyc
\arrow{d}[swap]{{\sf loc}}
\arrow{r}{\Yo}[swap, xshift=0cm, yshift=-0.8cm]{\rotatebox{90}{$\Leftarrow$}}
&
\PShv(\epicyc)
\arrow{r}{{\PShv({\rm Obs~\ref{t6'})}}}
&
\PShv((\Quiv^{\sf con})^\op)
\\
\BW^\op
\arrow{r}[swap]{\Yo}
&
\PShv(\BW^\op)
\arrow{u}[swap]{\text{restriction}}
\end{tikzcd}
~.
\end{equation}
Now, the Yoneda lemma gives that the two functors labeled as so are fully faithful.  
The upward functor is fully faithful because it is restriction along the localization of \Cref{t119}
The top right horizontal functor is fully faithful because it is induced by the functor $\epicyc \to (\Quiv^{\sf con})^{\op}$ that \Cref{t6'} states is fully faithful.
We conclude that the composite functor $\BW^{\op} \to \PShv((\Quiv^{\sf con})^{\op})$ from the bottom left term to the top right term in~(\ref{e56}) is fully faithful, as desired.

Now, let $\Gamma \in \Quiv^{\sf con}$ be a connected quiver.
Recall from Definition~\ref{d7} the object $\SS^1 \in \BW^{\op}\subset \M^{\sf con}$, which represents the unique equivalence class of an object in $\BW^{\op}$.  
We now show $\Hom_{\M^{\sf con}}(\SS^1 , \Gamma) = \emptyset$ for all $\Gamma\in \Quiv^{\sf con}$.
By Definition~\ref{d7},
$\SS^1 := \colim(\para \to \cyclic \to \epicyc \subset (\Quiv^{\sf con})^{\op} \subset \M^{\sf con})$.
So $\Hom_{\M^{\sf con}}(\SS^1 , \Gamma)$ is the space of extensions as in the diagram among $\infty$-categories:
\[
\begin{tikzcd}
\para \amalg \{\Gamma\}
\arrow{r}
\arrow{d}
&
\cyclic \amalg (\Quiv^{\sf con})^{\op}
\arrow{r}
&
\epicyc \amalg (\Quiv^{\sf con})^{\op}
\arrow{r}
&
(\Quiv^{\sf con})^{\op}
\\
(\para)^{\rcone}
\arrow[dashed, bend right=5]{rrur}
\end{tikzcd}
~.
\]
It is therefore sufficient to show there are no extensions as in the diagram among $\infty$-categories:
\begin{equation}\label{e46}
\begin{tikzcd}
\copara
\arrow{r}
\arrow{d}
&
\bLambda
\arrow{r}
&
\w{\bLambda}
\arrow{r}
&
\Quiv^{\sf con}
\\
(\copara)^\lcone
\arrow[dashed]{rrru}[sloped, swap]{\nexists}
\end{tikzcd}
~.
\end{equation}
Consider the $\infty$-category $\cI := ( - \rightrightarrows +)$ consisting of two objects and two parallel non-identity morphisms.  
Consider the functor $\cI \to \copara$ selecting $[1]^{\star \ZZ} \overset{\sf top}{\underset{\sf bottom}\rightrightarrows } [1]^{\star \ZZ}$ in which ${\sf top}(i,k):= (i,1)$ and ${\sf bottom}(i,k):= (i,0)$.  
By explicit computation, observe that the equalizer of ${\sf top}$ and ${\sf bottom}$ is empty:
\[
\lim
\Bigl(
\cI
\to
\copara
\to
\bLambda
\to
\w{\bLambda}
\to
\Quiv^{\sf con}
\to
\Cat_{(\infty,1)}
\Bigr)
~=~
\emptyset
~.
\]
Recall that, for $\cC\in \Cat_{(\infty,1)}$, the space $\Hom_{\Cat_{(\infty,1)}}(\cC , \emptyset) \neq \emptyset$ is nonempty if and only if $\cC = \emptyset$ is initial.  
Using this, the existence of the canonical functor
\[
\lim
\Bigl(
\copara
\to
\bLambda
\to
\w{\bLambda}
\to
\Quiv^{\sf con}
\to
\Cat_{(\infty,1)}
\Bigr)
~\longrightarrow~
\lim
\Bigl(
\cI
\to
\copara
\to
\bLambda
\to
\w{\bLambda}
\to
\Quiv^{\sf con}
\to
\Cat_{(\infty,1)}
\Bigr)
\]
therefore implies
\[
\lim
\Bigl(
\copara
\to
\bLambda
\to
\w{\bLambda}
\to
\Quiv^{\sf con}
\to
\Cat_{(\infty,1)}
\Bigr)
~=~
\emptyset
~.
\]
Using this recollection again, in turn, implies there are no extensions as in diagram~\Cref{e46}.

We now show that the canonical map between spaces,
\begin{equation}
\label{k3}
\Hom_{
(\Quiv^{\sf con})^{\op}
\underset{\w{\bLambda}^{\op}} \bigstar 
\BW^{\op}
}
\bigl(
\Gamma
,
\SS^1
\bigr)
\longrightarrow
\Hom_{\M^{\sf con}}(\Gamma , \SS^1)
~,
\end{equation}
is an equivalence.
We now explain that this map canonically factors as a composite equivalence:
\begin{align}
\label{e48}
\Hom_{
(\Quiv^{\sf con})^{\op}
\underset{\w{\bLambda}^{\op}} \bigstar 
\BW^{\op}
}
\bigl(
\Gamma
,
\SS^1
\bigr)
\xla{~\simeq~}
&
\bigl |
(\epicyc)_{/\SS^1}^{\Gamma/}
\bigr |
\\
\label{e49}
\xla{~\simeq~}
&
\bigl |
({\epicyc}_{|\SS^1})^{\Gamma/}
\bigr |
\\
\label{e50}
\xla{~\simeq~}
&
\bigl |
( \para )^{\Gamma/}
\bigr |
\\
\label{e51}
\xra{~\simeq~}
&
\Bigl(
( \para \to (\Quiv^{\sf con})^{\op} )^{\widehat{}}_{\sf r.fib}
\Bigr)_{|\Gamma}
\\
\label{e52}
\xla{~\simeq~}
&
\Hom_{{\sf RFib}_{(\Quiv^{\sf con})^{\op}}}\Bigl(
((\Quiv^{\sf con})^{\op})_{/\Gamma}
,
( \para \to (\Quiv^{\sf con})^{\op} )^{\widehat{}}_{\sf r.fib}
\Bigr)
\\
\label{e53}
\xra{~\simeq~}
&
\Hom_{\PShv((\Quiv^{\sf con})^{\op})}\Bigl(
\Gamma
,
{\sf St}\bigl(
( \para \to (\Quiv^{\sf con})^{\op} )^{\widehat{}}_{\sf r.fib}
\bigr)
\Bigr)
\\
\label{e54}
\xra{~\simeq~}
&
\Hom_{\PShv((\Quiv^{\sf con})^{\op})}\Bigl(
\Gamma
,
\SS^1
\Bigr)
\\
\label{e55}
\xla{~\simeq~}
&
\Hom_{\M^{\sf con}}\Bigl(
\Gamma
,
\SS^1
\Bigr)
~.
\end{align}
The equivalence~(\ref{e48}) is the characterization of hom-spaces in parametrized joins (\Cref{t200}(2)).
The equivalence~(\ref{e49}) is the reduction of hom-spaces of parametrized joins (\Cref{t200}(3)) facilitated by \Cref{t119} which implies the functor $\epicyc \to \BW^{\op}$ is a coCartesian fibration. 
The equivalence~(\ref{e50}) follows from the outer square of \Cref{t119} being a pullback.
The $\infty$-category $( \para \to (\Quiv^{\sf con})^{\op} )^{\widehat{}}_{\sf r.fib}$ is the domain of the right-fibrationing of the functor $\para \to (\Quiv^{\sf con})^{\op}$.  
It is a localization of the Cartesian-fibrationing of the functor $\para \to (\Quiv^{\sf con})^{\op}$, which is the composite functor 
\[
\Ar\left( (\Quiv^{\sf con})^{\op} \right)\underset{(\Quiv^{\sf con})^{\op}} \times \para
\xra{~\pr~}
\Ar\left( (\Quiv^{\sf con})^{\op} \right)
\xra{~\ev_s~}
(\Quiv^{\sf con})^{\op}
~.\footnote{The implicit functor in this pullback is $\Ar\left( (\Quiv^{\sf con})^{\op} \right) \xra{\ev_t} (\Quiv^{\sf con})^{\op}$, target evaluation.}
\] 
The fiber of this Cartesion-fibrationing over $\Gamma \in (\Quiv^{\sf con})^{\op}$ therefore the $\infty$-undercategory $(\para)^{\Gamma/}$.
Proposition~3.25 of~\cite{fibrations} states that the fiber over $\Gamma \in (\Quiv^{\sf con})^{\op}$ of this right-fibrationing is the $\infty$-groupoid-completion of this $\infty$-undercategory $(\para)^{\Gamma/}$.
The right-fibrationing of $\ast \xra{\langle \Gamma \rangle} (\Quiv^{\sf con})^{\op}$ is $((\Quiv^{\sf con})^{\op})_{/\Gamma} \xra{\rm forget} (\Quiv^{\sf con})^{\op}$.
The equivalence~(\ref{e51}) then follows from the universal property of right-fibrationing.
The equivalence~(\ref{e52}) is a consequence of the fact that the Straightening construction is fully faithful.
The equivalence~(\ref{e53}) is an instance of the fact that the straightening of a right-fibrationing is identical with the colimit:
\[
{\sf St}\Bigl(
\cK^{\widehat{}}_{\sf r.fib}
\to 
\cC
\Bigr)
~\simeq~
\colim\bigl( \cK \to \cC \xra{\Yo} \PShv(\cC) \bigr) 
~.
\]
The equivalence~(\ref{e54}) is the definition of the object 
$
\SS^1\in \M^{\sf con} 
\underset{\rm Obs~\ref{t40}(3)}
\subset \PShv\bigl(\Quiv^{\sf con})^{\op} \bigr)
$ 
from Definition~\ref{d7}.
The equivalence~(\ref{e55}) follows from \Cref{t40}(3).
Finally, the sequence of maps indeed factors the map \Cref{k3} because, from the definition of $\SS^1$ as a colimit in $\M^{\sf con} \subset \PShv\left( \left(\Quiv^{\sf con}\right)^{\op}\right) \simeq {\sf RFib}_{(\Quiv^{\sf con})^{\op}}$, the diagram among spaces
{\Small
\[
\xymatrix{
\bigl |
( \para )^{\Gamma/}
\bigr |
\ar[d]_-{\Cref{e48}-\Cref{e50}}
\ar[r]^-{\Cref{e51}}
&
\Bigl(
( \para \to (\Quiv^{\sf con})^{\op} )^{\widehat{}}_{\sf r.fib}
\Bigr)_{|\Gamma}
\ar[r]^-{\Cref{e52}}
&
\Hom_{{\sf RFib}_{(\Quiv^{\sf con})^{\op}}}\Bigl(
((\Quiv^{\sf con})^{\op})_{/\Gamma}
,
( \para \to (\Quiv^{\sf con})^{\op} )^{\widehat{}}_{\sf r.fib}
\Bigr)
\ar[d]^-{\Cref{e53}-\Cref{e55}}
\\
\Hom_{
(\Quiv^{\sf con})^{\op}
\underset{\w{\bLambda}^{\op}} \bigstar 
\BW^{\op}
}
\bigl(
\Gamma
,
\SS^1
\bigr)
\ar[rr]^-{\Cref{k3}}
&&
\Hom_{\M^{\sf con}}(\Gamma , \SS^1)
}
\]
}
commutes.
This completes the proof of the proposition.

\end{proof}

\begin{cor}
\label{t51'}
The $\infty$-category $\M^{\sf con}$ has the following explicit description.
\begin{enumerate}

\item
$\M^{\sf con}$ is a $(2,1)$-category.

\item
Each object in $\M^{\sf con}$ is either a connected finite directed graph or equivalent with $\SS^1$.

\item
For $M,N \in \M^{\sf con}$, the space of morphisms in $\M^{\sf con}$ is canonically identified as
\[
\Hom_{\M^{\sf con}}(M, N)
~\simeq~
\begin{cases}
M^{(0)}
\coprod
\bigl( \sZ^{\sf dir}(M) \setminus M^{(0)} \bigr)
\times
\WW
&
,
\text{ if }
M \in \Quiv^{\sf con}
~\&~
N \simeq \SS^1
\\
\Hom_{\Quiv}(N,M)
&
,
\text{ if }
M,N \in \Quiv^{\sf con}
\\
\WW
&
,
\text{ if }
M , N \simeq \SS^1
\\
\emptyset
&
,
\text{ if }
M \simeq \SS^1
~\&~
N \in \Quiv^{\sf con}
\end{cases}
~.
\]
where, if $M$ is a finite directed graph, $M^{(0)}$ is the set of vertices of $M$ and $\sZ^{\sf dir}(M)$ is the set of directed cycles in $M$.

\item
Through these identifications, the evident $\WW^{\op} = \End_{\M^{\sf con}}(\SS^1)$-module structure agrees with that on $\Hom_{\M^{\sf con}}(-,\SS^1)$ given by post-composition.

\end{enumerate}

\end{cor}

\begin{proof}

To prove statement~(1) is to prove that, for each pair $M,N\in \M^{\sf con}$ of objects, the space $\Hom_{\M^{\sf con}}(M,N)$ is a 1-type.  
Through statements~(2) \&(3), this is to show that each of named spaces of morphisms is a 1-type.
Because the circle $\TT$ is a 1-type, the continuous monoid $\WW^{\op} \underset{\rm Obs~\ref{t96}}{\simeq} \NN^\times \ltimes \TT$ is a continuous monoid-object in 1-types.
For each pair $\Gamma,\Xi$ of finite directed graphs
the spaces $\Gamma^{(0)}$ and $\sZ^{\sf dir}(\Gamma)$ are 0-types, and \Cref{t26}(1) implies $\Hom_{\Quiv}(\Gamma,\Xi)$ is a 0-type.
In summary, statement~(1) follows from statements~(2)\&(3).
It remains to establish statements~(2)-(4).

\Cref{t51} implies the functor 
$
(\Quiv^{\sf con})^{\op}
\underset{\epicyc}
\bigstar
\BW^{\op}
\to
\M^{\sf con}
$
is surjective on spaces of objects.
This implies statement~(2).

\Cref{t51} implies the functors $(\Quiv^{\sf con})^{\op} \to \M^{\sf con}$ and $\BW^{\op} \to \M^{\sf con}$ are fully faithful.
This establishes the middle two identifications of hom spaces in statement~(3).
\Cref{t51} also implies the last identification of the hom space in statement~(3).
It remains to establish the first identification of the hom space in statement~(3).
So let $\Gamma$ be a connected quiver.
Statement~(3) is therefore proved upon establishing the following sequence of equivalences among spaces:
\begin{eqnarray*}
\Gamma^{(0)}
\coprod
\bigl( \sZ^{\sf dir}(\Gamma) \setminus \Gamma^{(0)} \bigr)
\times
\WW
&
\underset{\rm Cor~\ref{t52}}{~\simeq~}
&
\bigl|
\copara_{/\Gamma}
\bigr|
\\
&
\underset{\rm Cor~\ref{t119}}{~\simeq~}
&
\bigl|
\bigl( \w{\bLambda}_{/\Gamma} \bigr)_{|\SS^1}
\bigr|
\\
&
\underset{\rm Cor~\ref{t119} ~\&~ Lem~\ref{t200}(3)}{~\simeq~}
&
\bigl|
\w{\bLambda}^{\SS^1/}_{/\Gamma} 
\bigr|
\\
&
~\simeq~
&
\Hom_{
\BW
\underset{\w{\bLambda}} \bigstar
\Quiv^{\sf con}
}
\bigl(
\SS^1 , \Gamma
\bigr)
\\
&
\underset{\rm (-)^{\op}}{~\simeq~}
&
\Hom_{
(\Quiv^{\sf con})^{\op}
\underset{\w{\bLambda}^{\op}} \bigstar
\BW^{\op}
}
\bigl(
 \Gamma
 ,
 \SS^1
\bigr)
\\
&
\underset{\rm Prop~\ref{t51}}{~\simeq~}
&
\Hom_{
\M^{\op}
}
\bigl(
 \Gamma
 ,
 \SS^1
\bigr)
\end{eqnarray*}
The first equivalence is \Cref{t52}.
The second equivalence follows from the base-change statement of \Cref{t119}.
The third equivalence is a consequence of Lemma~\ref{t200}, using the Cartesian fibration statement of \Cref{t119}.
The unlabeled equivalence follows from the characterizing definition of spaces of morphisms in a parametrized join.  
The penultimate equivalence follows from the definitional fact that, for $\cK$ an $\infty$-category, and for $x,y\in \cK$ objects, then $\Hom_{\cK}(x,y) \simeq \Hom_{\cK^{\op}}( y,x)$.
The last equivalence follows from the fully faithfulness of \Cref{t51}.

Statement~(4) follows from the fact that the identification of \Cref{t52} is as a $\WW^{\op}$-module.

\end{proof}

\subsection{The category $\M$}
\label{sec.M}

Just as the category $\Quiv$ can be constructed from the category $\Quiv^{\sf con}$ by freely adjoining finite coproducts (Proposition~\ref{t45}), here we freely adjoint finite products to the $\infty$-category $\M^{\sf con}$ of Definition~\ref{d7}.

\begin{definition}
\label{d7.2}
The $\infty$-category $\M$, as it is equipped with a functor $\Quiv^{\op} \xra{\delta}\M$, is initial among all such for which 
\begin{enumerate}
\item
the colimit of the composite functor $\para \xra{~(\ref{e21})~} \Quiv^{\op} \xra{~\delta~} \M$ exists;

\item
$\M$ admits finite products, and the functor $\delta$ preserves finite products.

\end{enumerate}
The \bit{smooth oriented circle} is the colimit
\[
\SS^1
~:=~
\colim
\Bigl(
\para \xra{ (\ref{e21}) } \Quiv^{\op} \xra{ \delta } \M
\Bigr)
~{}~
\underset{\rm Notation}{~\simeq~}
~{}~
\underset{\mu^{\circ} \in \para}
\colim
~
\ol{\mu}
~\in~
\M
~.
\]
The \bit{$0$-disk} and the \bit{$1$-disk} are, respectively, the objects
\[
\DD^0
~:=~
\delta\bigl(
\rho(0)
\bigr)
~\in~
\M
\qquad
\text{ and }
\qquad
\DD^1
~:=~
\delta\bigl(
\rho(1)
\bigr)
~\in~
\M
~.
\]

\end{definition}

Definition~\ref{d7} immediately yields the following.
\begin{observation}\label{t15.2}
Let $\Quiv^{\op} \xra{F}\cX$ be a finite-product-preserving functor to an $\infty$-category with finite products and geometric realizations.  
\begin{enumerate}
\item
There is a finite-product-preserving extension among $\infty$-categories, initial among all such:
\[
\begin{tikzcd}[column sep=3.5cm]
\Quiv^\op
\arrow{r}{\forall~F ~{\text{($\times$-preserving)}}}
\arrow{d}[swap]{\delta}
&
\cX
\\
\M
\arrow[dashed, bend right=10]{ru}[sloped, swap]{\text{$\exists$ initial $\tilde{F}$ ($\times$-preserving)}}
\end{tikzcd}
~.
\]

\item
This functor $\w{F}$ has the property that it preserves the colimit of $\para \xra{(\ref{e21})} \Quiv^{\op}$, which is to say the canonical morphism in $\cX$,
\[
\colim
\Bigl(
\para \xra{(\ref{e21})} \Quiv^{\op} \xra{F} \cX
\Bigr)
\xra{~\simeq~}
\w{F}
\Bigl(
\colim\bigl( 
\para \xra{ (\ref{e21}) } \Quiv^{\op} \xra{ \delta } \M
\bigr)
\Bigr)
\]
is an equivalence.

\end{enumerate}

\end{observation}

Definition~\ref{d7} can be rephrased as follows.
\begin{observation}
\label{t40.2}
\begin{enumerate}
\item[]

\item
The restricted Yoneda functor 
\begin{equation}
\label{e63.2}
\M
\xra{~M\mapsto ~\Hom_{\M}\bigl(\delta(-),M\bigr)~} \PShv(\Quiv^{\op})
\end{equation}
is the functor of \Cref{t15.2}.

\item
In particular, this functor~(\ref{e63.2}) preserves finite products and the colimit of the diagram $\para \xra{ (\ref{e21}) } \Quiv^{\op}$.

\item
Furthermore, this functor~(\ref{e63.2}) fully faithful; with image the smallest full $\infty$-subcategory that contains $\Quiv^{\op} \underset{\Yo}\subset \PShv(\Quiv^{\op})$, that is closed under finite products, and that contains the colimit of $\para \xra{(\ref{e21})} \Quiv^{\op} \xra{\Yo} \PShv(\Quiv^{\op})$. 

\item
In particular, the canonical functor $\Quiv^{\op} \xra{\delta} \M$ is fully faithful.

\end{enumerate}

\end{observation}

\begin{notation}
In light of \Cref{t40.2}(4), we often do not distinguish in notation an object $\Gamma\in \Quiv$ and its image $\delta(\Gamma)\in \M$.
\end{notation}

\begin{observation}
\label{t49.2}
By Definition~\ref{d7}, 
there is a canonical functor 
\begin{equation}
\label{e68}
\M^{\sf con}
\longrightarrow 
\M
~,
\end{equation}
which has the following properties.
\begin{enumerate}

\item
This functor~(\ref{e68}) is initial among such functors under $(\Quiv^{\sf con})^{\op}$.

\item
This functor~(\ref{e68}) preserves the colimit of the diagram $\para \xra{(\ref{e21})} (\Quiv^{\sf con})^{\op}$:
\[
\M^{\sf con}
\underset{\rm Def~\ref{d7}} \ni
~
\SS^1
~\longmapsto~
\SS^1
~
\underset{\rm Def~\ref{d7.2}} \in 
\M
~.
\]

\item
This functor~(\ref{e68}) is fully faithful since the functor $\PShv((\Quiv^{\sf con})^{\op}) \to \PShv(\Quiv^{\op})$, given by left Kan extension along the fully faithful inclusion $(\Quiv^{\sf con})^{\op} \hookrightarrow \Quiv^{\op} \xra{\Yo} \PShv(\Quiv^{\op})$, is fully faithful.

\end{enumerate}

\end{observation}

\begin{notation}
\label{d40}
For $M ,N \in \M$, we denote their categorical product in $\M$ as $M \sqcup N\in \M$.\footnote{Warning: $M \sqcup N$ is not the coproduct of $M$ and $N$ in $\M$.}

\end{notation}

\begin{remark}
We use \Cref{d40} because we believe it averts needless confusion.  
Indeed, with that notation, the defining condition that $\Quiv^{\op} \xra{\delta} \M$ preserves finite products implies, for each $\Gamma , \Xi \in \Quiv$, there is an identification in $\M$:
\[
\delta
\left(
\Gamma \amalg \Xi
\right)
~\simeq~
\delta(\Gamma) \sqcup \delta(\Xi)
~.
\]
\end{remark}

\begin{observation}
\label{t58}
Let $\Gamma$ be a quiver.
Let $C$ be a finite set.
Since products in $\PShv(\Quiv^{\op})$ distribute over colimits in each variable, we have that the canonical morphism in $\PShv(\Quiv^{\op})$,
\[
\underset{(\mu_c)_{c\in C} \in (\para)^{\times C}}\colim
~
\Bigl(
\bigl(
\underset{c\in C} \prod
\ol{\mu}_c 
\bigr)
\times 
\Gamma
\Bigr)
\xra{~\simeq~}
(\SS^1)^{\times C}
\times 
\Gamma
~,
\]
is an equivalence.\footnote{This product $\underset{c\in C} \prod
\ol{\mu}_c \times \Gamma$ in $\PShv(\Quiv^{\op})$ is the value of the $C$-indexed coproduct $\bigl( \underset{c\in C}\coprod
\ol{\mu}_c  \bigr) \amalg \Gamma$ in $\Quiv$ by the Yoneda functor $\Quiv^{\op} \xra{\Yo} \PShv(\Quiv^{\op})$.
}
In particular, the lefthand term belongs to $\M$,\footnote{Through \Cref{d40}, this object is denoted $(\SS^1)^{\sqcup C} \sqcup \Gamma \in \M$.} 
and as so witnesses a colimit in $\M$.

\end{observation}

For the next result, let $\Xi , \Gamma \in \Quiv$.
Consider the functor 
\begin{equation}
\label{e71}
( \copara )^{\times C} 
\longrightarrow
\Quiv
~,
\qquad
( \mu_c)_{c\in C} 
\longmapsto 
\Bigl(
\underset{c\in C} \coprod
\ol{\mu}_c
\Bigr)
\amalg
\Gamma
~.
\end{equation}
With respect to this functor, consider the $\infty$-overcategory: $\left( ( \copara )^{\times C} \right)_{/\Xi}
~:=~
( \copara )^{\times C} \underset{ \Quiv} \times \Quiv_{/\Xi}$.

\begin{cor}
\label{t70}
Let $\Xi , \Gamma \in \Quiv$.
The canonical functor
\begin{equation}
\label{e75}
\left( ( \copara )^{\times C} \right)_{/\Xi}
\longrightarrow
\Hom_{\M}
\bigl(
\delta(\Xi )
, 
(\SS^1)^{\sqcup C}\sqcup \delta(\Gamma)
\bigr)
~,
\end{equation}
\[
\Bigl(
\Bigl(
\underset{c\in C} \coprod
\ol{\mu}_c
\Bigr)
\amalg
\Gamma
\xra{f}
\Xi
\Bigr)
\longmapsto
\Bigl(
\delta(\Xi) \xra{\delta(f)} 
\Bigl(
\underset{c\in C} \bigsqcup
\delta(\ol{\mu}_c)
\Bigr)
\sqcup
\delta(\Gamma)
\xra{\rm canonical}
(\SS^1)^{\sqcup C}\sqcup \delta(\Gamma)
\Bigr)
~,
\]
witnesses an $\infty$-groupoid-completion: $\bigl| \left( ( \copara )^{\times C} \right)_{/\Xi}  \bigr| \xra{\simeq} \Hom_{\M}\bigl( \delta(\Xi) , (\SS^1)^{\sqcup C} \sqcup \delta(\Gamma) \bigr)$.

\end{cor}

\begin{proof}
The functor~(\ref{e75}) canonically factors as the following sequence functors, 
\begin{eqnarray}
\nonumber
\left( ( \copara )^{\times C} \right)_{/\Xi}
&
\xra{\rm localization} 
&
\colim\Bigl(
( \para )^{\times C} 
\xra{(\ref{e71})}
\Quiv^{\op}
\xra{\Hom_{\Quiv}(-,\Xi)}
\Spaces
\Bigr)
\\
\nonumber
&
\xra{~\simeq~}
&
\Hom_{\PShv(\Quiv^{\op})}\Bigl(
\Xi 
,
\colim\bigl(
( \para )^{\times C} 
\xra{(\ref{e71})}
\Quiv^{\op}
\xra{\Yo}
\PShv(\Quiv^{\op})
\bigr)
\Bigr)
\\
\nonumber
&
\xra{~\simeq~}
&
\Hom_{\PShv(\Quiv^{\op})}\Bigl(
\Xi 
,
(\SS^1)^{\times C} \prod \Gamma
\Bigr)
\\
\nonumber
&
\xla{~\simeq~}
&
\Hom_{\M}\bigl( \delta(\Xi) , (\SS^1)^{\sqcup C} \sqcup \delta(\Gamma) \bigr)
~,
\end{eqnarray}
which we explain.
The $\infty$-category $\bigl( ( \copara )^{\times C} \bigr)_{/\Xi}$ over $( \copara )^{\times C}$ is the unstraightening of the functor $( \para )^{\times C} 
\xra{(\ref{e71})}
\Quiv^{\op}
\xra{\Hom_{\Quiv}(-,\Xi)}
\Spaces
$.
Therefore, there is a canonical functor $\bigl( ( \copara )^{\times C} \bigr)_{/\Xi} \to \colim\Bigl(
( \para )^{\times C} 
\xra{(\ref{e71})}
\Quiv^{\op}
\xra{\Hom_{\Quiv}(-,\Xi )}
\Spaces
\Bigr)$
witnessing an $\infty$-groupoid-completion.
The second map between spaces is an equivalence because the evaluation functor $\PShv(\Quiv^{\op}) \xra{\ev_{\Xi}} \Spaces$ preserves colimits.
The first statement of \Cref{t58} gives that the third map between spaces is an equivalence.
The second statement of \Cref{t58} gives that the last map between spaces is an equivalence.

\end{proof}

\begin{lemma}
\label{finality}
Let $C$ be a finite set.
Let $\Gamma$ be a quiver.
The functor
\begin{equation}
\label{e72}
(\para)^{\times C}
\longrightarrow
\Quiv^{\op}_{/(\SS^1)^{\amalg C} \amalg \Gamma}
~:=~
\Quiv^{\op}
\underset{\M}
\times
\M_{/(\SS^1)^{\sqcup C} \sqcup \Gamma}
~,\qquad
(\mu_c)_{c\in C}
\longmapsto
\Bigl(
\underset{c\in C} \bigsqcup
\ol{\mu}_c
\Bigr)
\sqcup
\Gamma
~,
\end{equation}
is final.
In particular, \Cref{t62} implies the $\infty$-category $\Quiv^{\op}_{/(\SS^1)^{\sqcup C} \sqcup \Gamma}$ is sifted.

\end{lemma}

\begin{proof}
Let $\Xi \in \Quiv$.
Let $\Xi \xra{f} (\SS^1)^{\sqcup C} \sqcup \Gamma$ be a morphism in $\M$, regarded as an object in $\Quiv^{\op}_{/(\SS^1)^{\sqcup C}\sqcup \Gamma}$.
With respect to the functor~(\ref{e72}), consider the $\infty$-undercategory $\left( ( \para )^{\times C} \right)^{f/} := ( \para )^{\times C} \underset{ \Quiv^{\op}_{/(\SS^1)^{\sqcup C}\sqcup \Gamma} } \times (\Quiv^{\op}_{/(\SS^1)^{\sqcup C}\sqcup \Gamma})^{f/}$.
Observe the commutative diagram among $\infty$-categories:
\begin{equation}
\label{e74}
\begin{tikzcd}
\left( ( \para )^{\times C} \right)^{f/}
\arrow{r}
\arrow{d}
&
\bigl( ( \para )^{\times C} \bigr)^{\Xi/}
\arrow{d}{(\ref{e75})}
\\
\ast
\arrow{r}[swap]{\langle f \rangle}
&
\Hom_{\M}
\bigl(
\Xi 
, 
(\SS^1)^{\sqcup C}\sqcup \Gamma 
\bigr)
\end{tikzcd}
~.
\end{equation}
Taking $\infty$-groupoid-completions results in a commutative diagram among $\infty$-groupoids:
\begin{equation}
\label{e73}
\begin{tikzcd}
\left| 
\bigl( ( \para )^{\times C} \bigr)^{f/}
\right|
\arrow{r}
\arrow{d}
&
\left|
\bigl( ( \para )^{\times C} \bigr)^{\Xi/}
\right|
\arrow{d}
\\
\ast
\arrow{r}[swap]{\langle f \rangle}
&
\Hom_{\M}
\bigl(
\Xi 
, 
(\SS^1)^{\sqcup C}\sqcup \Gamma 
\bigr)
\end{tikzcd}
~.
\end{equation}
Notice that the diagram~(\ref{e74}) is a pullback.
Using that the bottom right term in the diagram~(\ref{e74}) is an $\infty$-groupoid, it follows that the diagram~(\ref{e73}) is also a pullback.  
Therefore, the top left $\infty$-groupoid in~(\ref{e73}) is contractible for each 
$f \in \Hom_{\M}
\bigl(
\Xi 
, 
(\SS^1)^{\sqcup C}\sqcup \Gamma 
\bigr)$ 
if and only if the right vertical map is an equivalence.  
By Quillen's Theorem A, this is to say that the functor~(\ref{e72}) is final if and only if the right vertical map in~(\ref{e73}) is an equivalence, which \Cref{t70} ensures. 

\end{proof}

\subsection{An explicit description of $\M$}
The next results give an explicit description of the $\infty$-category $\M$.

The next result characterizes the spaces of morphisms in $\M$.
\begin{lemma}
\label{t37.2}
Let $A$, $B$, $C$, and $D$ be finite sets.
Let $(\Gamma_\alpha)_{\alpha \in A}$, and let $(\Xi_\beta)_{\beta \in B}$, be an $A$-indexed, and a $B$-indexed, sequence of connected quivers.
Consider the objects $M:=(\SS^1)^{\sqcup C} \sqcup \underset{\alpha \in A} \bigsqcup \Gamma_\alpha \in \M$ and $N:= (\SS^1)^{\sqcup D} \sqcup \underset{\beta \in B} \bigsqcup \Xi_\beta \in \M$.
There is a canonical identification of the space of morphisms in $\M$:
\[
\Hom_{\M}\bigl(
M
,
N
\bigr)
~\simeq~
\Bigl(
\Hom_{\M^{\sf con}}(\SS^1,\SS^1)^{\amalg C}
\amalg
\underset{\alpha\in A}
\coprod
\Hom_{\M^{\sf con}}(\Gamma_\alpha , \SS^1)
\Bigr)^{\times D}
\times
\underset{\beta\in B} \prod
\Bigl(
\underset{\alpha\in A}
\coprod
\Hom_{\M^{\sf con}}(\Gamma_\alpha , \Xi_\beta)
\Bigr)
~.
\]

\end{lemma}

\begin{proof}
Denote $\Gamma := \underset{\alpha \in A}\coprod
\Gamma_\alpha\in \Quiv$ and $\Xi := \underset{\beta \in B}\coprod
\Xi_\beta\in \Quiv$.
We explain the following sequence of equivalences among spaces:
\begin{align}
\label{f1}
\Hom_{
\M
}
\bigl(
M
,
N
\bigr)
~\simeq~
&
\Hom_{
\M
}
\bigl(
( \SS^1)^{\sqcup C}
\sqcup
\Gamma
,
(\SS^1)^{\sqcup D} \sqcup \Xi
\bigr)
\\
\label{f2}
\xla{~\simeq~}
&
\Hom_{
\M
}
\Bigl(
\underset{c \in C}
\bigsqcup
\Bigl(
\underset{ {\lambda_c}^{\circ} \in \para}
\colim~
\overline{\lambda}_c
\Bigr)
\sqcup
\Gamma
,
(\SS^1)^{\sqcup D} \sqcup \Xi
\Bigr)
\\
\label{f3.1}
\xra{~\simeq~}
&
\Hom_{
\M
}
\Bigl(
\underset{ (\lambda_c) \in (\para)^{\times C}}
\colim~
\Bigl(
\bigsqcup_{c\in C}
\overline{\lambda}_c
\sqcup
\Gamma
\Bigr)
,
(\SS^1)^{\sqcup D} \sqcup \Xi
\Bigr)
\\
\label{f3}
\xra{~\simeq~}
&
\underset{(\lambda_c)_{c\in C} \in (\copara)^{\times C}}
\lim
\Hom_{\M}
\Bigl(
\underset{c\in C} \bigsqcup
\overline{\lambda}_c
\sqcup
\Gamma
,
(\SS^1)^{\sqcup D} \sqcup \Xi
\Bigr)
\\
\label{f4}
\xra{~\simeq~}
&
\underset{(\lambda_c)_{c\in C} \in (\copara)^{\times C}}
\lim
\Bigl(
\Hom_{\M}
\Bigl(
\underset{c\in C} \bigsqcup
\overline{\lambda}_c
\sqcup
\Gamma
,
\SS^1
\Bigr)^{\times D}
\times
\Hom_{\M}
\Bigl(
\underset{c\in C} \bigsqcup
\overline{\lambda}_c
\sqcup
\Gamma
,
\Xi
\Bigr)
\Bigr)
\\
\label{f5}
{~\simeq~}
&
\Bigl(
\underset{(\lambda_c)_{c\in C} \in (\copara)^{\times C}}
\lim
\Hom_{\M}
\Bigl(
\underset{c\in C} \bigsqcup
\overline{\lambda}_c
\sqcup
\Gamma
,
\SS^1
\Bigr)
\Bigr)^{\times D}
\times
\underset{(\lambda_c)_{c\in C} \in (\copara)^{\times C}}
\lim
\Hom_{\M}
\Bigl(
\underset{c\in C} \bigsqcup
\overline{\lambda}_c
\sqcup
\Gamma
,
\Xi
\Bigr)
\\
\label{f6}
\xla{~\simeq~}
&
\Bigl(
\underset{(\lambda_c)_{c\in C} \in (\copara)^{\times C}}
\lim
\Hom_{\M}
\Bigl(
\underset{c\in C} \bigsqcup
\overline{\lambda}_c
\sqcup
\Gamma
,
\underset{\mu^{\circ} \in \para}
\colim~
\ol{\mu}
\Bigr)
\Bigr)^{\times D}
\times
\underset{(\lambda_c)_{c\in C} \in (\copara)^{\times C}}
\lim
\Hom_{\M}
\Bigl(
\underset{c\in C} \bigsqcup
\overline{\lambda}_c
\sqcup
\Gamma
,
\Xi
\Bigr)
\\
\label{f7}
\xla{~\simeq~}
&
\Bigl(
\underset{(\lambda_c)_{c\in C} \in (\copara)^{\times C}}
\lim~
\underset{\mu^{\circ} \in \para}
\colim~
\Hom_{\Quiv}
\Bigl(
\ol{\mu}
,
\Gamma
\amalg
\underset{c\in C} \coprod
\overline{\lambda}_c
\Bigr)
\Bigr)^{\times D}
\times
\underset{(\lambda_c)_{c\in C} \in (\copara)^{\times C}}
\lim
\Hom_{\M}
\Bigl(
\underset{c\in C} \bigsqcup
\overline{\lambda}_c
\sqcup
\Gamma
,
\Xi
\Bigr)
\\
\label{f8}
\xla{~\simeq~}
&
\Bigl(
\underset{(\lambda_c)_{c\in C} \in (\copara)^{\times C}}
\lim~
\underset{\mu^{\circ} \in \para}
\colim~
\Hom_{\Quiv}
\Bigl(
\ol{\mu}
,
\Gamma
\amalg
\underset{c\in C} \coprod
\overline{\lambda}_c
\Bigr)
\Bigr)^{\times D}
\times
\underset{(\lambda_c)_{c\in C} \in (\copara)^{\times C}}
\lim
\Hom_{\Quiv}
\Bigl(
\Xi
,
\Gamma
\amalg
\underset{c\in C} \coprod
\overline{\lambda}_c
\Bigr)
\\
\label{f9}
\xla{~\simeq~}
&
\Bigl(
\underset{(\lambda_c)_{c\in C} \in (\copara)^{\times C}}
\lim~
\underset{\mu^{\circ} \in \para}
\colim~
\Hom_{\Quiv}
\Bigl(
\ol{\mu}
,
\Gamma
\amalg
\underset{c\in C} \coprod
\overline{\lambda}_c
\Bigr)
\Bigr)^{\times D}
\times
\underset{\beta\in B}
\prod
\underset{(\lambda_c)_{c\in C} \in (\copara)^{\times C}}
\lim
\Hom_{\Quiv}
\Bigl(
\Xi_\beta
,
\Gamma
\amalg
\underset{c\in C} \coprod
\overline{\lambda}_c
\Bigr)
\\
\label{f10}
\xla{~\simeq~}
&
X
^{\times D}
\times
\underset{\beta\in B}
\prod
Y_\beta
~.
\end{align}

The equivalence~(\ref{f1}) is the definitions of $M,N\in \M$.
The equivalence~(\ref{f2}) is a direct consequence of \Cref{t49.2}(1).
The equivalence~(\ref{f3.1}) is a direct consequence of \Cref{t58}.
The equivalence~(\ref{f3}) is the universal property of colimits, which corepresent limits of spaces of morphisms.
The equivalence~(\ref{f4}) is the universal property of products, which represent products of spaces of morphisms.
The equivalence~(\ref{f5}) is the fact that limits commute with products.
The equivalence~(\ref{f6}) is the Definition~\ref{d7.2} of $\SS^1\in \M$.
The equivalence~(\ref{f7}) follows from \Cref{t40.2}, using that, for $x\in \cK$ an object in an $\infty$-category, the evaluation functor $\PShv(\cK) \xra{\ev_x} \Spaces$ preserves colimits.
The equivalence~(\ref{f8}) is a direct consequence of the defining functor $\Quiv^{\op} \xra{\delta} \M$ being fully faithful (see \Cref{t40.2}(4)).
The equivalence~(\ref{f9}) is the definition of $\Xi\in \Quiv$ as a coproduct, and the fact that limits commute with products.  
The equivalence~(\ref{f10}) is just notation, which will be explained below.

Next, we explain the following sequences of equivalences among spaces:
\begin{align}
\label{h1}
X
{~:=~}
&
\underset{(\lambda_c)_{c\in C} \in (\copara)^{\times C}}
\lim~
\underset{\mu^{\circ} \in \para}
\colim~
\Hom_{\Quiv}
\Bigl(
\ol{\mu}
,
\Gamma
\amalg
\underset{c\in C} \coprod
\overline{\lambda}_c
\Bigr)
\\
\label{h2}
\xla{~\simeq~}
&
\underset{(\lambda_c)_{c\in C} \in (\copara)^{\times C}}
\lim~
\underset{\mu^{\circ} \in \para}
\colim~
\Bigl(
\underset{c\in C} \coprod
\Hom_{\Quiv^{\sf con}}
\Bigl(
\ol{\mu}
,
\overline{\lambda}_c
\Bigr)
\amalg
\underset{\alpha \in A} \coprod
\Hom_{\Quiv^{\sf con}}
\bigl(
\ol{\mu}
,
\Gamma_\alpha
\bigr)
\Bigr)
\\
\label{h3}
{~\simeq~}
&
\underset{(\lambda_c)_{c\in C} \in (\copara)^{\times C}}
\lim~
\Bigl(
\underset{c\in C} \coprod
\Bigl(
\underset{\mu^{\circ} \in \para}
\colim~
\Hom_{\Quiv^{\sf con}}
\Bigl(
\ol{\mu}
,
\overline{\lambda}_c
\Bigr)
\Bigr)
\amalg
\underset{\alpha \in A} \coprod
\Bigl(
\underset{\mu^{\circ} \in \para}
\colim~
\Hom_{\Quiv^{\sf con}}
\bigl(
\ol{\mu}
,
\Gamma_\alpha
\bigr)
\Bigr)
\Bigr)
\\
\label{h4}
\xla{~\simeq~}
&
\underset{c\in C} \coprod
\Bigl(
\underset{\lambda_c\in \copara}
\lim~
\underset{\mu^{\circ} \in \para}
\colim~
\Hom_{\Quiv^{\sf con}}
\Bigl(
\ol{\mu}
,
\overline{\lambda}_c
\Bigr)
\Bigr)
\amalg
\underset{\alpha \in A} \coprod
\Bigl(
\underset{\mu^{\circ} \in \para}
\colim~
\Hom_{\Quiv^{\sf con}}
\bigl(
\ol{\mu}
,
\Gamma_\alpha
\bigr)
\Bigr)
\\
\label{h5}
\xra{~\simeq~}
&
\underset{c\in C} \coprod
\Bigl(
\underset{\lambda_c\in \copara}
\lim~
\Hom_{\M}
\Bigl(
\overline{\lambda}_c
,
\SS^1
\Bigr)
\Bigr)
\amalg
\underset{\alpha \in A} \coprod
\Hom_{\M}
\bigl(
\Gamma_\alpha
,
\SS^1
\bigr)
\\
\label{h6}
\xla{~\simeq~}
&
\underset{c\in C} \coprod
\Hom_{\M^{\sf con}}
\Bigl(
\underset{{\lambda_c}^{\circ} \in \para}
\colim~
\overline{\lambda}_c
,
\SS^1
\Bigr)
\amalg
\underset{\alpha \in A} \coprod
\Hom_{\M^{\sf con}}
\bigl(
\Gamma_\alpha
,
\SS^1
\bigr)
\\
\label{h7}
\xla{~\simeq~}
&
\Hom_{\M^{\sf con}}(\SS^1,\SS^1)^{\amalg C}
\amalg
\underset{\alpha \in A} \coprod
\Hom_{\M}
\bigl(
\Gamma_\alpha
,
\SS^1
\bigr)
~.
\end{align}
The equivalence~(\ref{h1}) is the definition of the space $X$.
The equivalence~(\ref{h2}) is a direct consequence of \Cref{t60}, using that the quiver $\ol{\mu}$ is connected.  
The equivalence~(\ref{h3}) is an instance of the fact that colimits commute with coproducts. 
The equivalence~(\ref{h4}) follows from the fact that the paracyclic category $\copara$ is cosifted (see \Cref{t62})).
The equivalence~(\ref{h5}) follows from \Cref{t40.2}(2), using that, for $x\in \cK$ an object in an $\infty$-category, the evaluation functor $\PShv(\cK) \xra{\ev_x} \Spaces$ preserves colimits.
The equivalence~(\ref{h6}) is the universal property of colimits, which corepresent limits of spaces of morphisms.
The equivalence~(\ref{h7}) follows from the definition of $\SS^1\in \M^{\sf con}$ (see Definition~\ref{d7}).

Next, let $\beta \in B$.
We now explain the following sequences of equivalences among spaces:
\begin{eqnarray}
\label{g1}
Y_\beta
&
~:=~
&
\underset{(\lambda_c)_{c\in C} \in (\copara)^{\times C}}
\lim
\Hom_{\Quiv}
\Bigl(
\Xi_\beta
,
\amalg
\Gamma
\underset{c\in C} \coprod
\overline{\lambda}_c
\Bigr)
\\
\label{g2}
&
\xla{~\simeq~}
&
\underset{(\lambda_c)_{c\in C} \in (\copara)^{\times C}}
\lim
\Bigl(
\Hom_{\Quiv}
\Bigl(
\Xi_\beta
,
\underset{c\in C} \coprod
\overline{\lambda}_c
\Bigr)
\amalg
\Hom_{\Quiv}
\Bigl(
\Xi_\beta
,
\Gamma
\Bigr)
\Bigr)
\\
\label{g3}
&
\xla{~\simeq~}
&
\Bigl(
\underset{(\lambda_c)_{c\in C} \in (\copara)^{\times C}}
\lim
\Hom_{\Quiv}
\Bigl(
\Xi_\beta
,
\underset{c\in C} \coprod
\overline{\lambda}_c
\Bigr)
\Bigr)
\amalg
\Hom_{\Quiv}
\Bigl(
\Xi_\beta
,
\Gamma
\Bigr)
\\
\label{g4}
&
\xra{~\simeq~}
&
\Bigl(
\underset{\lambda \in \copara}
\lim
\Hom_{\Quiv}
\Bigl(
\Xi_\beta
,
\overline{\lambda}^{\amalg C}
\Bigr)
\Bigr)
\amalg
\Hom_{\Quiv}
\Bigl(
\Xi_\beta
,
\Gamma
\Bigr)
\\
\label{g5}
&
\longrightarrow
&
\Bigl(
\underset{\lambda \in \copara}
\lim
\Hom_{\Quiv}
\Bigl(
\Xi_\beta
,
\overline{\lambda}
\Bigr)
\Bigr)
\amalg
\Hom_{\Quiv}
\Bigl(
\Xi_\beta
,
\Gamma
\Bigr)
\\
\label{g6}
&
\xla{~\simeq~}
&
\Hom_{\M^{\sf con}}
\Bigl(
\underset{\lambda^{\circ} \in \para}
\colim~
\overline{\lambda}
,
\Xi_\beta
\Bigr)
\amalg
\Hom_{\Quiv}
\Bigl(
\Xi_\beta
,
\Gamma
\Bigr)
\\
\label{g7}
&
{~\simeq~}
&
\Hom_{\M^{\sf con}}
\Bigl(
\SS^1
,
\Xi_\beta
\Bigr)
\amalg
\Hom_{\Quiv}
\Bigl(
\Xi_\beta
,
\Gamma
\Bigr)
\\
\label{g8}
&
\xra{~\simeq~}
&
\emptyset
\amalg
\Hom_{\Quiv}
\Bigl(
\Xi_\beta
,
\Gamma
\Bigr)
~=~
\Hom_{\Quiv}
\Bigl(
\Xi_\beta
,
\Gamma
\Bigr)
\\
\label{g9}
&
\xla{~\simeq~}
&
\underset{\alpha \in A}
\coprod
\Hom_{\Quiv^{\sf con}}
\Bigl(
\Xi_\beta
,
\Gamma_\alpha
\Bigr)
\\
\label{g10}
&
\xla{~\simeq~}
&
\underset{\alpha \in A}
\coprod
\Hom_{\M^{\sf con}}
\Bigl(
\Gamma_\alpha
,
\Xi_\beta
\Bigr)
~.
\end{eqnarray}
The equivalence~(\ref{g1}) is the definition of the space $Y_\beta$.
The equivalence~(\ref{g2}) is a direct consequence of \Cref{t60}, using that the quiver $\Xi_\beta$ is connected.  
The equivalence~(\ref{g3}) follows from $\copara$ being cosifted~(\Cref{t62}); the equivalence~(\ref{g4}) also follows from $\copara$ being cosifted.
The map~(\ref{g5}) is implemented by restriction along the codiagonal morphism $\ol{\lambda}^{\amalg C} \to \ol{\lambda}$ in $\Quiv$, functorially in $\lambda \in \copara$.
We postpone explaining why this map~(\ref{g5}) is an equivalence. 
The equivalence~(\ref{g6}) is the universal property of colimits, which corepresent limits of spaces of morphisms.
The equivalence~(\ref{g7}) is the definition of $\SS^1 \in \M^{\sf con}$ (see Definition~\ref{d7}).
The equivalence~(\ref{g8}) follows from Proposition~\ref{t51}.
Note that each of the maps~(\ref{g5}),~(\ref{g6}),~(\ref{g7}),~(\ref{g8}) respects the evident coproduct description.
Because the left cofactor of the codomain of~(\ref{g8}) is empty, it then follows that the left cofactor of the domain and codomain of~(\ref{g5}) are both empty as well.
In particular, the map~(\ref{g5}) is an equivalence, as desired.  
Moving on, the equivalence~(\ref{g9}) uses the definition of $\Gamma$, together with \Cref{t60}(2).
The equivalence~(\ref{g10}) is an instance of the fully faithfulness of $(\Quiv^{\sf con})^{\op} \to \M^{\sf con}$ (see \Cref{t40}(4)).

\end{proof}

After Proposition~\ref{t51},
\Cref{t37.2} gives the following.
\begin{cor}
\label{t63}
Let $C$ and $D$ be finite sets.
Let $\Gamma$ and $\Xi$ be quivers.
There is a canonical identification of the space of morphisms in $\M$:
\[
\Hom_{\M}\bigl(
(\SS^1)^{\sqcup C} 
\sqcup
\Gamma
,
(\SS^1)^{\sqcup D} 
\sqcup
\Xi
\bigr)
~\simeq~
\Bigl(
\WW^{\amalg C}
\amalg
\Gamma^{(0)}
\amalg
\TT \times \NN^\times \times
\bigl(
\sZ^{\sf dir}(\Gamma)
\!\setminus\!
\Gamma^{(0)}
\bigr)
\Bigr)^{\times D}
\times
\Hom_{\Quiv}(\Xi,\Gamma)
~,
\]
where $\Gamma^{(0)}$ is the set of vertices of $\Gamma$ and $\sZ^{\sf dir}(\Gamma) \!\setminus\! \Gamma^{(0)}$ is the set of non-constant directed cycles in $\Gamma$.

\end{cor}

\begin{lemma}
\label{t58'}
Let $M\in \M \subset \PShv(\Quiv^{\op})$.
Let $C$ be a finite set.
The canonical morphism in $\PShv(\Quiv^{\op})$,
\[
\underset{(\mu_c)_{c\in C} \in (\para)^{\times C}}\colim
~
\Bigl(
\bigl(
\underset{c\in C} \prod
\ol{\mu}_c 
\bigr)
\times 
M
\Bigr)
\xra{~\simeq~}
(\SS^1)^{\times C}
\times 
M
~,
\]
is an equivalence.
In particular, the lefthand term belongs to $\M$,\footnote{Through \Cref{d40}, this object is denoted $(\SS^1)^{\sqcup C} \sqcup \Gamma \in \M$.} 
and as so witnesses a colimit in $\M$.

\end{lemma}

\begin{proof}
Through \Cref{t63}, choose a finite set $D$ and a quiver $\Gamma$ and an equivalence $M \simeq (\SS^1)^{\sqcup D} \sqcup \Gamma$ in $\M$.
Consider the canonical diagram in $\PShv(\Quiv^{\op})$:
\[
\begin{tikzcd}
	{\underset{(\mu_c)_{c\in C} \in (\para)^{\times C}}\colim ~ \left(\left( \underset{c\in C} \prod \ol{\mu}_c \right) \times \underset{(\mu_c)_{d\in D} \in (\para)^{\times D}}\colim ~ \bigl( \underset{d\in D} \prod \ol{\mu}_d \bigr) \times  \Gamma \right)} & {\underset{(\mu_c)_{c\in C} \in (\para)^{\times C}}\colim ~ \left(\left( \underset{c\in C} \prod \ol{\mu}_c \right) \times M \right)} \\
	{\underset{(\mu_e)_{e\in C \amalg D} \in (\para)^{\times (C \amalg D)}}\colim ~ \left(  \underset{e\in C \amalg D} \prod \ol{\mu}_e \times  \Gamma \right)} \\
	{(\SS^1)^{\times (C \amalg D)} \times \Gamma} & {(\SS^1)^{\times C} \times M}
	\arrow["{{\rm (a)}}", from=1-1, to=1-2]
	\arrow[from=1-2, to=3-2]
	\arrow["{{\rm (b)}}", from=2-1, to=1-1]
	\arrow["{{\rm (c)}}"', from=2-1, to=3-1]
	\arrow["{{\rm (d)}}"', from=3-1, to=3-2]
\end{tikzcd}
~.
\]
We seek to show the unlabeled right vertical morphism is an equivalence.  
\Cref{t58}, applied to the inner colimit, gives that the morphism ${\rm (a)}$ is an equivalence.
The morphism ${\rm (b)}$ is an equivalence because the $\infty$-category $\para$ is sifted (\Cref{t62}).
\Cref{t58} gives that the morphism ${\rm (c)}$ is an equivalence.
The bottom horizontal equivalence is $(\SS^1)^{\sqcup C} \sqcup -$ applied to the identification $(\SS^1)^{\sqcup D} \sqcup \Gamma \sqcup M$.
The right vertical morphism is therefore an equivalence by the 2-of-3 property of equivalences in an $\infty$-category.

\end{proof}

The following three results are direct consequences of \Cref{t37.2}.
\begin{cor}
\label{t23}
The full $\infty$-subcategory $\M^{\sf con} \underset{\rm Obs~\ref{t49.2}(3)}\subset \M$ freely generates $\M$ via finite categorical products.
More precisely, the following assertions are true.
\begin{enumerate}
\item
The $\infty$-category $\M$ admits finite products.

\item
Let $M\in \M$ be an object.
There is a finite set $A$ and an $A$-indexed sequence $(M_\alpha)_{\alpha \in A}$ of objects in $\M^{\sf con}$ together with an equivalence in $\M$:
\[
M
\xra{~\simeq~}
\underset{\alpha \in A}
\bigsqcup
M_\alpha
~.
\]

\item
Let $M , M' \in \M$ be objects.
Let $N \in \M^{\sf con}$.
The canonical map
\[
\Hom_{\M}(M, N)
\coprod
\Hom_{\M}(M', N)
\xra{~\simeq~}
\Hom_{\M}(M \sqcup M' , N)
\]
is an equivalence between spaces.

\end{enumerate}

\end{cor}

\begin{cor}
\label{t69}
Each object in $\M$ is a finite disjoint union of oriented circles and connected quivers.\footnote{In other words, each object in $\M$ is a finite product of colimits 
$\underset{\mu^{\circ} \in \para}
\colim
~
\ol{\mu}$
and connected quivers.}  
More precisely, the moduli space of its objects is the free commutative monoid
\[
\Obj(\M)
~\simeq~
\Free_{\sf Com}\Bigl(
\Obj(\M^{\sf con})
\Bigr)
~\simeq~
\Free_{\sf Com}\Bigl(
\sB\TT
\amalg
\underset{
[\Gamma] \in \pi_0 \Obj({\sf diGraph^{fin.con}})
}
\coprod
\sB\Aut_{{\sf diGraph^{fin.con}}}(\Gamma)
\Bigr)
~.
\]

\end{cor}

After \Cref{t51'}(1), \Cref{t37.2} implies the following.
\begin{cor}
\label{t53}
The a priori $(\infty,1)$-category $\M$ is in fact a $(2,1)$-category.

\end{cor}

\subsection{Refinement morphisms in $\M$}

Here, we define and study the subcategory of refinement morphisms in $\M$.

\begin{definition}
A morphism $f: M \to N$ in $\M^{\sf con}$ is a \bit{refinement morphism} if one of the following conditions is satisfied.
\begin{enumerate}
\item
$M, N \in (\Quiv^{\sf con})^{\op}$ and $f^{\circ}$ is a refinement morphism in the sense of \Cref{d1}.
\item
$M \simeq N \simeq \SS^1$ and $f$ is an isomorphism
\item
$ M \in (\Quiv^{\sf con})^{\op}$, $N \simeq \SS^1$ and $f$ is given by (in the identification of \Cref{t51'})
\[ \sW^{\sf dir}(M) \times \TT \subset \bigl( \sZ^{\sf dir}(M) \setminus M^{(0)} \bigr)
\times
\WW ,\]
where
\[ \sW^{\sf dir}(M) \subset \sZ^{\sf dir}(M) \setminus M^{(0)} \]
is the subset of directed cycles in which every edge appears exactly once.
\end{enumerate}
A morphism in $\M$ is a \bit{refinement morphism} if it is the product of refinement morphisms in $\M^{\sf con}$.
We denote by
\[ \M^{\sf ref} \subset \M \]
the subcategory consisting of refinement morphisms.
\end{definition}

\begin{definition}
\label{d21.1}
Let $M \in \M$.
The full $\infty$-subcategory
\[
\Quiv(M)
~\subset~
\Quiv^{\op}_{/M}
~:=~
\Quiv^{\op} \underset{\M} \times \M_{/M}
\]
consists of those $\Gamma \to M$ that are refinement morphisms.

\end{definition}

\begin{observation}
\label{t71}
Let $M,N \in \M$.
\begin{enumerate}

\item
Taking products defines a functor
\[
\M_{/M}
\times
\M_{/N}
\longrightarrow
\M_{/M\sqcup N}
~,\qquad
(M'\to M , N'\to N)
\mapsto
(M' \sqcup N' \to M \sqcup N)
~.
\]

\item
This functor restricts as an equivalence:
\[
\M_{/^{\sf ref} M}
\times
\M_{/^{\sf ref} N}
\xra{~\simeq~}
\M_{/^{\sf ref} M\sqcup N}
~,
\]
where $\M_{/^{\sf ref} M} \subset \M_{/ M}$ is the full subcategory consisting of refinement maps to $M$.

\item
The above functor further restricts as an equivalence:
\[
\Quiv(M)
\times
\Quiv(N)
\xra{~\simeq~}
\Quiv(M\sqcup N)
~.
\]

\end{enumerate}

\end{observation}

\begin{observation}
\label{r11}
Let $C$ be a finite set.
Let $\Gamma \in \Quiv$.
\begin{enumerate}
\item
Reviewing the definition of refinement morphisms in $\M$ reveals that the functor~(\ref{e72}) factors:
\begin{equation}
\label{e76}
(\para)^{\times C}
\longrightarrow
\Quiv\bigl( (\SS^1)^{\sqcup C} \sqcup \Gamma \bigr)
~,\qquad
(\mu_c)_{c\in C}
\longmapsto
\Bigl(
\underset{c\in C} \bigsqcup
\ol{\mu}_c \Bigr)
\sqcup 
\Gamma
~.
\end{equation}

\item
If $\Gamma = \emptyset$, then this morphism~(\ref{e76}) is an equivalence.

\item 
If $C = \emptyset$, then this morphism~(\ref{e76}) is the inclusion of a final object.

\item
Through \Cref{t71}(4), the functor~(\ref{e76}) is a fully faithful right adjoint.  

\end{enumerate}

\end{observation}

\begin{lemma}
\label{t72}
Let $M \in \M$.
The canonical functor
\[
\Quiv( M )
\longrightarrow
\Quiv^{\op}_{/M}
\]
is final.  
In particular, using Observation~\ref{r11}(4) and \Cref{t62}, both of these $\infty$-categories $\Quiv(M)$ and $\Quiv^{\op}_{/M}$ are sifted.  
\end{lemma}

\begin{proof}
\Cref{t69} implies there is an equivalence $M \simeq (\SS^1)^{\sqcup C}\sqcup \Gamma$ in $\M$ for some finite set $C$ and some $\Gamma \in \Quiv$.
Observation~\ref{r11}(1) grants a filler among $\infty$-categories:
\[
\begin{tikzcd}[row sep=1.5cm, column sep=0.5cm]
&
\Quiv\bigl(
(\SS^1)^{\sqcup C} \sqcup \Gamma
\bigr)
\arrow{rd}
\\
( \para )^{\times C}
\arrow{rr}[swap]{(\ref{e72})}
\arrow[dashed]{ru}[sloped]{\rm Obs~\ref{r11}(1)}
&&
\Quiv^{\op}_{/ (\SS^1)^{\sqcup C} \sqcup \Gamma }
\end{tikzcd}
~.
\]
Lemma~\ref{finality} states that the bottom horizontal functor is final.
Observation~\ref{r11}(4) implies the diagonal upward arrow is final.
By the 2-out-of-3 property of final functors, the diagonal downward functor is final, as desired.

\end{proof}

\subsection{Excision sites}
The object $\SS^1 \in \M$ is defined as a colimit of a functor $\para \to \M$.
\Cref{t36} implies $\SS^1$ can be computed as a geometric realization:
\[
|
\SS^1_\bullet
|
~=~
\colim
\Bigl(
\SS^1_\bullet
\colon
\bDelta^{\op} 
\xra{\Cref{e15}}
\para \xra{ (\ref{e21}) } \Quiv^{\op} \xra{ \delta } \M
\Bigr)
~\simeq~
\SS^1
~\in~
\M
\]
where, for each $[p]\in \bDelta$, the object $\SS^1_p \in \Quiv$ is the pushout in $\digraphs$:
\[
\begin{tikzcd}[column sep=1.5cm]
\SS^0
\arrow{d}
\arrow{r}{-1}
&
\SS^0
\arrow{r}
&
\DD^1
\arrow{d}
\\
{[p]}
\arrow{rr}
&&
\SS^1_p
\end{tikzcd}
~.
\]
Here, $\SS^0 = \{\pm 1\}$ is the two-element set, regarded as a quiver with no non-degenerate edges;
the horizontal map $\SS^0 \to \DD^1$ is given by $-1\mapsto 0$ and $+1\mapsto 1$;
the left vertical map is given by $-1\mapsto 0$ and $+1 \mapsto p$.
As \Cref{t4.5} below, we show that $\M$ admits geometric realizations of slightly more general simplicial objects in $\M$.
The general statement is technical, so we introduce the following.

\begin{construction}
\label{d4.5}
Let $\w{\Gamma}$ be a finite directed graph.
Let $S$ be a finite set.
Let 
\[
(\SS^0)^{\amalg S} 
\overset{\varphi}{~\hookrightarrow~} 
\w{\Gamma}^{(0)}
\]
be an injection into the set of vertices with the property that, for each $s\in S$, the vertex $\varphi_s(-1)$ has exactly one incident edge and it is in-coming, and the vertex $\varphi_s(+1)$ has exactly one incident edge and it is out-going.  
For each $[p]\in \bDelta$ regarded as a linearly-directed graph, consider the pushout in $\digraphs$:
\[
\begin{tikzcd}
(\SS^0)^{\amalg S} 
\arrow{d}
\arrow{r}{\varphi}
&
\w{\Gamma}
\arrow{d}
\\
{[p]^{\amalg S}}
\arrow{r}
&
\Gamma_p
\end{tikzcd}
\]
in which the left vertical map is the $S$-fold coproduct of the map $\SS^0 \to [p]$ given by $-1\mapsto 0$ and $+1 \mapsto p$.
Using that the values $[\bullet] \in \Quiv$ assemble as a functor $\bDelta \to \Quiv$, the above values assemble as a functor
\[
\bDelta
\xra{~\Gamma_\bullet~}
\Quiv
~,\qquad
[p]
\longmapsto
\Gamma_p
~.
\]
For $M'\in \M$, taking disjoint union with $M'$ determines a simplicial object in $\M$:
\begin{equation}
\label{s11}
M_\bullet
\colon
\bDelta^{\op}
\xra{~\Gamma_\bullet~}
\Quiv^{\op}
\xra{~\delta~}
\M
\xra{~\sqcup M'~}
\M
~,\qquad
[p]
\longmapsto
M_p
:=
\Gamma_p \sqcup M'
~.
\end{equation}

\end{construction}

\begin{terminology}
\label{d111}
Let $M\in \M$ be an object.
An \bit{excision site (for $M$)} is the data $(\w{\Gamma},S,\varphi, M')$ of \Cref{d4.5}, together with an identification $M \simeq \left| M_\bullet \right|$ in $\M$ of the colimit of the resulting simplicial object \Cref{s11}.

\end{terminology}

\begin{figure}
\centering
\begin{subfigure}{\textwidth}
\centering
\begin{tikzpicture}
\node (A) at (1.2,0) {\textbullet};
\node (B) at (3, 1.5) {\textbullet};
\node (C) at (3, -1.5) {\textbullet};

\draw[->] (0,0) arc[start angle=360, end angle=0, radius=1.5];

\draw[->, shorten <=1.5mm] (C) arc[start angle=90, end angle=420, radius=3mm];
\draw[->] (B) to[bend right] (C);
\draw[->] (A) to (B);
\draw[->] (A) to (C);

\draw[->, shorten >= 2mm, shorten <=2.5mm] (C) arc(310:410:1.5 and 2);

\node (XX) at (6,1) {}; 
\draw[->] (XX) arc[start angle = 120, end angle = 480, radius = 1.5];

\end{tikzpicture}

\caption{An object $M\in \M$.}

\end{subfigure}

\begin{subfigure}{\textwidth}
\centering

\begin{tikzpicture}
\node (A) at (1.2,0) {\textbullet};
\node (B) at (3, 1.5) {\textbullet};
\node (C) at (3, -1.5) {\textbullet};

\draw[->] (0,0) arc[start angle=360, end angle=0, radius=1.5];

\draw[->, shorten <=1.5mm] (C) arc[start angle=90, end angle=420, radius=3mm];
\draw[->] (B) to[bend right] (C);
\draw[->] (A) to (B);
\draw[->] (A) to (C);	
		
\draw[->, shorten >= 2mm, shorten <=2mm] (C) arc(310:343:1.5 and 2) node (X) {\color{red}{\footnotesize{\textbullet}}};
	
\draw[<-, shorten >= 2mm, shorten <=2mm] (B) arc(410:380:1.5 and 2) node (Z) {\color{red}{\footnotesize{\textbullet}}};

\node (XX) at (6,1) {\color{red}\footnotesize{\textbullet}};
\draw[->, shorten >= 1.5mm, shorten <=1.5mm] (XX) arc[start angle = 120, end angle = 420, radius = 1.5] node {\color{red}\footnotesize{\textbullet}};

\end{tikzpicture}

\caption{An excision site $(\w{\Gamma},S,\varphi,M')$ for $M$, where $\w{\Gamma}$ is the directed graph that is the union of the middle and right components, 
the set $S$ is a 2-element set,
the map $\varphi$ is the inclusion of the univalent vertices (in red), 
and $M'$ is the left oriented circle.}

\end{subfigure}

\begin{subfigure}{\textwidth}
\centering

\begin{tikzpicture}
\node (A) at (1.2,0) {\textbullet};
\node (B) at (3, 1.5) {\textbullet};
\node (C) at (3, -1.5) {\textbullet};

\draw[->] (0,0) arc[start angle=360, end angle=0, radius=1.5];

\draw[->, shorten <=1.5mm] (C) arc[start angle=90, end angle=420, radius=3mm];
\draw[->] (B) to[bend right] (C);
\draw[->] (A) to (B);
\draw[->] (A) to (C);
	
\draw[->, shorten >= 2mm, shorten <=2mm] (C) arc(310:343:1.5 and 2) node (X) {\color{red}{\footnotesize{\textbullet}}};
\draw[->, shorten >= 2mm, shorten <=2mm] (X) arc(343:362:1.5 and 2) node (Y) {\color{red}{\footnotesize{\textbullet}}};
\draw[->, shorten >= 2mm, shorten <=2mm] (Y) arc(362:380:1.5 and 2) node (Z) {\color{red}{\footnotesize{\textbullet}}};
\draw[->, shorten >= 2mm, shorten <=2mm] (Z) arc(380:410:1.5 and 2);
	
\node (XX) at (6,1) {\color{red}\footnotesize{\textbullet}};
\draw[->, shorten >= 2mm, shorten <=2mm] (XX) arc[start angle = 120, end angle = 420, radius = 1.5] node (XY) {\color{red}\footnotesize{\textbullet}};
\draw[->, shorten >= 2mm, shorten <=2mm] (XY) arc[start angle = 420, end angle = 450, radius = 1.5] node (XZ) {\color{red}\footnotesize{\textbullet}};
\draw[->, shorten >= 2mm, shorten <=2mm] (XZ) arc[start angle = 450, end angle = 480, radius = 1.5];

\end{tikzpicture}

\caption{The value $M_2\in \M$ of the simplicial object $M_\bullet$.}

\end{subfigure}

\caption{An object $M$ with an excision site and a picture of $M_2$.}

\end{figure}

\begin{lemma}
\label{t4.5}
Consider the context $(\w{\Gamma},S,\varphi, M')$ of \Cref{d4.5}.
The colimit
\[
M
~:=~
\colim
\Bigl(
\bDelta^{\op}
\xra{~M\bullet~}
\M
\Bigr)
~\in~
\M
\]
exists, and each of the canonical morphisms $M_p \to M$ is a refinement.

\end{lemma}

\begin{proof}
We first establish the case in which $M' = \emptyset$, so that $M_\bullet = \Gamma_\bullet$.
Observe the factorization
\[
\Gamma_\bullet
\colon
\bDelta
\xra{~\rm diagonal~}
\bDelta^{\times S}
\xra{~\Gamma_{\vec{\bullet}}~}
\Quiv
\]
where, for $\bigl( [p_s] \bigr)_{s\in S} \in \bDelta^{\times S}$, the value $\Gamma_{([p_s])_{s\in S}} \in \Quiv$ is the pushout in $\digraphs$:
\[
\begin{tikzcd}
(\SS^0)^{\amalg S} 
\arrow{d}
\arrow{r}{\varphi}
&
\w{\Gamma}
\arrow{d}
\\
{\underset{s\in S} \coprod [p_s]}
\arrow{r}
&
\Gamma_{(p_s)_{s\in S}}
\end{tikzcd}
~.
\]
Using that $\bDelta^{\op}$ is sifted, the lemma is therefore implied by the case in which $S$ is a singleton. 
So assume $S$ is a singleton.

Consider the largest subgraph $\Gamma'\subseteq \w{\Gamma}$ that does not contain vertices in the image $\varphi(\SS^0) \subseteq \w{\Gamma}^{(0)}$ of $\varphi$.
There are two cases for the map $\SS^0 \xra{\varphi} \w{\Gamma}$.
\begin{itemize}

\item[]
{\bf Case 1:}
the unique incoming edge to the vertex $\varphi(-1)$ 
is equal to the unique outgoing edge to the vertex $\varphi(+1)$.  
In this case, the unique such edge is necessarily a cofactor $\DD^1$ of $\w{\Gamma}$, which is to say, there is an isomorphism 
\[
\w{\Gamma} 
~\cong~ 
\DD^1 \amalg \Gamma'
\]
under $\SS^0$ (via the inclusion $\SS^0 \hookrightarrow \DD^1$).
Therefore, the functor $\bDelta \xra{\Gamma_\bullet} \Quiv$ is isomorphic with the composite functor
\[
\bDelta
\xra{~\Cref{e15}~}
\copara
\xra{~\Cref{e21}~}
\Quiv
\xra{~- \amalg \Gamma'~}
\Quiv
~.
\]
Because $\bDelta^{\op} \xra{\Cref{e15}} \para$ is final (\Cref{t36}), \Cref{t58} implies the colimit of the composite functor
\[
\bDelta^{\op}
\xra{~\Cref{e15}~}
\para
\xra{~\Cref{e21}~}
\Quiv^{\op}
\xra{~- \sqcup \Gamma'~}
\Quiv^{\op}
\xra{~\delta~}
\M
\]
exists.  
Notice that, for each $[p]\in \bDelta$, the morphism $\SS^1_p \to \SS^1$ is a refinement, and therefore the morphism $\Gamma_p \cong \SS^1_p \sqcup \Gamma' \to \SS^1 \sqcup \Gamma' \cong M$ is a refinement.
This proves statement~(1) in this case.

\item[]
{\bf Case 2.}
the unique incoming edge $e_-$ to the vertex $\varphi(-1)$ 
is not equal to the unique outgoing edge $e_+$ to the vertex $\varphi(+1)$.  
Consider the maps between directed graphs $\DD^1_- \xra{\langle e_- \rangle} \w{\Gamma} \xla{\langle e_+ \rangle} \DD^1_+$ selecting these edges.
In this case, $\w{\Gamma}$ fits into a pushout diagram in $\digraphs$,
\[
\begin{tikzcd}
\partial_- \DD^1
\amalg
\partial_+ \DD^1
\arrow{d}
\arrow{r}
&
\Gamma'
\arrow{d}
\\
\DD^1_- \amalg \DD^1_+
\arrow{r}
&
\w{\Gamma}
\end{tikzcd}
~,
\]
where $\partial_{\pm} \DD^1_{\pm}  = \{\pm 1\} \subset (\DD^1_\pm)^{(0)}$ are the respective target/source vertices of the directed edge $\DD^1_{\pm}$.
Therefore, for each $[p]\in \bDelta$, the value $\Gamma_p$ fits into a pushout in $\digraphs$:
\[
\begin{tikzcd}
\partial_- \DD^1
\amalg
\partial_+ \DD^1
\arrow{d}
\arrow{r}
&
\Gamma'
\arrow{d}
\\
{[p]^{\lcone \rcone}}
\arrow{r}
&
\Gamma_p
\end{tikzcd}
~,
\]
where the vertical map selects, respectively, the initial and final cone points.
Using \Cref{t25}, which states that $\digraphs \xra{\Free} \Quiv$ preserves cobase-change along monomorphisms, the functor $\bDelta \xra{\Gamma_\bullet} \Quiv$ fit into a pushout among functors:
\[
\begin{tikzcd}
\partial_- \DD^1
\amalg
\partial_+ \DD^1
\arrow{d}
\arrow{r}
&
\Gamma'
\arrow{d}
\\
{[\bullet]^{\lcone \rcone}}
\arrow{r}
&
\Gamma_\bullet
\end{tikzcd}
~.
\]

Now, consider the pushout in $\digraphs$:
\begin{equation}
\label{s15}
\begin{tikzcd}
\partial_- \DD^1
\amalg
\partial_+ \DD^1
\arrow{d}
\arrow{r}
&
\Gamma'
\arrow{d}
\\
\emptyset^{\lcone \rcone}
\arrow{r}
&
\Gamma_\emptyset
\end{tikzcd}
~.
\end{equation}
By construction, for each $[p]\in \bDelta$, the unique morphism $\emptyset \to [p]$ in $\Quiv$ determines a refinement morphism $\emptyset^{\lcone \rcone} \to [p]^{\lcone \rcone}$ in $\Quiv$.  
Applying $\Quiv^{\op} \xra{\delta} \M$ to this refinement morphism determines a refinement morphism $[p]^{\lcone \rcone} \to \emptyset^{\lcone \rcone}$ in $\M$.
It remains to show $\Gamma_\emptyset$ witnesses the sought colimit.

The unique morphism $\emptyset \xra{!} [\bullet]$ induced an extension to the left-cone on $\bDelta$: 
\[
(\bDelta)^{\lcone} \xra{ \emptyset^{\lcone \rcone} \xra{!^{\lcone \rcone}} [\bullet]^{\lcone \rcone} } \Quiv
~.
\]
Observe that this extension witnesses a limit.
In other symbols, there is an identification in $\Quiv^{\op}$ of the geometric realization:
\[
| [\bullet]^{\lcone \rcone} |
~=~
\colim
\Bigl(
\bDelta^{\op}
\xra{  [\bullet]^{\lcone \rcone} } \Quiv^{\op}
\Bigr)
\xra{~\simeq~}
\emptyset^{\lcone \rcone}
~.
\]
By definition of $\Gamma_\emptyset \in \digraphs$, this identification supplies an identification in $\Quiv^{\op}$ of the geometric realization
\[
| \Gamma_\bullet |
~=~
\colim
\Bigl(
\bDelta^{\op}
\xra{  \Gamma_\bullet } \Quiv^{\op}
\Bigr)
\xra{~\simeq~}
\Gamma_\emptyset
~.
\]
Now, this geometric realization is a \bit{split geometric realization}. 
Indeed, the functor $\bDelta \xra{ [\bullet]^{\lcone \rcone} } \Quiv$ factors through the functor 
\[
\bDelta
\xra{~[\bullet]^{\lcone \rcone}~}
\bDelta
~\hookrightarrow~
\Quiv
~.
\]
It follows that the functor $\bDelta \xra{\Gamma_\bullet} \Quiv$ also factors through $\bDelta \xra{[\bullet]^{\lcone \rcone}} \bDelta$.
Therefore, the geometric realization $|\Gamma_\bullet|$ is also split.  
Consequently, for any functor $\Quiv^{\op} \xra{F} \cZ$ to an $\infty$-category, the canonical morphism in $\cZ$,
\[
\colim
\Bigl(
\bDelta^{\op}
\xra{~\Gamma_\bullet~}
\Quiv^{\op}
\xra{~F~}
\cZ
\Bigr)
\xra{~\simeq~}
F(\Gamma_\emptyset) 
~,
\]
is an equivalence.
In particular, there is an identification of the colimit in $\M$,
\[
\colim
\Bigl(
\bDelta^{\op}
\xra{~\Gamma_\bullet~}
\Quiv^{\op}
\xra{~\delta~}
\M
\Bigr)
\xra{~\simeq~}
\delta(\Gamma_\emptyset)
~=~
M
~.
\]

\end{itemize}

We now prove the general case, in which $M'\in \M$ is arbitrary.  
Denote $M'' := \colim\left( \bDelta^{\op} \xra{\Gamma_\bullet} \M \right)$ -- above, we argued that $M''$ exists, and each canonical morphism $\Gamma_p \to M''$ is a refinement.  
Using \Cref{t63}, choose a finite set $C$, a quiver $\Xi$, and an equivalence $M' \simeq (\SS^1)^{\sqcup C} \sqcup \Xi$ in $\M$.  
Then the functor $\bDelta^{\op} \xra{M_\bullet} \M$ is equivalent with the composite functor
\begin{equation}
\label{s1}
\bDelta^{\op}
\xra{\Gamma_\bullet \sqcup \Xi}
\Quiv^{\op}
\xra{\delta}
\M
\xra{-\sqcup (\SS^1)^{\sqcup C}}
\M
~.
\end{equation}
As the $\infty$-category $\bDelta^{\op}$ is sifted, colimits indexed by it distribute over products in $\infty$-categories of presheaves.  
Therefore, we have an identification of the colimit in $\M$:
\begin{equation}
\label{s3}
\colim\left( \bDelta^{\op}
\xra{\Gamma_\bullet \sqcup \Xi}
\Quiv^{\op}
\xra{\delta}
\M \right)
~\simeq~
\colim\left(\bDelta^{\op}
\xra{\Gamma_\bullet}
\Quiv^{\op}
\xra{\delta}
\M
\right)
\sqcup
\Xi
~\simeq~
M'' \sqcup \Xi
~.
\end{equation}
By definition of $\SS^1 \in \M$, each $\SS^1$ factor of the values of this functor is a paracyclic colimit.  
So consider the composite functor
\begin{equation}
\label{s2}
\bDelta^{\op}
\times
(\bDelta^{\op})^{\times C}
\xra{ \left( \Gamma_\bullet \sqcup \Xi, (\Cref{e21}\circ \Cref{e15})_{c\in C} \right) }
\Quiv^{\op} \times (\Quiv^{\op})^{\times C}
\xra{~\bigsqcup~}
\Quiv^{\op}
\xra{~\delta~}
\M
~.
\end{equation}
\Cref{t36} implies the left Kan extension of \Cref{s2} along the projection 
$
\bDelta^{\op}
\times
(\bDelta^{\op})^{\times C}
\xra{\pr_1}
\bDelta^{\op}
$
is the simplicial object \Cref{s1}, whose colimit we seek to prove exists and is $M'' \sqcup M'$.
We can compute this colimit, alternatively, as the colimit of the left Kan extension along the other projection 
$
\bDelta^{\op}
\times
(\bDelta^{\op})^{\times C}
\xra{\pr}
(\bDelta^{\op})^{\times C}
$.
The left Kan extension of \Cref{s2} along this other projection evaluates on an object as a colimit whose existence was argued above; the identification \Cref{s3} supplies an identification of the resulting functor as
\[
(\bDelta^{\op})^{\times C}
\xra{~(\Cref{e21}\circ \Cref{e15})_{c\in C}~}
(\Quiv^{\op})^{\times C}
\xra{~\bigsqcup~}
\Quiv^{\op}
\xra{\delta}
\M
\xra{~\sqcup M'' \sqcup \Xi}
\M
~.
\]
\Cref{t58'} ensures the colimit of this functor exists, and is $M'' \sqcup (\SS^1)^{\sqcup C} \sqcup \Xi \simeq M'' \sqcup M' =: M$.
Finally, because each of the canonical morphisms $\Gamma_p \to M''$ is a refinement, then each of the canonical morphisms $M_p = \Gamma_p \sqcup M' \to M'' \sqcup M' = M$ is a refinement.

\end{proof}

Let $(\w{\Gamma},S,\varphi, M')$ be an excision site for an object $M\in \M$.
Denote the base-changes among $\infty$-categories:
\[
\begin{tikzcd}
	{\Ar(\M)^{|\Quiv^{\op}}} & {\Ar(\M)} \\
	{\Quiv^{\op}} & \M
	\arrow[hook, from=1-1, to=1-2]
	\arrow[swap, "\ev_s", from=1-1, to=2-1]
	\arrow["{\ev_s}", from=1-2, to=2-2]
	\arrow["\delta", hook, from=2-1, to=2-2]
\end{tikzcd}
~.
\]
Denote the base-changes among $\infty$-categories:
\begin{equation}
\label{s13}
\begin{tikzcd}
	{\Ar(\M)^{|\Quiv^{\op}}_{|\bDelta^{\op}}} & {\Ar(\M)^{|\Quiv^{\op}}_{|(\bDelta^{\op})^{\rcone}}} & {\Ar(\M)^{|\Quiv^{\op}}} \\
	{\bDelta^{\op}} & {(\bDelta^{\op})^{\rcone}} & \M
	\arrow[hook, from=1-1, to=1-2]
	\arrow[swap, "{\ev_t}", from=1-1, to=2-1]
	\arrow[from=1-2, to=1-3]
	\arrow["{\ev_t}", from=1-2, to=2-2]
	\arrow["{\ev_t}", from=1-3, to=2-3]
	\arrow[hook, from=2-1, to=2-2]
	\arrow["{M_\bullet}", from=2-2, to=2-3]
\end{tikzcd}
~.
\end{equation}
The right vertical functor is a coCartesian fibration, and therefore all of the vertical functors are coCartesian fibrations.
The fiber of the right vertical functor over $N\in \M$ is $\Quiv^{\op}_{/N}$.
Therefore, the fiber of the middle vertical functor over the cone-point is $\Quiv^{\op}_{/M}$.
Because the cone-point in $(\bDelta^{\op})^{\rcone}$ is final, coCartesian monodromy functors assemble as a functor
\begin{equation}
\label{s14}
\Ar(\M)^{|\Quiv^{\op}}_{|\bDelta^{\op}}
\xra{~\rm monodromy~}
\Quiv^{\op}_{/M}
~.
\end{equation}
The fiber over each object $[p]\in \bDelta$ of the left vertical functor in \Cref{s13} is $\Quiv^{\op}_{/M_p}$.
We explain the composite functor:
\begin{equation}
\label{s19}
\Quiv(M')
=
\ast 
\times \Quiv(M')
\xra{ \langle \Gamma_p \rangle \times \id }
\Quiv(\Gamma_p) \times \Quiv(M')
~\simeq~
\Quiv(M)
\hookrightarrow
\Quiv^{\op}_{/M}
~,
\end{equation}
\[
(\w{M}' \xra{r} M')
\longmapsto
\left(
D_p
\sqcup
\w{M}' 
\xra{ \id \sqcup r} D_p \sqcup M' 
\xra{\rm refinement}
M
\right)
~.
\]
Using that $\Gamma_p \in \Quiv$, the category $\Quiv(\Gamma_p)$ has the identity morphism $(\Gamma_p \xra{\id} \Gamma_p)$ as a final object (see \Cref{r11}(3)).
The first functor is the product of the functor that selects this final object and an identity functor.
Being a product of final functors, the first functor is final.
Using the identification $M \simeq \Gamma_p \sqcup M'$, the middle equivalence follows from \Cref{t71}(3).
The last functor is the defining fully faithful inclusion, which \Cref{t72} states is final.
We have established the composite functor \Cref{s19}, and also its finality.
Now, as $[p] \in \bDelta$ varies, these final functors canonically organize as a final functor $\eta$ fitting into the diagram among $\infty$-categories:
\[\begin{tikzcd}
	{\bDelta^{\op} \times \Quiv(M')} && {\Ar(\M)^{|\Quiv^{\op}}_{|\bDelta^{\op}}} \\
	& {\bDelta^{\op}}
	\arrow["\eta", from=1-1, to=1-3]
	\arrow["\pr"', from=1-1, to=2-2]
	\arrow["{\ev_t}", from=1-3, to=2-2]
\end{tikzcd}
~.
\]
Because each of the canonical morphisms $M_p \to M$ in $\M$ is a refinement (\Cref{t4.5}), there is a (necessarily unique) factorization of the composite functor
\begin{equation}
\label{s9}
\begin{tikzcd}
	{\bDelta^{\op} \times \Quiv(M')} && {\Quiv(M)} \\
	{\Ar(\M)^{|\Quiv^{\op}}_{|\bDelta^{\op}}} && {\Quiv^{\op}_{/M}}
	\arrow[dashed, from=1-1, to=1-3]
	\arrow["\eta"', from=1-1, to=2-1]
	\arrow[hook, from=1-3, to=2-3]
	\arrow[swap, "{\Cref{s14}}"', from=2-1, to=2-3]
\end{tikzcd}
~.
\end{equation}

The next technical result is key for excision of factorization homology, as developed in the next section.
\begin{lemma}
\label{s10}
Each of the functors
\[
\bDelta^{\op} \times \Quiv(M') 
\xra{~\Cref{s9}~}
\Quiv(M)
\qquad\text{ and }\qquad
\Ar(\M)^{|\Quiv^{\op}}_{|\bDelta^{\op}}
\xra{~\Cref{s14}~}
\Quiv^{\op}_{/M}
\]
is final.

\end{lemma}

\begin{proof}
The functor $\eta$ constructed above was already observed to be final.  
\Cref{t72} asserts that the right vertical functor in \Cref{s9} is final.  
By the 2-out-of-3 properties of final functors, the dashed filler in \Cref{s9} is final if and only if \Cref{s14} is final. 
We prove the dashed filler in \Cref{s9} is final.

\Cref{t71}(3) supplies an identification $\Quiv(M) \xra{\simeq} \Quiv(M'') \times \Quiv(M')$, where $M'':= \left| \Gamma_p \right|$, which \Cref{t4.5} ensures exists.  
With respect to this identification, there is an identification as a product:
\[
\left(
\bDelta^{\op} \times \Quiv(M') 
\xra{~\Cref{s9}~}
\Quiv(M)
\right)
~\simeq~
\left(
\bDelta^{\op} \xra{\Gamma_\bullet} \Quiv(M'')
\right)
\times \Quiv(M')
~.
\]
We can therefore reduce to the case in which $M'=\emptyset$, which is to say $M_\bullet = \Gamma_\bullet$.
Using that $\bDelta^{\op}$ is sifted, we can reduce further to the case in which $S$ is a singleton.  
As in the proof of \Cref{t4.5}, there are then two cases to consider.  

In the first case, $\Gamma_\bullet \simeq \SS^1_\bullet \sqcup \Gamma'$.
In this case, $M = \SS^1 \sqcup \Gamma'$, and by \Cref{t71}(3) there is a canonical identification $\Quiv(M) \xra{\simeq} \Quiv(\SS^1) \times \Quiv(\Gamma')$.
Through this identification, 
the functor $\bDelta^{\op} \xra{\Cref{s9}}
\Quiv(M)$ is a composite 
\[
\Cref{s9}
\colon
\bDelta^{\op} \xra{\Cref{s9}} \Quiv(\SS^1) \xra{ \id \times \langle \Gamma' \xra{=} \Gamma' \rangle} \Quiv(\SS^1) \times \Quiv(\Gamma')
~\simeq~
\Quiv(M)
~,
\]
involving the functor \Cref{s9} applied to the case of $\SS^1_\bullet$.  
As $\Gamma'$ is itself a quiver, $\Quiv(\Gamma')$ has a final object, which is that selected by the second factor of the above composite.  
Through the identification $\Quiv(\SS^1) \simeq \para$ of \Cref{r11}(2), the finality of the above composite functor is therefore implied by \Cref{t36}.

In the second case, $\Gamma_\bullet \simeq [\bullet]^{\lcone \rcone}\underset{\partial_- \DD^1 \sqcup \partial_+ \DD^1} \bigsqcup \Gamma'$.
In this case, $M \simeq \delta(\Gamma_\emptyset)$ is in the image of $\Quiv^{\op} \xra{\delta} \M$, where the quiver $\Gamma_\emptyset$ is a pushout \Cref{s15} in $\digraphs$.
Base-change in $\Quiv$ of refinement morphisms along each of the morphsims in the diagram \Cref{s15} exist, and organize as a diagram among categories:
\[
\begin{tikzcd}
\Quiv\left(
\Gamma_\emptyset
\right)
\arrow{d}
\arrow{r}
&
\Quiv\left(
\Gamma'
\right)
\arrow{d}
\\
\Quiv\left(
\emptyset^{\lcone \rcone}
\right)
\arrow{r}
&
\Quiv\left(
\partial_- \DD^1
\amalg
\partial_+ \DD^1
\right)
\end{tikzcd}
~.
\]
Furthermore, as a refinement of the quiver $\Gamma_\emptyset$ is precisely a pair of refinements, one of $\emptyset^{\lcone \rcone}$ and one of $\Gamma'$, this diagram is a pullback:
\[
\Quiv(M)
~\simeq~
\Quiv\left(
\Gamma_\emptyset
\right)
\xra{~\simeq~}
\Quiv\left(
\emptyset^{\lcone \rcone}
\right)
\times
\Quiv\left(
\Gamma'
\right)
~.
\]
Through this identification,
the functor $\bDelta^{\op} \xra{\Cref{s9}}
\Quiv(M)$ is a composite 
\[
\Cref{s9}
\colon
\bDelta^{\op} \xra{\Cref{s9}} \Quiv(\emptyset^{\lcone \rcone}) \xra{ \id \times \langle \Gamma' \xra{=} \Gamma' \rangle} \Quiv(\emptyset^{\lcone \rcone}) \times \Quiv(\Gamma')
~\simeq~
\Quiv(M)
~,
\]
involving the functor \Cref{s9} applied to the case of $[\bullet]^{\lcone \rcone}$.  
As $\Gamma'$ is itself a quiver, $\Quiv(\Gamma')$ has a final object, which is that selected by the second factor of the above composite.  
We are therefore reduced to showing the functor 
$\bDelta^{\op} \xra{\Cref{s9}} \Quiv(\emptyset^{\lcone \rcone})$ 
is final.
Now, observe an identification $\Quiv(\emptyset^{\lcone \rcone}) \subset (\bDelta^{\emptyset^{\lcone \rcone}/})^{\op}$
as the full subcategory consisting of those $\emptyset^{\lcone \rcone} \to [p]$ that are active and injective.
In other words, $\Quiv(\emptyset^{\lcone \rcone})$ is the opposite of the category $\bDelta^{\emptyset^{\lcone \rcone}/^{\sf act,inj}}$ in which an object is a finite linearly ordered sets with distinct minimum and maximum, and a morphism is a map between linearly ordered sets that preserves extrema.  
This identification supplies an identification between categories under $\bDelta^{\op}$:
\[
\left(
\bDelta^{\op} \xra{ \Cref{s9} } \Quiv(\emptyset^{\lcone \rcone})
\right)
~\simeq~
\left(
\bDelta^{\op} \xra{ [\bullet]^{\lcone \rcone} } 
(\bDelta^{\emptyset^{\lcone \rcone}/^{\sf act,inj}})^{\op}
\right)
~.
\]
Let $\left( \emptyset^{\lcone \rcone} \xra{\sigma} [p] \right) \in \bDelta^{\emptyset^{\lcone \rcone}/^{\sf act,inj}}$.
Consider the overcategory
$
\bDelta_{/\sigma} 
:= 
\bDelta \underset{ 
\bDelta^{\emptyset^{\lcone \rcone}/^{\sf act,inj}})
}
\times
\left(
\bDelta^{\emptyset^{\lcone \rcone}/^{\sf act,inj}})
\right)_{/\sigma}
.
$
As every morphism $[q]\to [p]$ uniquely extends as an active morphism $[q]^{\lcone \rcone} \to [p]$,
restriction along $[\bullet] \hookrightarrow [\bullet]^{\lcone \rcone}$ defines an equivalence between categories:
$
\bDelta_{/\sigma}
\xra{\simeq}
\bDelta_{/[p]}
$.
The overcategory $\bDelta_{/[p]}$ has a final object, and therefore its $\infty$-groupoid-comletion is contractible.  
Quillen's Theorem A can therefore be applied to
reveal that the functor $\bDelta \xra{ [\bullet]^{\lcone \rcone} } 
\bDelta^{\emptyset^{\lcone \rcone}/^{\sf act,inj}}$ is initial.
Therefore, the functor $\bDelta^{\op} \xra{ \Cref{s9} } \Quiv(\emptyset^{\lcone \rcone})$ is final, as desired.

\end{proof}

The simplicial object $M_\bullet$ determined from an excision site, as in \Cref{d4.5} has a conceptual description as a cyclic bar construction.
\begin{lemma}
\label{t4.7}
Let $(\w{\Gamma},S,\varphi,M')$ be an excision site.
\begin{enumerate}
\item
The composite functor
\[
\cA
\colon
\bDelta^{\op}
\xra{~\rho~}
\Quiv^{\op}
\xra{~(-)^{\sqcup S} ~}
\Quiv^{\op}
\hookrightarrow
\M
~,\qquad
[p]
\longmapsto 
\rho(p)^{\sqcup S}
~,
\]
is a category-object.

\item
The object $\w{\Gamma} \sqcup M' \in \M$ has the structure of a $(\cA,\cA)$-bimodule.

\item
The simplicial object $\bDelta^{\op} \xra{M_\bullet} \M$ is the cyclic bar construction of $\cA$ with coefficients in the bimodule $\w{\Gamma}\sqcup M'$.

\end{enumerate}

\end{lemma}

\begin{proof}
The second functor in the definition of $\cA$ is given by product with the finite set $S$.
Because products with a finite set preserve basic closed cover diagrams in $\Quiv$ (which are certain pushouts in $\Cat_{(\infty,1)}$), this composite functor carries basic closed covers to basic closed covers.  
Consequently, $\cA$ is a category-object, which proves statement~(1).

Consider the functor 
\[
F
\colon
(\bDelta_{/[1]})^{\op}
\longrightarrow
\M
\]
given as follows.
Its restriction along both of the functors $\bDelta^{\op} = (\bDelta_{/\{0\}})^{\op} \hookrightarrow  (\bDelta_{/[1]})^{\op}$ and $\bDelta^{\op} = (\bDelta_{/\{1\}})^{\op} \hookrightarrow  (\bDelta_{/[1]})^{\op}$ is the category-object $\cA$;
its restriction along the functor $(\bDelta_{/^{\sf surj}[1]})^{\op} \hookrightarrow  (\bDelta_{/[1]})^{\op}$ is given by 
\[
\left(
[p]
\xra{\sigma}
[1]
\right)
\longmapsto
\cA(\sigma^{-1}(0))
\underset{S} \times
\left( \w{\Gamma}\sqcup M' \right)
\underset{S} \times
\cA(\sigma^{-1}(1))
~,
\]
where $S \la \w{\Gamma} \sqcup M' \to S$ are the two closed morphisms associated with the inclusion $\varphi$, and where $\cA(\sigma^{-1}(0)) \to \cA({\sf Max}(\sigma^{-1}(0)) = S$ and $\cA(\sigma^{-1}(1)) \to \cA({\sf Min}(\sigma^{-1}(1)) = S$.
These specifications define $F$ on a subcategory of $(\bDelta_{/[1]})^{\op}$ that contains all objects and all closed morphisms.  
Defining $F$ on the entirety of $(\bDelta_{/[1]})^{\op}$ can be achieved by hand, using that $\Quiv^{\op}$ is an ordinary category, or through a case-by-case analysis as in the proof of \Cref{t4.5}.  
Evidently, $F$ carries (the opposites of) closed covers over $[1]$ to closed covers in $\M$.
As so, $F$ codifies a $(\cA,\cA)$-bimodule structure on $\w{\Gamma} \sqcup M'$ in $\M$.
This establishes statement~(2).

For $[p]\in \bDelta$, there is a canonical morphism in $\bDelta$ from the join, $[p] \star [p] \xra{! \star !} \ast \star \ast = [1]$, which fits into a commutative diagram in $\bDelta$:
\[
\begin{tikzcd}
	{[p]} & {[p]\star [p]} & {[p]} \\
	{\{0\}} & {[1]} & {\{1\}}
	\arrow[from=1-1, to=1-2]
	\arrow[from=1-1, to=2-1]
	\arrow[from=1-2, to=2-2]
	\arrow[from=1-3, to=1-2]
	\arrow[from=1-3, to=2-3]
	\arrow[from=2-1, to=2-2]
	\arrow[from=2-3, to=2-2]
\end{tikzcd}
~,
\]
in which the upper horizontal morphisms are the respective inclusions of the left and right joinands.  
These data organize, as $[p]\in \bDelta$ varies, as a cospan in $\Fun(\bDelta, \bDelta_{/[1]} )$:
\[
\left(
[\bullet]
\da
\{0\}
\right)
\longrightarrow
\left(
[\bullet] \star [\bullet]
\da
[1]
\right)
\longleftarrow
\left(
[\bullet]
\da
\{1\}
\right)
~.
\]
Taking opposites then post-composing with the functor $F$ results in a span in $\Fun(\bDelta^{\op} , \M)$:
\[
\cA
\xla{~R~}
F\left(
[\bullet] \star [\bullet]
\da
[1]
\right)
\xra{~L~}
\cA
~,
\]
using the defining identifications 
$F\left(
[\bullet]
\da
\{0\}
\right)
=
\cA
=
F\left(
[\bullet]
\da
\{1\}
\right)$.

Inspecting the defining values of $F$, of $\cA$, and of $M_\bullet$, reveals that, for each $[p]\in \bDelta$, there is a canonical diagram in $\cM$,
\begin{equation}
\label{k20}
\begin{tikzcd}
	{M_p} && {\cA([p])} && {F([p] \star [p] \da [1])}
	\arrow["{{\sf cls}}", from=1-1, to=1-3]
	\arrow["R", shift left=2, from=1-3, to=1-5]
	\arrow["L"', shift right=2, from=1-3, to=1-5]
\end{tikzcd}
~.
\end{equation}
which in fact is $-\sqcup M'$ applied to such a diagram in $(\digraphs)^{\op} \subset \Quiv^{\op} \subset \M$.
Notice that this diagram in $(\digraphs)^{\op}$ witnesses an equalizer, and the diagram \Cref{k20} in $\M$ is also witnesses an equalizer.  
The diagram \Cref{k20} in $\M$ organizes as diagram in $\Fun(\bDelta^{\op},\M)$, which therefore also witnesses an equalizer.  
In other words, $M_\bullet$ is the cyclic bar construction of the $(\cA,\cA)$-bimodule codified by $F$.

\end{proof}

\section{Factorization homology}

In this section, we fix an $\infty$-category $\cX$ with the following properties.
\begin{enumerate}
\item
$\cX$ admits finite limits.

\item
$\cX$ admits geometric realizations.

\item
For each $X\in \cX$, the functor $X\times - \colon \cX \to \cX$ preserves geometric realizations.

\end{enumerate}
We implement a purely combinatorial version of factorization homology $\int_M \cC \in \cX$ of any category-object $\cC$ in $\cX$ over an object $M\in \M$.
As Appendix~\S\ref{sec.facts}, we show this combinatorial version of factorization homology agrees with the geometric version, defined in~\cite{AFR2}.

\subsection{Factorization homology}

Recall the \Cref{notation.brackets.functor.from.delta.op} of the functor $\bDelta \xra{\rho} \Quiv$, and the Definition~\ref{d7.2} of the functor $\Quiv^{\op}\xra{\delta} \M$.
\begin{definition}
\label{d21}
\bit{Combinatorial factorization homology} is the composite functor
\[
\int
\colon
\fCat_1[\cX]
\xra{~\rho_\ast~}
\Fun(\Quiv^{\op}, \cX)
\xra{~\delta_!~}
\Fun(\M , \cX)
~,\qquad
\cC
\longmapsto
\Bigl(
M
\mapsto
\int_M \cC
\Bigr)
~,
\]
of right Kan extension along $\bDelta^{\op} \xra{\rho} \Quiv^{\op}$ followed by left Kan extension along $\Quiv^{\op} \xra{\delta}\M$.
For $\cC\in \fCat_1[\cX]$ and $M\in \M$, the \bit{combinatorial factorization homology of $\cC$ over $M$} is the value
\[
\int_M \cC
~:=~
\colim
\Bigl(
\Quiv^{\op}_{/M}
\xra{\rm forget}
\Quiv^{\op}
\xra{\rho_\ast(\cC) \colon \Gamma
\mapsto 
{\sf Rep}_\cC(\Gamma)
}
\cX
\Bigr)
~\in
\cX
~.
\]

\end{definition}

\begin{observation}
\label{t74}
The assumptions on $\cX$ ensure combinatorial factorization homology $\int$ exists.
Indeed, by \Cref{t65}, these assumptions ensure that the right Kan extension $\rho_\ast$ exists; 
by \Cref{finality}, these assumptions ensure that the left Kan extension $\delta_!$ exists.

\end{observation}

Recall from \Cref{d21.1} the full $\infty$-subcategory $\Quiv(M) \subseteq \Quiv^{\op}_{/M}$ for each $M\in \M$.
The following is a direct consequence of \Cref{t72}, which states that this $\infty$-subcategory is final.
\begin{cor}
\label{t72.1}
Let $\cC\in \fCat_1[\cX]$ be a category-object in $\cX$.
For each $M\in \M$, combinatorial factorization homology can be computed as the colimit indexed by $\Quiv(M)$:
the canonical morphism in $\cX$,
\[
\underset{\Gamma \xra{\sf ref} M} \colim {\sf Rep}_\cC(\Gamma)
~=~
\colim
\Bigl(
\Quiv(M) 
\hookrightarrow
\Quiv^{\op}_{/M}
\xra{\rm forget}
\Quiv^{\op}
\xra{\rho_\ast(\cC) \colon \Gamma
\mapsto 
{\sf Rep}_\cC(\Gamma)
}
\cX
\Bigr)
\xra{~\simeq~}
\int_M \cC
~,
\]
is an equivalence.

\end{cor}

\begin{remark}
\label{r50}
The full $\infty$-subcategory $\Quiv(M) \subseteq \Quiv^{\op}_{/M}$ consists of far fewer objects, and simpler morphisms, than its ambient $\infty$-category.  
In this way, \Cref{t72.1} offers a simpler colimit formula for $\int_M \cC$ than the defining formula of \Cref{d21}.  
\end{remark}

\begin{remark}
\Cref{r50} suggests that one might take the colimit formula of \Cref{t72.1} as the working definition of $\int_M \cC$.  
However, while $M\mapsto \Quiv^{\op}_{/M}$ assembles a functor $\M \to \Cat_{(\infty,1)}$, the association $M\mapsto \Quiv(M)$ is not functorial in $M\in \M$.  
The full functoriality of $\int_M \cC$ is therefore not (obviously) available were one to take the formula of \Cref{t72.1} as a definition of $\int_M \cC$. 
\end{remark}

After 
\Cref{t36},
the following is a direct consequence of Lemma~\ref{finality}, using the assumed properties of $\cX$.
\begin{cor}
\label{t64}
Combinatorial factorization homology $\int$ exists.
Furthermore, it preserves finite products in both variables: 
\[
\int_{M\sqcup N} \cC \xra{~\simeq~} \int_M \cC \times \int_N \cC
\qquad
\text{ and }
\qquad
\int_M \cC \times \cD \xra{~\simeq~} \int_M \cC \times \int_M \cD
~.
\]

\end{cor}

The values of combinatorial factorization homology are summarized as follows.
Recall the notation introduced in Definition~\ref{d7.2}.
\begin{example}
\label{r10}
Let $\cC\in \fCat_1[\cX]$ be a category-object in $\cX$.
Combinatorial factorization homology of $\cC$ takes the following values.
\begin{enumerate}

\item
For each finite directed graph $\Gamma$, the value 
\[
\int_{\delta(\Gamma)} \cC
~\simeq~
\rho_\ast(\cC) \bigl( \Gamma \bigr)
~=:~
{\sf Rep}_\cC(\Gamma)
~.
\]
Indeed, this is the fact that the unit of the $(\delta_!,\delta^\ast)$-adjunction is an equivalence, which is so because $\delta$ is fully faithful (see \Cref{t40.2}(4)).

\item
For each $[p]\in \bDelta$, the value 
\[
\int_{\rho(p)} \cC
~\simeq~
\cC([p])
~.
\]
Indeed, after the first identification is the previous point, this follows from the fact that the counit of the $(\rho^\ast  , \rho_\ast)$-adjunction is an equivalence, which is so because $\rho$ is fully faithful.
In particular, 
\[
\int_{\DD^0} \cC 
~\simeq~
\Obj(\cC)
\qquad
\text{ and }
\qquad
\int_{\DD^1} \cC 
~\simeq~
\Mor(\cC)
~.
\]

\item
The value on the oriented circle is Hochschild homology:
\[
\int_{\SS^1} \cC 
~\simeq ~
\sHH(\cC)
~.
\]
Indeed, this follows from the fact that the functor $\para \to \Quiv^{\op}_{/\SS^1}$ is final 
(see \Cref{finality}).

\item
Let $M \in \M$.
Through \Cref{t69} there is a finite set $C$ and a directed graph $\Gamma$ together with an identification $M \simeq (\SS^1)^{\sqcup C}\sqcup \Gamma$ in $\M$.
The value
\[
\int_M \cC
~\simeq~
\sHH(\cC)^{\times C}
\times
{\sf Rep}_{\cC}(\Gamma)
~.
\]
Indeed, after the above points, using the assumption that products in $\cX$ distribute over geometric realizations, this follows from the fact that $\para$ is sifted (\Cref{t62}) and that the functor $(\para)^{\times C} \to \Quiv^{\op}_{/(\SS^1)^{\sqcup C}\sqcup \Gamma}$ is final 
(see \Cref{finality}).

\end{enumerate}

\end{example}

Example~\ref{r10} reveals that the values of combinatorial factorization homology are all familiar.
What is novel is the functoriality among these values, and in particular the natural symmetries of these values, which combinatorial factorization homology succinctly codifies.  
This is the subject of \S\ref{sec.symmetries}.

\subsection{Excision}

The next result is a version of \bit{excision} for combinatorial factorization homology $\int$.  
It is a direct consequence of \Cref{s10}.

We have the following immediate consequence of \Cref{t4.7}.
\begin{cor}
\label{t4.6}
Let $(\w{\Gamma},S,\varphi,M')$ be an excision site.
Let $\cC\in \fCat_1[\cX]$ be a category-object in $\cX$.
\begin{enumerate}

\item
The object $\int_{\w{\Gamma} \sqcup M'} \cC \in \cX$ has the structure of a $(\cC^{\times S} , \cC^{\times S})$-bimodule.

\item
The simplicial object $\bDelta^{\op} \xra{\int_{M_\bullet} \cC} \cX$ is the cyclic bar construction of $\cC^{\times S}$ with coefficients in the bimodule $\int_{\w{\Gamma}\sqcup M'} \cC$.

\end{enumerate}

\end{cor}

\begin{prop}
\label{t4.4}
Let $\cC \in \fCat_1[\cX]$ be a category-object in $\cX$.
Let $(\w{\Gamma},S,\varphi, M')$ be an excision site for an object $M\in \M$.
The canonical morphism in $\cX$,
\[
\underset{[p]^\circ  \in \bDelta^{\op}}\colim
\left(
\int_{M_p} \cC
\right)
~=~
\left|
\int_{M_\bullet} \cC
\right|
\xra{~\simeq~}
\int_M \cC
\]
is an equivalence.

\end{prop}

\begin{proof}
We explain how the diagram \Cref{s13} fits into the diagram among $\infty$-categories:
\begin{equation}
\label{s8}
\begin{tikzcd}
	{\Ar(\M)^{|\Quiv^{\op}}_{|\bDelta^{\op}}} & {\Ar(\M)^{|\Quiv^{\op}}_{|(\bDelta^{\op})^{\rcone}}} & {\Ar(\M)^{|\Quiv^{\op}}} & {\Quiv^{\op}} & \cX \\
	{\bDelta^{\op}} & {(\bDelta^{\op})^{\rcone}} & \M
	\arrow[hook, from=1-1, to=1-2]
	\arrow[from=1-1, to=2-1]
	\arrow[from=1-2, to=1-3]
	\arrow[from=1-2, to=2-2]
	\arrow["{\ev_s}", from=1-3, to=1-4]
	\arrow["{\ev_t}", from=1-3, to=2-3]
	\arrow["{\rho_\ast \cC}", from=1-4, to=1-5]
	\arrow[hook, from=2-1, to=2-2]
	\arrow["{M_\bullet}", from=2-2, to=2-3]
	\arrow["{\int \cC}"', from=2-3, to=1-5]
	\arrow["\Leftarrow"{marking, allow upside down, pos=0.3}, shift left=5, draw=none, from=2-3, to=1-5]
\end{tikzcd}
~.
\end{equation}
The squares in \Cref{s8} are (defined as) pullbacks.
By definition of factorization homology, the inner lax-commutative triangle in \Cref{s8} witnesses a left Kan extension.
The right vertical functor in \Cref{s8} is a coCartesian fibration.
Therefore, so are the other vertical functors in \Cref{s8}
Using that left Kan extension along a coCartesian fibration is given by fiberwise colimit, the other resulting lax-commutative triangles in \Cref{s8} also witnesses a left Kan extension.  
In particular, the left Kan extension to $\bDelta^{\op}$, here, is the simplicial object $\bDelta^{\op} \xra{\int_{M_\bullet} \cC} \cX$.
Because the cone-point in $(\bDelta^{\op})^{\rcone}$ is final, the fiber $\Quiv^{\op}_{/M}$ of the middle vertical functor over this cone-point is final in $\Ar(\M)^{|\Quiv^{\op}}_{|(\bDelta^{\op})^{\rcone}}$.  
Therefore the left Kan extension to $(\bDelta^{\op})^{\rcone}$ evaluates on the cone-point as $\int_M \cC$, and the colimit of this left Kan extension to $(\bDelta^{\op})^{\rcone}$ is this value $\int_M \cC$.

Now, regard \Cref{s8} as a diagram among $\infty$-categories over $\cX$.
Taking colimits in $\cX$ produces a diagram in $\cX$:
\begin{equation}
\label{s7}
\begin{tikzcd}
	{\colim \left( \Ar(\M)^{|\Quiv^{\op}}_{|\bDelta^{\op}} \xra{\rho_\ast \cC} \cX \right)} & {\colim\left( \Quiv^{\op}_{/M} \xra{\rho_\ast \cC} \cX \right)} \\
	{\left| \int_{M_\bullet} \cC \right|} & {\int_M \cC}
	\arrow[from=1-1, to=1-2]
	\arrow[from=1-1, to=2-1]
	\arrow[from=1-2, to=2-2]
	\arrow[from=2-1, to=2-2]
\end{tikzcd}
~.
\end{equation}
The left vertical morphism is an equivalence because the colimit of a left Kan extension is a colimit.
The upper horizontal morphism is an equivalence because the functor
\[
\Ar(\M)^{|\Quiv^{\op}}_{|\bDelta^{\op}}
\longrightarrow
\Quiv^{\op}_{/M}
\]
is final (\Cref{s10}).  
The right vertical morphism is an equivalence by definition of factorization homology.  
The 2-out-of-3 property of equivalences ensures the lower horizontal morphism is an equivalence, as desired.

\end{proof}

\begin{example}
After \Cref{t4.6}, \Cref{t4.4} is a \bit{local-to-global formula} for combinatorial factorization homology.  
It gives a method to compute the value $\int_M \cC$ of combinatorial factorization homology on a general object $M\in \M$ in terms of simpler values.
Here are some such examples.
\begin{enumerate}
\item
In the case that $\Gamma_\bullet = \rho( [\bullet]^{\lcone \rcone} )$, \Cref{t4.4} implies the canonical morphism given by the composition rule of $\cC$,
\[
\circ
\colon
\Bigl|
\Mor(\cC)^{\underset{\Obj(\cC)} \times \bullet+1}
\Bigr|
\xra{~\simeq~}
\Mor(\cC)
~,
\]
is an equivalence.
Taking fibers over $(c,d)\in \Obj(\cC)^{\times 2}$ via the source-target morphism $\Mor(\cC) \to \Obj(\cC)^{\times 2}$ recovers the familiar identity: the coend of the left $\cC$-module $\Hom_\cC(c,\bullet)$ with the right $\cC$-module $\Hom_\cC(\bullet,d)$ is the object of morphisms in $\cC$ from $c$ to $d$:
\[
\Hom_\cC(c,\bullet) \underset{ \bullet \in \cC} \bigotimes \Hom_\cC(\bullet,d)
\xra{~\simeq~}
\Hom_\cC(c,d)
~.
\]

\item
In the case that $S = \{a,b\}$ and $\w{\Gamma}$ is the disconnected quiver
\[
\{-a\}
\longrightarrow
\{+a\}
\qquad
\{-b\}
\longleftarrow
\{+b\}
,
\]
\Cref{t4.4} implies the canonical morphism given by the composition rule of $\cC$,
\[
\circ
\colon
\Hom_\cC(\bullet,\bullet')
\underset{ (\bullet^\circ , \bullet') \in \cC^{\op} \times \cC} \bigotimes
\Hom_\cC(\bullet',\bullet)
\xra{~\simeq~}
\sHH(\cC)
~,
\]
is an equivalence, which
witnesses the (non-stable) Hochschild homology of $\cC$ as the coend of the identity $(\cC,\cC)$-bimodule $\cC$ with itself.

\item
In the case that $\w{\Gamma}$ of \Cref{d4.5} is the quiver
\[
\{-1\}
\longrightarrow
\{+1\}
\]
and $S$ is a singleton, 
\Cref{t4.4} implies the canonical morphism given by the composition rule of $\cC$,
\[
\circ
\colon
\Bigl|
{\sf Rep}_\cC( \chi_\bullet)
\Bigr|
\xra{~\simeq~}
\sHH(\cC)
~,
\]
is an equivalence.
Here, the cosimplicial quiver
$
\chi_\bullet \colon 
\bDelta 
\xra{~\Cref{e15}~}
\copara
\xra{~\Cref{e21}~} 
\Quiv
$ 
evaluates on $[p]$ as the cyclically-directed quiver whose (cyclically-directed) set of vertices is $\{0,1,\dots,p\}$.
\end{enumerate}

\end{example}

\subsection{Natural symmetries of Hochschild homology}
\label{sec.symmetries}
Example~\ref{r10}(4) suggests that the most interesting value of combinatorial factorization homology is that over the oriented circle, which is Hochschild homology.
In this subsection, we codify the natural symmetries of this value as a (\bit{non-stable}) \bit{cyclotomic object}.  
Namely, the value $\sHH(\cC)$ naturally has the structure of a \bit{proper-genuine $\TT$-module} that is fixed with respect to a natural $\NN^\times$-action on such.
We refer the reader to Appendix~\ref{sec.A} for definitions of these bolded concepts.

\begin{theorem}
\label{t73}
\begin{enumerate}

\item[]

\item\label{theorem A}
The Hochschild homology functor lifts:
\[
\begin{tikzcd}
&
{\sf Cyc}^\unst(\cX)
\arrow{d}{{\rm forget}}
\\
\fCat_1[\cX]
\arrow{r}[swap]{\sHH}
\arrow[dashed]{ru}
&
\cX
\end{tikzcd}
~.
\]
In other words, for each category-object $\cC$ in $\cX$, its Hochschild homology $\sHH(\cC)$ admits the structure of a non-stable cyclotomic object in $\cX$;
this structure is functorial in $\cC \in \fCat_1[\cX]$.

\item\label{theorem B}
There is a canonical natural transformation $\Obj \to \sHH$ between functors $\fCat_1[\cX] \to \cX$, which is invariant with respect to the above non-stable cyclotomic structure on $\sHH$:
\[
\Obj
\longrightarrow
\sHH^{\sf Cyc}
~.
\]

\end{enumerate}

\end{theorem}

\begin{proof}
Denote by $\langle \DD^0 , \SS^1 \rangle \subset \M$ the full $\infty$-subcategory consisting of the objects $\DD^0 , \SS^1 \in \M$.  
Note that the functor $\int_{\SS^1} \colon \fCat_1[\cX] \to \cX$ factors:
\begin{eqnarray*}
\int_{\SS^1}
\colon
\fCat_1[\cX]
&
\xra{~\int~}
&
\Fun(\M , \cX)
\\
&
\xra{~\rm restriction~}
&
\Fun\bigl(
\langle \DD^0 , \SS^1 \rangle
,
\cX
\bigr)
\\
&
\xra{~\rm restriction~}
&
\Fun\bigl(
\fB \End_{\M}(\SS^1)
,
\cX
\bigr)
=:
\Mod_{\End_{\M}(\SS^1)}(\cX)
\\
&
\xra{~\ev_{\SS^1}~}
&
\cX
~.
\end{eqnarray*}
\Cref{t51} specializes as an identification $(\BW^{\op})^{\lcone} \simeq \langle \DD^0 , \SS^1 \rangle$ under an identification $\WW^{\op} \simeq \End_{\M}(\SS^1)$.
Using this, \Cref{t44} gives an identification $\Mod_{\End_{\M}(\SS^1)}(\cX) \simeq \Mod_{\WW^{\op}}(\cX) \simeq {\sf Cyc}^\unst(\cX)$.
The first statement then follows from identification $\int_{\SS^1}(\cC)  \simeq \sHH(\cC)$ of \Cref{r10}(3) for each $\cC \in \fCat_1[\cX]$.

The second statement then follows upon observing that the $\infty$-category $\Fun( (\BW^{\op})^{\lcone} , \cX)$ over $\Fun(\BW^{\op},\cX) =: \Mod_{\WW^{\op}}(\cX)$ classifies a $\WW^{\op}$-module in $\cX$ equipped with a $\WW^{\op}$-invariant map to it.

\end{proof}

\begin{remark}
\label{r12}
\Cref{t73}(2) can be interpreted as a non-stable cyclotomic trace map.

\end{remark}

\appendix

\section{Non-stable cyclotomic objects}
\label{sec.A}

Here we introduce the notion of a (\bit{non-stable}) \bit{cyclotomic object} in $\cX$.  
The term \bit{non-stable} is used here to reflect that $\cX$ is not assumed to be a stable $\infty$-category, and that its symmetric monoidal structure is understood as the Cartesian one.  
The notion of a \bit{cyclotomic object} in $\Spectra$ is developed in~\cite{BluMan}.  
The work~\cite{cyclo} studies cyclotomic objects in some generality, and in particular explains how $\Spaces \xra{\Sigma^\infty_+} \Spectra$ carries (non-stable) cyclotomic objects to (stable) cyclotomic objects.  

In this section, we introduce (\bit{non-stable}) \bit{cyclotomic objects} in $\cX$ and establish a few equivalent definitions of such (\Cref{t44}).

\begin{definition}\label{d17}
\begin{itemize}
\item[]

\item
The poset $\Ndiv$ is that of natural numbers with partial order given by divisibility: the relation $r \leq s$ in $\Ndiv$ means $r$ divides $s$. We also denote this as $r|s$.

\item
The \bit{proper orbit category} (of $\TT$) is the $\infty$-category $\Orbit_{\TT}^{<}$ of transitive $\TT$-spaces with isotropy a proper (equivalently, finite) subgroup of $\TT$, and $\TT$-equivariant maps between them. 

\item
For $\cX$ an $\infty$-category, the $\infty$-category of \bit{proper-genuine $\TT$-modules (in $\cX$)} is
\[
\Mod_{\TT}^{\gen^<}(\cX)
~:=~
\Fun\bigl(
({\Orbit}_{\TT}^{<})^{\op}
,
\cX
\bigr)
~.
\]

\end{itemize}

\end{definition}

The action $\NN^\times \underset{(\ref{e45})}\lacts \TT$ as a topological group determines an action on the proper orbit $\infty$-category:
\begin{equation}
\label{e44}
\NN^\times \underset{\rm Obs~\ref{t43}}\simeq (\NN^\times)^{\op} \lacts {\Orbit}^{<}_{\TT}
~,\qquad
n\cdot (\TT \lacts T)
:=
(\TT \xra{z\mapsto z^n} \TT \lacts T)
~.
\end{equation}
Precomposition by this $\NN^\times$-action~\Cref{e44} defines an $\NN^\times = (\NN^\times)^{\op}$-action on the $\infty$-category $\Mod_{\TT}^{\gen^<}(\cX)$.

\begin{definition}
\label{d18}
The $\infty$-category of \bit{non-stable cyclotomic objects} in an $\infty$-category $\cX$ is that of the $\NN^\times$-invariant proper-genuine $\TT$-modules:
\[
{\sf Cyc}^\unst(\cX)
~:=~
\Mod_{\TT}^{\gen^<}(\cX)^{{\sf h}\NN^\times}
~.
\]

\end{definition}

\begin{remark}
Informally, a non-stable cyclotomic object in $\cX$ consists of the following.
\begin{itemize}
\item
A $\TT$-module $\left( \TT \underset{\alpha}\lacts X \right)$ in $\cX$~.

\item
For each $r\in \NN^\times$, a morphism between $\TT$-modules in $\cX$:
\[
r^\ast 
\left(\TT \underset{ \alpha} \lacts X \right)
~:=~
\left(\TT \xra{z\mapsto z^r} \TT \underset{ \alpha} \lacts X \right)
\xra{~c_{r}~}
\left(\TT \underset{ \alpha} \lacts X \right)
~.
\]

\item
For each pair $s,r\in \NN^\times$, 
a 2-cell witnessing commutativity among $\TT$-modules in $\cX$:
\[
\begin{tikzcd}
r^\ast 
s^\ast 
\left(\TT \underset{\alpha} \lacts X \right)
\arrow{r}{r^\ast c_{s}}
\arrow{d}[sloped, anchor=north]{\sim}
&
r^\ast \left(\TT \underset{\alpha} \lacts X \right)
\arrow{d}{c_{r}}
\\
(sr)^\ast
\left(\TT \underset{\alpha} \lacts X \right)
\arrow{r}[swap]{c_{sr}}
&
\left(\TT \underset{ \alpha} \lacts X \right)
\end{tikzcd}
~.
\]

\item
For each triple $r,s,t \in \NN^\times$, a similar commutative cube among $\TT$-modules in $\cX$ whose faces are (possibly pulled back from) the above commutative squares.  

\item
Etcetera.

\end{itemize}

\end{remark}

Restriction along ${\Orbit}_{\TT}^{<} \xra{!} \ast$, which is evidently $\NN^\times$-invariant, defines a functor
\[
{\sf triv}
\colon
\cX
\longrightarrow
{\sf Cyc}^\unst(\cX)
~.
\]

\begin{definition}
\label{d22}
The \bit{cyclotomic fixed points} functor is the right adjoint to ${\sf triv}$:
\[
(-)^{\sf Cyc}
\colon
{\sf Cyc}^\unst(\cX)
\longrightarrow
\cX
~.
\]

\end{definition}

\begin{observation}
\label{t46}
There is a functor
\[
{\Orbit}_{\TT}^{<}
\longrightarrow
\NN^{\sf div}
~,\qquad
(\TT \lacts T)
\mapsto | \TT_t |
\qquad
(~\text{ for some } t\in T~ )
~,
\]
whose value on a transitive $\TT$-space $T$ with proper isotropy is the order of the isotropy of some element in $T$.
Exploiting that the codomain of this functor is a poset, this functor is unique with the named values on objects.  
This functor has the following properties.
\begin{enumerate}

\item
Two subgroups of $\TT$ with the same cardinality are identical.
Therefore, the fiber of this functor over $r\in \NN^{\sf div}$ is the full $\infty$-subcategory of ${\Orbit}_{\TT}^{<}$ on $\frac{\TT}{\sC_r}$.  
It follows that the fiber of this functor over $r\in \NN^{\sf div}$ is the $\infty$-groupoid 
$
\sB \Bigl( 
\frac{\TT}{\sC_r}
\Bigr)\simeq \BT
$.
In particular, this functor is conservative.

\item
The space of morphisms in ${\Orbit}_{\TT}^{<}$ over a morphism $r | s$ in $\NN^{\sf div}$ is 
\[
_{\frac{\TT}{\sC_r}\backslash}
{\sf Map}^{\TT}\Bigl(\frac{\TT}{\sC_r} , \frac{\TT}{\sC_s} \Bigr)_{/\frac{\TT}{\sC_s}}
~,
\]
with the source-map an equivalence.
In particular, this functor is a left fibration.

\item
The straightening of this left fibration is the functor
\[
\sB\Bigl(
\frac{\TT}{\sC_\bullet}
\Bigr)
\colon
\NN^{\sf div}
\longrightarrow
\Spaces
\]
characterized by the following values on objects and generating morphisms:
\begin{itemize}

\item
the value of this functor on each $r\in \NN^{\sf div}$ is the space $\BT \simeq \sB\Bigl(
\frac{\TT}{\sC_r}
\Bigr)
$;

\item
the value of this functor on each morphism $(r | s)$ in $\NN^{\sf div}$ is the map $\BT \xra{\sB (z\mapsto z^{\frac{s}{r}})} \BT$.

\end{itemize}

\item
With respect to the $(\NN^\times)^{\op}$-action on the poset $\NN^{\sf div}$ given by $r^\circ \cdot d := dr$, 
this functor is canonically $(\NN^\times)^{\op}$-equivariant.

\end{enumerate}

\end{observation}

Here is the main technical result in this subsection.
\begin{lemma}
\label{t68}
There is a canonical identification of the $\infty$-category of right-lax coinvariants with respect to the action~\Cref{e44}:
\[
\Bigl(
{\Orbit}^{<}_{\TT}
\Bigr)_{{\sf r.lax}\NN^\times}
\xra{~\simeq~}
\Ar(\BW^{\op})
~,
\]
where the codomain is regarded as a Cartesian fibration over $\BN$ via the composite functor
\begin{equation}
\label{e59}
\Ar(\BW^{\op})
\xra{\ev_s}
\BW^{\op}
\xra{\fB {\sf proj}}
\BN
~.
\end{equation}

\end{lemma}

\begin{proof}

Observe the unique functor between categories,
\begin{equation}
\label{e60}
(\BN)^{\ast/}
\longrightarrow
\NN^{\sf div}
~,\qquad
(\ast \xra{d} \ast)
\mapsto
d
~,
\end{equation}
whose value on each object is as depicted.  
Using that, for each $d\in \NN^\times$, the map $\NN^\times \xra{r\mapsto dr} \NN^\times$ is a monomorphism between spaces, the functor~\Cref{e60} is an equivalence.

Now, consider the diagram among $\infty$-categories:
\begin{equation}
\label{e61}
\begin{tikzcd}
\Ar(\BW^{\op})^{|\BT}
\arrow{rr}
\arrow{rd}
\arrow{dd}
&&
\BT
\arrow{dd}
\arrow{dr}
&
\\
&
\Ar(\BW^{\op})
\arrow[crossing over]{rr}[pos=0.3]{\ev_s}
&&
\BW^{\op}
\arrow{dd}{\fB {\sf proj}}
\\
(\BN)^{\ast/}
\arrow{rr}
\arrow{dr}
&&
\ast
\arrow{dr}
&
\\
&
\Ar(\BN)
\arrow{rr}[swap]{\ev_s}
\arrow[leftarrow, crossing over]{uu}[swap, pos=0.3]{\Ar(\fB {\sf proj})}
&&
\BN
\end{tikzcd}
~.
\end{equation}
By definition of $\infty$-undercategories, the bottom square is a pullback square.
The definition of the monoid $\WW$ is such that the right square is also a pullback.
It follows that the left square is a pullback.
Because $\BW^{\op} \xra{\fB {\sf proj}} \BN$ is a right fibration, so is $\Ar(\BW^{\op}) \xra{\Ar(\fB {\sf proj})} \Ar(\BN)$.
We conclude that the functor $\Ar(\BW^{\op})^{|\BT} \to (\BN)^{\ast/}$ is a right fibration.  
By direct inspection, the straightening of this right fibration is the functor
\[
(\NN^{\sf div})^{\op}
~
\underset{\Cref{e60}}\simeq
~
\bigl(
(\BN)^{\ast/}
\bigr)^{\op}
\longrightarrow
\Spaces
\]
characterized by the following values on objects and generating morphisms:
\begin{itemize}

\item
the value of this functor on each $d\in \NN^{\sf div}$ is the space $\BT \simeq \sB\Bigl(
\frac{\TT}{\sC_d}
\Bigr)
$;

\item
the value of this functor on each morphism $(d | k)$ in $\NN^{\sf div}$ is the map $\BT \xra{\sB (z\mapsto z^{\frac{k}{d}})} \BT$.

\end{itemize}
By \Cref{t46}, there results an equivalence over $\NN^{\sf div} \simeq (\BN)^{\ast/}$:
\begin{equation}
\label{e62}
{\Orbit}^{<}_{\TT}
~\simeq~
\Ar(\BW^{\op})^{|\BT}
~.
\end{equation}
By direct inspection, this equivalence~\Cref{e62} is canonically $(\NN^\times)^{\op}$-equivariant.
This, in turn, lends to a canonical equivalence between Cartesian fibrations over $\BN \simeq (\BN)^{\op}$:
\[
\Bigl(
{\Orbit}^{<}_{\TT}
\Bigr)_{{\sf r.lax} (\NN^\times)^{\op}}
~\simeq~
\Bigl(
\Ar(\BW^{\op})^{|\BT}
\Bigr)_{{\sf r.lax} (\NN^\times)^{\op}}
~.
\]
The diagram~\Cref{e61} witnesses an identification 
$
\Bigl(
\Ar(\BW^{\op})^{|\BT}
\Bigr)_{{\sf r.lax} (\NN^\times)^{\op}}
\simeq
\Ar(\BW^{\op})
$
as Cartesian fibrations over $\BN$.

\end{proof}

\begin{cor}
\label{t47}
There is a canonical identification of the $\infty$-category of coinvariants with respect to the action~\Cref{e44}:
\[
\Bigl(
{\Orbit}^{<}_{\TT}
\Bigr)_{{\sf h}\NN^\times}
\xra{~\simeq~}
\BW^{\op}
~.
\]

\end{cor}

\begin{proof}
Through \Cref{t68}, the corollary follows upon showing the functor
\[
\Ar(\BW^{\op})
\xra{\ev_t}
\BW^{\op}
\]
witnesses a localization on those morphisms in $\Ar(\BW^{\op})$ that are Cartesian with respect to the composite functor~\Cref{e59}.
Certainly, this functor witnesses a localization on those morphisms in $\Ar(\BW^{\op})$ that are Cartesian with respect to the functor $\Ar(\BW^{\op}) \xra{\ev_s} \BW^{\op}$.
Using that $\BT$ is an $\infty$-groupoid, thereby implying $\BW^{\op} \to \BN$ is conservative, this class of morphisms in $\Ar(\BW^{\op})$ is precisely the class of morphisms that are Cartesian with respect to the composite functor~\Cref{e59}.

\end{proof}

\begin{cor}
\label{t44}
Let $\cX$ be an $\infty$-category.
There are canonical equivalences among $\infty$-categories
\[
\Mod_{\TT}(\cX)^{{\sf r.lax}\NN^\times}
~\simeq~
\Mod_{\WW^{\op}}(\cX)
~\simeq~
\Mod_{\TT}^{\gen^<}(\cX)^{{\sf h}\NN^\times}
~=:~
{\sf Cyc}^\unst(\cX)
~,
\]
Furthermore, the equivalence $\Fun(\BW^{\op},\cX) =: \Mod_{\WW^{\op}}(\cX) \simeq {\sf Cyc}^\unst(\cX)$ extends as a canonically commutative diagram:
\[
\begin{tikzcd}[row sep=1.5cm]
\Fun(\BW^{\op},\cX)
\arrow{rd}[sloped, swap]{\lim}
\arrow{r}{\sim}
&
\Mod_{\WW^{\op}}(\cX) 
\arrow{d}{(-)^{{\sf h} \WW^{\op}}}
\arrow{r}{\sim}
&
{\sf Cyc}^\unst(\cX)
\arrow{ld}[sloped, swap]{(-)^{\sf Cyc}}
\\
&
\cX
\end{tikzcd}
~.
\]
\end{cor}

\begin{proof}
The proof of the first statement is complete upon explaining the following sequences of equivalences among $\infty$-categories:
\begin{eqnarray*}
\Mod_{\TT}(\cX)^{{\sf r.lax}\NN^\times}
~\simeq~
\Fun\bigl( \BT , \cX \bigr)^{{\sf r.lax} \NN^\times}
&
~\simeq~
&
\Fun\bigl( \BT_{{\sf r.lax} \NN^{\times}} , \cX \bigr)
\\
&
~\simeq~
&
\Fun\bigl( \BW^{\op}, \cX \bigr)
~\simeq~
\Mod_{\WW^{\op}}(\cX)
\\
&
~\simeq~
&
\Fun\bigl( 
(
{\Orbit}^{<}_{\TT}
)_{{\sf h}\NN^\times}
,
\cX
\bigr)
\\
&
~\simeq~
&
\Fun\bigl( 
{\Orbit}^{<}_{\TT}
,
\cX
\bigr)^{{\sf h}\NN^\times}
~\simeq~
\Mod_{\TT}^{\gen^<}(\cX)^{{\sf h}\NN^\times}
~.
\end{eqnarray*}
The three equivalences that are not centered follow from the definition of $\Mod_-(\cX)$, and the \Cref{d17} of $\Mod_{\TT}^{\gen^<}(\cX)$.
It remains to prove the four aligned equivalences.
For the first aligned equivalence, recall that the $\NN^\times$-action on $\Mod_{\TT}(\cX)$ is pre-composition of the $\NN^\times$-action on $\TT$.  
So the right-lax invariants by this $\NN^\times$-action on $\Mod_{\TT}(\cX)$ is functors from the right-lax coinvariants, which explains the first centered equivalence.  
For the second aligned equivalence, the definition of the monoid $\WW := \TT \rtimes \NN^\times$ implies 
$\BW^{\op} \simeq \BT_{{\sf r.lax}\NN^\times}$.  
The third aligned equivalence is $\Fun\bigl( - , \cX\bigr)$ applied to \Cref{t47}.
The fourth aligned equivalence identifies functors from $\NN^\times$-coinvariants as 
$\NN^\times$-invariants of functors.  

The second statement follows upon observing that the functor $\cX \to \Mod_{\WW^{\op}}(\cX)$ given by restriction along $\BW^{\op} \to \ast$ is identified through the above sequence of equivalences with the functor ${\sf triv}\colon \cX \to {\sf Cyc}^\unst(\cX)$, then using that right adjoints are unique.

\end{proof}

\begin{remark}
A proper-genuine $\TT$-module in an $\infty$-category $\cX$ consists of a considerable amount of homotopy coherence data, and the structure of being $\NN^\times$-invariant consists of yet more.
Hence, one might expect it to be impractical to explicitly construct an object in $\Mod_{\TT}^{\gen^<}(\cX)^{{\sf h}\NN^\times}$.
To the contrary, \Cref{t44} states that the requisite homotopy coherence data actual cancel each other out, in a certain sense: an $\NN^\times$-invariant proper-genuine $\TT$-module in $\cX$ is simply a $\WW^{\op}$-module in $\cX$ (which entails substantially less homotopy coherence data).
\end{remark}

\section{Agreement of factorization homologies}
\label{sec.facts}

The main result in this section is \Cref{t90}, which articulates a precise sense in which combinatorial factorization homology as in \Cref{d21} agrees with the geometric version of factorization homology, as defined in~\cite{AFR2}.  
To make these two constructions comparable, we first identify $\Quiv^{\op}$ with $\Disk^{\sofr}$ and $\M$ with $\Mfld^{\sofr}$.

\subsection{Recollections from other works}
\label{sec.recollections}

We summarize some notions from~\cite{AFR1} \&~\cite{AFR2}.

\begin{enumerate}
\item
Constructed in~\S6.3 of~\cite{AFR1} is the $\infty$-category $\cBun$, which classifies proper constructible bundles between stratified spaces.  
This is to say, for $K$ a stratified space, the moduli space of proper constructible bundles over $K$ is identical with the space of functors to $\cBun$ from its exit-path $\infty$-category, $\Exit(K) \to \cBun$.
So, an object in $\cBun$ is a compact stratified space; 
a morphism from $X_0$ to $X_1$ is a proper constructible bundle $X \to \Delta^1$ (where the codomain is understood with the two strata $\Delta^{\{0\}}$ and $\Delta^1 \setminus \Delta^{\{0\}}$) equipped with identifications $X_0 \cong X_{|\Delta^{\{0\}}}$ and $X_1 \cong X_{|\Delta^{\{1\}}}$.

\item
Constructed in~\S6.4 of~\cite{AFR1} is the $\infty$-category $\cExit$, equipped with a functor $\cExit \to \cBun$.
For $K$ a stratified space, and for $\Exit(K) \xra{\langle X \to K \rangle} \cBun$ classifying a proper constructible bundle, there is a canonical identification of the base-change $\cExit_{|\Exit(K)} \simeq \Exit(X)$ over $\Exit(K)$.
In particular, the fiber of $\cExit \to \cBun$ over $K\in \cBun$ is the exit-path $\infty$-category $\Exit(K)$.

\item
Introduced in~\S2.1.2 of~\cite{AFR2} is the $\infty$-category $\cVect^{\sf inj}$.
An object in $\cVect^{\sf inj}$ is a finite-dimensional $\RR$-vector space; the space of morphisms from $V$ to $W$ is the Stiefel space of injections from $V$ to $W$.

\item
Constructed in~\S2.1.3\&2.1.4 of~\cite{AFR2} is a functor $\cExit \xra{\tau} \cVect^{\sf inj}$.
Its value on $x\in \Exit(X)$ is the $\RR$-vector space $\sT_x X$, which is the tangent space at $x$ of the stratum of $X$ in which $x\in X$ belongs.  
More generally, for $X \to K$ a proper constructible bundle, the resulting composite functor $\Exit(X) \to \cExit \xra{\tau} \cVect^{\sf inj}$ evaluates on $x\in \Exit(X)$ as the vertical tangent space $\sT^{\sf fib}_x X$ at $x$ of the constructible bundle $X \to K$.

\item
Defined in~\S2.4 of~\cite{AFR2} is the $\infty$-category
\[
\cMfd_1^{\sf sfr}
~,
\]
which is characterized by declaring, for $K$ a stratified space, the datum of a functor $\Exit(K) \to \cMfd_1^{\sf sfr}$ to be that of a proper constructible bundle $X \to K$ equipped with a (\bit{fiberwise}) \bit{solid 1-framing}, which is a lift
\[
\begin{tikzcd}
&
\cVect^{\sf inj}_{/\RR^1}
\arrow{d}{\rm forget}
\\
\Exit(X)
\arrow{r}[swap]{\sT^{\sf fib} X}
\arrow[dashed]{ru}
&
\cVect^{\sf inj}
\end{tikzcd}
~.
\]
Forgetting the solid $1$-framing defines a functor
\[
\cMfd_1^{\sf sfr}
\longrightarrow
\cBun
~.
\]
So, an object $M \in \cMfd_1^{\sf sfr}$, termed a \bit{solidly $1$-framed stratified space}, is a compact stratified space $X$ equipped with an injection $\sT_x X \overset{\varphi}\hookrightarrow \RR^n$ for each $x\in X$ compatibly.  
In particular, for $M =(X,\varphi) \in \cMfd_1^{\sf sfr}$ an object, the dimension of each stratum of $X$ is bounded above by $1$.
Consequently, an object in $\cMfd_1^{\sf sfr}$ is a finite disjoint union of oriented connected graphs and oriented circles.

The $\infty$-category $\cMfd_1^{\sf sfr}$ admits finite products, which are given by disjoint unions of solidly 1-framed stratified spaces.
Keeping with \Cref{d40}, for $M,N \in \cMfd_1^{\sf sfr}$, we denote their categorical product in $\cMfd_1^{\sf sfr}$ as $M \sqcup N \in \cMfd_1^{\sf sfr}$.

\item
Constructed in~\S1.4 of~\cite{AFR1} are surjective monomorphisms between $\infty$-categories,
\[
{\sf Cylr}
\colon
({\sf c}\cS{\sf trat}^{\sf p.cbl})^{\op}
\longrightarrow
\cBun
\longleftarrow
{\sf c}\cS{\sf trat}^{\sf ref}
\colon
{\sf Cylo}
~,
\]
where ${\sf c}\cS{\sf trat}^{\sf p.cbl}$ is an $\infty$-category in which an object is a compact stratified space and a morphism is a proper constructible bundle, and where ${\sf c}\cS{\sf trat}^{\sf ref}$ is an $\infty$-category in which an object is a compact stratified space and a morphism is a refinement.
Base-change of these monomorphisms along the forgetful functor $\cMfd_n^{\sf sfr} \to \cBun$ define $\infty$-subcategories
\[
\cMfd_n^{\sf sfr,cls}
~,~
\cMfd_n^{\sf sfr,cr}
~\subset~
\cMfd_n^{\sf sfr, idl}
~\subset~
\cMfd_n^{\sf sfr}
~\supset~
\cMfd_n^{\sf sfr, ref}
\]
consisting, respectively, of the images under ${\sf Cylr}$ of the proper constructible embeddings, of the surjective proper constructible bundles, of the proper constructible bundles, and the image under ${\sf Cylo}$ of the refinements.

\item
Introduced in~\S3.4 of~\cite{AFR2} is the notion of a \bit{closed cover}, which is a diagram in $\cBun$ that is the image under ${\sf Cylr}$ of (the opposite of) a finite colimit diagram in ${\sf c}\cS{\sf trat}^{\sf p.emb} \subset {\sf c}\cS{\sf trat}^{\sf p.cbl}$, the $\infty$-subcategory of proper constructible embeddings.
A \bit{closed cover} in $\cMfd_n^{\sf sfr}$ is a diagram that lies over a closed cover in $\cBun$.
Closed covers are, in particular, limit diagrams.

\item
\label{recollection.eight}
Defined in~\S3.5 of~\cite{AFR2} is the full $\infty$-subcategory 
\[
\cDisk_1^{\sf sfr}
~\subset~
\cMfd_1^{\sf sfr}
~,
\]
which is the smallest full $\infty$-subcategory that is closed under the formation of closed covers and that contains the closed disks $\DD^0$ and $\DD^1$ as they are endowed with their standard solid 1-framings.
In particular, an object in $\cDisk_1^{\sf sfr}$ is a finite oriented graph.  
Consequently, $\Exit(D)$ is a finite poset of depth 1.
For each object $D\in \cDisk_1^{\sf sfr}$, assigning to each point $d\in D$ the closure $\ol{D_d}$ of the stratum in $D$ containing $d\in D$ defines a functor 
\begin{equation}
\label{t79}
\Exit(D)^{\op} \longrightarrow (\cDisk_1^{\sf sfr})^{D/}
~,\qquad
(d\in D)
\mapsto
(D \xra{\sf cls} \ol{D_d})
~.
\end{equation}
The adjoint of this functor 
$(\Exit(D)^{\rcone})^{\op} \longrightarrow \cDisk_1^{\sf sfr}$
is a closed cover of $D$.
Lastly, contractibility of the space $\Aut_{\cDisk_1^{\sf sfr}}(\DD^1) \simeq \Diff^{\sf fr}(\DD^1)$ ultimately implies $\cDisk_1^{\sf sfr}$ is an ordinary category.

\item
\label{recollection.nine}
Constructed in~\S3.8 of~\cite{AFR2} is the \bit{cellular realization} functor
\[
\langle - \rangle
\colon
\bDelta^{\op}
\longrightarrow
\cDisk_1^{\sf sfr}
~,
\]
whose value on $[p]$ is $\DD^0$ if $p=0$ and is $\DD^1$ if $p=1$ and for $p > 1$ is a refinement of $\DD^1$ with a total of $p+1$ 0-dimensional strata two of which are the boundary points.
Lemma~3.51 of~\cite{AFR2} proves that the cellular realization functor $\langle - \rangle$ is fully faithful, and carries (the opposites of) Segal diagrams to closed covers.

Consider the restricted Yoneda functor along $\langle - \rangle$:
\[
(\cDisk_1^{\sf sfr})^{\op}
\longrightarrow
\PShv(\bDelta)
~,\qquad
D
\longmapsto
\Hom_{\cDisk_1^{\sf sfr}}( D , \langle - \rangle)
~.
\]
The above implies this functor is fully faithful, and factors through $\Cat_{(\infty,1)} \subset \PShv(\bDelta)$,
\[
\fC
\colon
(\cDisk_1^{\sf sfr})^{\op}
\longrightarrow
\Cat_{(\infty,1)}
~,
\]
such that the diagram among $\infty$-categories
\begin{equation}
\label{t82}
\begin{tikzcd}
\bDelta
\arrow{d}[swap]{\langle - \rangle}
\arrow{rd}[sloped]{\rm standard}
\\
(\Disk^{{\sf con},{\sofr}})^{\op}
\arrow{r}[swap]{\fC}
&
\Cat_{(\infty,1)}
\end{tikzcd}
\end{equation}
canonically commutes.

\item
\bit{Factorization homology} is defined in~\S4.3 of~\cite{AFR2} as the composite functor
\[
\int'
\colon
\Cat_{(\infty,1)}
\xra{ \langle - \rangle_\ast}
\Fun(\cDisk_1^{\sf sfr} , \Spaces)
\xra{ \iota_!}
\Fun(\cMfd_1^{\sf sfr} , \Spaces)
~,\qquad
\cC
\longmapsto
\bigl(
M
\mapsto \int'_M \cC
\bigr)
~,
\]
given by is right Kan extension along the cellular realization functor $\langle - \rangle$ followed by left Kan extension along the fully faithful inclusion $\cDisk_1^{\sf sfr} \xra{\iota} \cMfd_1^{\sf sfr}$.  
Explicitly, for $\cC$ an $(\infty,1)$-category, and for $M$ a solidly 1-framed stratified space, the factorization homology of $\cC$ over $M$ is the colimit
\[
\int'_M \cC
~\simeq~
\colim\Bigl(
\cDisk_{1/M}^{\sf sfr}
\xra{\rm forget}
\cDisk_1^{\sf sfr}
\xra{ D\mapsto \Hom_{\Cat_{(\infty,1)}}\bigl( \fC(D) , \cC\bigr) }
\Spaces
\Bigr)
~.
\]

\end{enumerate}

\begin{notation}
\label{d41}
In~\cite{AFR2}, factorization homology is simply denoted as $\int_M \cC$, without the ${~}'$.  
In this section, we use the notation $\int'_M \cC$ for geometric factorization homology defined in~\cite{AFR2} in effort to distinguish it from the combinatorial version of factorization homology (\Cref{d21}) defined in the body of this work.
The main theorem of this section (\Cref{t90}) articulates a precise sense in which that these two versions of factorization homology agree.
Therefore, this notation $\int'_M \cC$, in place of $\int_M \cC$, can be understood as temporary.

\end{notation}

\begin{definition}
Let $M\in \cMfd_1^{\sf sfr}$ be a solidly 1-framed stratified space.
The full $\infty$-subcategory 
\[
\bcD{\sf isk}(M)
~\subseteq~
\cDisk^{\sf sfr}_{1/M}
~:=~
{\cDisk}^{\sf sfr}_1
\underset{
\cMfd_1^{\sf sfr}
}
\times
\cMfd^{\sf sfr}_{1/M}
\]
consists of those objects $(D \to M)$ that are refinements.  

\end{definition}

We record a technical result concerning the $\infty$-category $\cMfd^{\sf sfr}_{1}$.
\begin{lemma}
\label{t88}
For each object $M\in \Mfld^{\sofr}$, the canonical functor
\[
\bcD{\sf isk}(M)
\longrightarrow
\cDisk^{\sf sfr}_{1/M}
\]
is final.

\end{lemma}

\begin{proof}
We first reduce to the case in which $M$ is connected.  
Suppose $M \cong M_- \sqcup M_+$ is a product in $\cMfd^{\sf sfr}_{1}$.
Taking products in $\cMfd^{\sf sfr}_{1}$ defines the bottom horizontal functor in the solid diagram among $\infty$-categories:
\[
\begin{tikzcd}
\Disk(M_-) \times \Disk(M_+)
\arrow[dashed]{r}
\arrow{d}
&
\Disk(M_- \sqcup M_+)
\arrow{d}
\\
\cDisk^{\sf sfr}_{1/M_-}
\times
\cDisk^{\sf sfr}_{1/M_+}
\arrow{r}[swap]{\sqcup}
&
\cDisk^{\sf sfr}_{1/M_- \sqcup M_+}
\end{tikzcd}
~.
\]
Refinement morphisms in $\cMfd^{\sf sfr}_{1}$ are refinement morphisms of underlying stratified spaces.
Products in $\cMfd^{\sf sfr}_{1}$ are disjoint unions of underlying stratified spaces.  
So products of refinement morphisms are again refinement morphisms.  
This observation supplies the filler, which is necessarily unique, in the above diagram among $\infty$-categories. 
Now, for $(D \xra{\sf ref} M_- \sqcup M_+) \in \Disk(M_- \sqcup M_+)$, consider the active factors $(D_\pm \to M_\pm)$ of the composites $D \xra{\sf ref} M_- \sqcup M_+ \xra{\pr} M_\pm$.
In terms of stratified spaces, $D_\pm = (M_- \sqcup M_+) \cap M_\pm$ is the intersection of the refinement $D$ of $M_- \sqcup M_+$ with the union of components $M_\pm \subseteq M_- \sqcup M_+$.
In particular, the active morphisms $(D_\pm \to M_\pm)$ are, in fact, refinements.  
The assignment $(D \to M_- \sqcup M_+)\mapsto \bigl( (D_- \to M_-) , (D_+ \to M_+) \bigr)$ is an inverse to the top horizontal dashed functor above.  
In particular, the top horizontal functor is final. 
Next, because $\Disk^{\sofr}$ has finite products, the bottom horizontal functor is a right adjoint.
Therefore, finality of the left vertical functor implies finality of the right vertical functor.
Using that every object in $\cMfd^{\sf sfr}_{1}$ is a finite product of connected objects, the lemma is implied by its case in which $M$ is connected.

So assume $M$ is connected.
Then $M \in \cDisk^{\sf sfr}_{1}$ or $M = \SS^1$.  
In the former case, the identity morphism determines a final object in both $\bcD{\sf isk}(M)$ and $\cDisk^{\sf sfr}_{1/M}$, which is preserved by the inclusion.  
Thus, it remains to consider the case that $M = \SS^1$.

Now by Quillen's Theorem A, it suffices to show that for any object $(D \ra \SS^1) \in \cDisk^{\sf sfr}_{1/\SS^1}$, the $\infty$-groupoid completion of the $\infty$-category
\begin{equation}
\label{slice.category.from.D.of.to.D.over.under.a.chosen.object}
\Disk(\SS^1)
\underset{\cDisk^{\sf sfr}_{1/\SS^1}}{\times}
(\cDisk^{\sf sfr}_{1/\SS^1})^{(D \ra \SS^1)/}
\end{equation}
of factorizations
\begin{equation}
\label{arbitrary.object.of.slice.under.R.to.Sone.in.D.of}
\begin{tikzcd}
D
\arrow{rr}
\arrow[dashed]{rd}
&
&
\SS^1
\\
&
D'
\arrow[dashed]{ru}[sloped,swap,pos=0.3]{\refi}
\end{tikzcd}
\end{equation}
(where $D' \in \D$) is contractible.  Observe that in diagram \Cref{arbitrary.object.of.slice.under.R.to.Sone.in.D.of}, the closed-active factorization of the downwards morphism $D \ra D'$ must compose to the closed-active factorization
$
D
\xra{\cls}
D_0
\xra{\act}
\SS^1
$
of the (chosen) horizontal morphism, since such factorizations are unique.  So taking closed-active factorizations defines an equivalence between the category \Cref{slice.category.from.D.of.to.D.over.under.a.chosen.object} and the category
\begin{equation}
\label{slice.under.R.zero.by.factorizn.system}
\Disk(\SS^1)
\underset{\cDisk^{\sf sfr}_{1/\SS^1}}{\times}
(\cDisk^{\sf sfr}_{1/\SS^1})^{(D_0 \ra \SS^1)/^\act}
\end{equation}
of factorizations
\[ \begin{tikzcd}
D
\arrow{rr}
\arrow{d}[swap]{\cls}
&
&
\SS^1
\\
D_0
\arrow{rru}[sloped,pos=0.55]{\act}
\arrow[dashed]{r}[swap]{\act}
&
D'
\arrow[dashed]{ru}[sloped,swap,pos=0.3]{\refi}
\end{tikzcd} \]
(where $D' \in \D$).  Thus, it suffices to show that the groupoid completion of the category \Cref{slice.under.R.zero.by.factorizn.system} is contractible.

To check that the groupoid completion $|\Cref{slice.under.R.zero.by.factorizn.system}|$ is contractible, it suffices to check that its homotopy groups vanish, and for this it suffices to show that any map to it from a sphere is freely nullhomotopic.  We check this using the theory of stratified spaces.

Observe first that any map
\begin{equation}
\label{d.minus.one.sphere.in.gpd.compln.of.double.slice}
S^{d-1}
\longra
|\Cref{slice.under.R.zero.by.factorizn.system}|
\end{equation}
of $\infty$-groupoids is represented by a functor
\[
\Exit(S^{d-1})
\xlongra{\tilde{X}}
\Cref{slice.under.R.zero.by.factorizn.system}
\]
of $\infty$-categories, where we abuse notation by also writing $S^{d-1}$ for the $(d-1)$-sphere (thought of as a manifold) equipped with some stratification (e.g.\! a triangulation).\footnote{This can be seen through the model of $(\infty,1)$-categories as quasi-categories.
There, the $\infty$-groupoid completion appearing as the codomain of \Cref{d.minus.one.sphere.in.gpd.compln.of.double.slice} can be presented as ${\sf Ex}^\infty$ applied to the quasi-category \Cref{slice.under.R.zero.by.factorizn.system}, which is the directed colimit of finite-fold iterations of Kan's ${\sf Ex}$ functor.  
Meanwhile, the domain of \Cref{d.minus.one.sphere.in.gpd.compln.of.double.slice} can be presented by the simplicial set $\SS^{d-1}$ associated to any triangulation of the $(d-1)$-sphere. 
Using that such a triangulation is finite, the domain of \Cref{d.minus.one.sphere.in.gpd.compln.of.double.slice} is a finite simplicial set.
Therefore, such a map \Cref{d.minus.one.sphere.in.gpd.compln.of.double.slice} factors through ${\sf Ex}^N \Cref{slice.under.R.zero.by.factorizn.system}$.
Using that ${\sf sd}$ and ${\sf Ex}$ are adjoint to one another, the morphism \Cref{d.minus.one.sphere.in.gpd.compln.of.double.slice} is therefore presented by a map between quasi-categories ${\sf sd}^N( \SS^{d-1}) \to \Cref{slice.under.R.zero.by.factorizn.system}$.  
Finally, note that the (nerve of the) poset ${\sf sd}^N(\SS^{d-1})$ is the exit-path $(\infty,1)$-category associated to the $N$-fold iterated bary-centric subdivision of the triangulated sphere $\SS^{d-1}$. 
}
There exists a unique extension of this stratification of $S^{d-1}$ to a stratification of $D^d$ in which the interior is a single stratum; we likewise abuse notation by simply denoting this again by $D^d$.

Now, observe that the stratified space $\cone(D^d)$ -- the cone on $D^d$ -- has the property that $\Exit(\cone(D^d)) \simeq \Exit(S^{d-1})^{\lcone \rcone}$: the cone point over $D^d$ corresponds to the left cone point, while the interior of the disk corresponds to the right cone point.  Hence, the functor $\tilde{X}$ 
is equivalent data to a functor
\[
\Exit(\cone(D^d))
\simeq
\Exit(S^{d-1})^{\lcone \rcone}
\longra
\M
~,
\]
equipped with certain additional structures and satisfying certain conditions, which amount to the following on the corresponding proper constructible bundle $X \da \cone(D^d)$:
\begin{itemize}
\item its fiber over the cone point is identified with $D$,
\item its fiber over the interior of $D^d$ is identified with $\SS^1$,
\item its restriction to any exiting path starting at the cone point is an active morphism, and
\item its restriction to any exiting path from $S^{d-1}$ to the interior of $D^d$ is a refinement morphism.
\end{itemize}
We will use this proper constructible bundle to construct a nullhomotopy of the original map \Cref{d.minus.one.sphere.in.gpd.compln.of.double.slice}.

So, consider the link $\Link_D(X)$ of $D$ in $X$; this admits a composite proper constructible bundle map
\[
\Link_D(X)
\longra
D \times X_{|D^d}
\longra
D \times D^d
~.
\]
We then form the ``$D^d$-parametrized reversed cylinder'' of this map, namely the pushout
\begin{equation}
\label{pushout.to.get.Y.living.over.cone.on.d.disk}
\begin{tikzcd}
D \times D^d
\arrow{r}
\arrow{d}
&
\cylr \left( D \times D^d \longla \Link_D(X) \right)
\arrow{d}
\\
D
\arrow{r}
&
Y
\end{tikzcd}
\end{equation}
of stratified spaces, where the upper map is the injective constructible bundle given by the inclusion of the fiber over $\Delta^{\{0\}} \subset \Delta^1$ and the left map is the projection.  By construction, the pushout square \Cref{pushout.to.get.Y.living.over.cone.on.d.disk} maps to the defining pushout square
\[ \begin{tikzcd}
\Delta^{\{0\}} \times D^d
\arrow{r}
\arrow{d}
&
\Delta^1 \times D^d
\arrow{d}
\\
\Delta^{\{0\}}
\arrow{r}
&
\cone(D^d)
\end{tikzcd} \]
by proper constructible bundles, and hence in particular we obtain a proper constructible bundle map $Y \da \cone(D^d)$.  Also by construction, we see that $Y$ comes equipped with a canonical morphism $Y \ra X$ of proper constructible bundles over $\cone(D^d)$; unwinding the definitions, we see that the natural transformation
\[ \begin{tikzcd}[column sep=1.5cm]
\Exit(\cone(D^d))
\arrow[bend left]{r}[pos=0.42]{Y}[swap, transform canvas={yshift=-0.35cm}]{\Downarrow}
\arrow[bend right]{r}[swap,pos=0.44]{X}
&
\M
\end{tikzcd} \]
that this classifies determines a natural transformation
\[ \begin{tikzcd}[column sep=1.5cm]
\Exit(S^{d-1})
\arrow[bend left]{r}{\tilde{Y}}[swap, transform canvas={yshift=-0.35cm}]{\Downarrow}
\arrow[bend right]{r}[swap]{\tilde{X}}
&
\Cref{slice.under.R.zero.by.factorizn.system}
\end{tikzcd} \]
to our originally chosen functor $\tilde{X}$ representing our chosen map \Cref{d.minus.one.sphere.in.gpd.compln.of.double.slice}.

We complete the proof by constructing a nullhomotopy of the map $S^{d-1} \xra{|\tilde{Y}|} |\Cref{slice.under.R.zero.by.factorizn.system}|$ of spaces.  For this, observe first that the map $\Link_D(X) \ra X_{|D^d}$ of stratified spaces is a refinement, so that the map $Y \ra X$ is as well.  We take the open cylinder $\cylo(Y \longra X)$ of the latter, and then take the pushout
\begin{equation}
\label{push.out.to.get.Z.witnessing.nullhtpy.of.Y}
\begin{tikzcd}
D \times \Delta^1
\arrow{r}
\arrow{d}
&
\cylo(Y \longra X)
\arrow{d}
\\
D
\arrow{r}
&
Z
\end{tikzcd}
\end{equation}
of stratified spaces, where the upper map is the inclusion of the fiber over $\Delta^{\{0\}} \times \Delta^1$ (where $\Delta^{\{0\}} \subset \cone(D^d)$ denotes the cone point) and the left map is the projection.  By construction, the pushout square \Cref{push.out.to.get.Z.witnessing.nullhtpy.of.Y} maps to the pushout square
\[ \begin{tikzcd}
\Delta^{\{0\}} \times \Delta^1
\arrow{r}
\arrow{d}
&
\cone(D^d) \times \Delta^1
\arrow{d}
\\
\Delta^{\{0\}}
\arrow{r}
&
\cone(D^d \times \Delta^1)
\end{tikzcd} \]
by proper constructible bundles, and hence in particular we obtain a proper constructible bundle map $Z \da \cone(D^d \times \Delta^1)$, which is classified by a functor
\begin{equation}
\label{functor.classifying.proper.cbl.bdl.Z}
\Exit(\cone(D^d \times \Delta^1))
\longra
\M
~.
\end{equation}

We now construct three maps $\cone(D^d) \ra \cone(D^d \times \Delta^1)$ of stratified spaces:
\begin{itemize}
\item we write $i_0$ for the inclusion of $\cone(D^d \times \Delta^{\{0\}})$,
\item we write $i_1$ for any map that carries the interior to the interior, and whose restriction to $\cone(S^{d-1}) \subset \cone(D^d)$ is the composite map
\[
\cone(S^{d-1})
\longra
\cone(\DD^0)
\xlongra{\cone(p)}
\cone(D^d \times \Delta^1)
\]
in which $\DD^0 \xra{p} D^d \times \Delta^1$ selects an interior point, and
\item we write $i$ for any inclusion which extends the inclusion of $\cone(S^{d-1} \times \Delta^{\{0\}})$ and takes the interior of $D^d$ into the interior of $\cone(D^d \times \Delta^1)$.
\end{itemize}
On exit-path $\infty$-categories, these evidently participate in a diagram
\[ \begin{tikzcd}[column sep=1.5cm]
\Exit(\cone(D^d))
\arrow[bend left=60]{r}{\Exit(i_1)}[swap, transform canvas={yshift=-0.35cm}]{\Uparrow}
\arrow{r}[pos=0.6]{\Exit(i_0)}
\arrow[bend right=60]{r}[swap]{\Exit(i)}[transform canvas={yshift=0.35cm}]{\Downarrow}
&
\Exit(\cone(D^d \times \Delta^1))
\end{tikzcd} \]
of natural transformations; unwinding the definitions, we see that this precomposes with the functor \Cref{functor.classifying.proper.cbl.bdl.Z} to determine a span of diagrams
\[ \begin{tikzcd}[column sep=1.5cm]
\Exit(S^{d-1})
\arrow[bend left=60]{r}[swap, transform canvas={yshift=-0.25cm}]{\Uparrow}
\arrow{r}[pos=0.35]{\tilde{Y}}
\arrow[bend right=60]{r}[swap]{\tilde{X}}[transform canvas={yshift=0.25cm}]{\Downarrow}
&
\Cref{slice.under.R.zero.by.factorizn.system}
\end{tikzcd} \]
in which the upper functor is constant.  This completes the proof.

\end{proof}

\begin{lemma}
\label{t84}
There is a canonical identification between continuous monoids:
\begin{equation}
\label{e90}
\WW^{\op}
\xra{~\simeq~}
\End_{\Mfld^{\sofr}}(\SS^1)
~.
\end{equation}

\end{lemma}

\begin{proof}
The underlying stratified space of $\SS^1 \in \Mfld^{\sofr}$ has a single stratum, which has dimension 1.
It follows that each endomorphism $\SS^1 \to \SS^1$ in $\Mfld^{\sofr}$ is a idle morphism.
By definition of idle morphisms, this is to say that the continuous monoid $\End_{\Mfld^{\sofr}}(\SS^1)$ is the opposite of that of framed proper fiber bundle maps $\SS^1 \to \SS^1$ and composition between such.  
This latter monoid is evidently equivalent with that of oriented self-covering maps of $\SS^1$.
Reporting the degree of a self-covering map therefore defines a morphism between continuous monoids:
\begin{equation}
\label{e83}
{\rm degree}
\colon
\End_{\Mfld^{\sofr}}(\SS^1)
\longrightarrow
\NN^{\times}
~.
\end{equation}
For $r \in \NN^\times$, the fiber over $r$ is thusly identified as the space ${\sf Cov}^{\sf or}_r(\SS^1,\SS^1)$ of degree $r$ oriented self-covers of $\SS^1$.
Precomposing by oriented diffeomorphisms, $\Diff^{\sf or}(\SS^1)^{\op} \lacts {\sf Cov}^r(\SS^1,\SS^1)$ is a torsor.

Now, the morphism~\Cref{e83} between continuous monoids admits a standard section:
its value on $r\in \NN^\times$ is the self-cover $\SS^1 \xra{z\mapsto z^r} \SS^1$.
Together with the rotation action $\TT  \lacts \SS^1$, the section extends as a morphism between continuous monoids
\begin{equation}
\label{e84}
\WW^{\op}
~:=~
\NN^\times
\ltimes 
\TT
\longrightarrow
\End_{\Mfld^{\sofr}}(\SS^1)
\end{equation}
over $\NN^\times$.
For $r \in \NN^\times$, the fibers of this morphism over $r$ is the map between spaces
\[
\TT
\longrightarrow
{\sf Cov}^{\sf or}_r(\SS^1,\SS^1)
\]
selecting the $\TT$-orbit (via precomposition) of the rotation action on the domain $\SS^1$.
Such rotation action defines a map $\TT \to \Diff^{\sf or}(\SS^1)$.
It is routine to verify that this map is an equivalence.  
It follows that the map~\Cref{e84} is an equivalence between continuous monoids.

\end{proof}

\subsection{Comparing $\Disk^{\sofr}$ and $\Quiv$}

The cellular realization functor $\bDelta^{\op} \xra{ \langle - \rangle} \Disk^{\sofr}$ restricts as a functor $\bDelta_{\leq 1}^{\op} \xra{ \langle - \rangle} \Disk^{{\sf idl} , {\sofr}}$.  
Using~\S\ref{sec.recollections}\Cref{recollection.eight}, together with compactness of objects in $\Disk^{\sf sfr}$,
the associated restricted Yoneda functor factors through presheaves of finite sets:
\begin{equation}
\label{e78}
{
(\Disk^{{\sf idl} , {\sf sfr}})^{\op}
\longrightarrow
\digraphs
~\subset~
\PShv(\bDelta_{\leq 1})
}
~,\qquad
D
\mapsto
\Gamma_D
:=
\Hom_{\Disk^{{\sf idl} , {\sf sfr}}}\bigl( D , \langle [\bullet] \rangle \bigr)
~.
\end{equation}

\begin{lemma}
\label{t77'}
The functor \Cref{e78} is an equivalence between categories.  
In particular, it carries (the opposites of) closed cover diagrams to colimit diagrams.  

\end{lemma}

\begin{proof}
Right away, observe that the composite functor
$
\bDelta_{\leq 1}
\xra{\langle - \rangle} 
(\Disk^{{\sf idl} , {\sf sfr}})^{\op}
\xra{\Cref{e78}}
\digraphs
$
is identical with the Yoneda functor.
Therefore, for each $\Gamma \in \digraphs$, the canonical functor
\begin{equation}
\label{f30}
\bigl( (\bDelta_{\leq 1})_{/\Gamma_D} \bigr)^{\rcone}
\longrightarrow
\digraphs
\end{equation}
is a colimit diagram.

Let $D\in \Disk^{{\sf idl} , {\sf sfr}}$.
Consider the functor 
\begin{equation}
\label{f31}
\Exit(D) 
\longrightarrow 
(\bDelta_{\leq 1})_{/\Gamma_D}
\end{equation}
that is uniquely determined by declaring its value on the object $d\in \Exit(D)$ to be the canonical closed morphism from $D$ to the closure of the stratum in which $d\in D$ belongs, post-composed with the unique isomorphism 
\[
D \xra{~\rm cls~}\ol{D_d} \xra{~\cong~} \DD^i = \langle [i] \rangle
~.
\]
Observe that the functor \Cref{f31} is fully faithful, with image consisting of those idle morphisms $D \to \langle [i] \rangle$ that are, in fact, closed morphisms.  
In particular, the functor \Cref{f31} is final.  
Next, notice that the functor \Cref{f31} canonically fills the commutative diagram
\[
\begin{tikzcd}
\Exit(D)^{\rcone}
\arrow{d}[swap]{\Cref{t79}}
\arrow[dashed]{r}{\Cref{f31}}
&
\bigl( (\bDelta_{\leq 1})_{/\Gamma_D}\bigr)^{\rcone}
\arrow{d}{\Cref{f30}}
\\
(\Disk^{{\sf idl} , {\sf sfr}})^{\op}
\arrow{r}[swap]{\Cref{e78}}
&
\digraphs
\end{tikzcd}
~.
\]
By way of~\S\ref{sec.recollections}\Cref{recollection.nine}, 
it follows that the functor \Cref{e78} carries the (opposite of the) limit diagram $(\Exit(D)^{\op})^{\lcone} \to \Disk^{{\sf idl} , {\sf sfr}}$ to a colimit diagram.

We now prove the functor \Cref{e78} is fully faithful.
Let $D,E \in \Disk^{{\sf idl} , {\sf sfr}}$.
We must prove the map between sets,
\begin{equation}
\label{f32}
\Hom_{\Disk^{{\sf idl} , {\sf sfr}}}(D,E)
\xra{~\Cref{e78}~}
\Hom_{\digraphs}( \Gamma_E , \Gamma_D)
~,
\end{equation}
is a bijection.
Via the functor $\Exit(E)^{\rcone} \to (\Disk^{{\sf idl} , {\sf sfr}})^{\op}$ of~\S\ref{sec.recollections}\Cref{recollection.nine}, this map fits into a commutative diagram among sets
\[
\begin{tikzcd}[column sep=1.5cm]
\Hom_{\Disk^{{\sf idl} , {\sf sfr}}}(D,E)
\arrow{r}{\Cref{e78}}
\arrow{d}
&
\Hom_{\digraphs}( \Gamma_E , \Gamma_D)
\arrow{d}
\\
\underset{e \in \Exit(E)}\lim
\Hom_{\Disk^{{\sf idl} , {\sf sfr}}}(D, \ol{E_e})
\arrow{r}[swap]{\Cref{e78}}
&
\underset{e \in \Exit(E)}\lim
\Hom_{\digraphs}(\Gamma_{\ol{E_e}} , \Gamma_D)
\end{tikzcd}
~.
\]
Because $\Exit(E)^{\rcone} \to (\Disk^{{\sf idl} , {\sf sfr}})^{\op}$ is a colimit diagram, the left vertical map is a bijection.  
Established at the start of this proof is that the right vertical map is a bijection as well.
Therefore, the upper horizontal map is a bijection if and only if the lower horizontal map is a bijection.  
For each $e\in \Exit(E)$, there is an isomorphism $\ol{E_e} \cong \DD^i \cong \langle [i]\rangle$ for $i=0,1$.
Therefore, the lower horizontal map is a bijection provided the map~\Cref{f32} is a bijection in the cases that $E \cong \langle [i]\rangle$ for $i=0,1$.
In the case that $E \cong \langle [i]\rangle$, bijectivity of \Cref{f32} is an instance of the Yoneda lemma.

We now prove the functor \Cref{e78} is surjective on objects, which will complete this proof.
Let $\Gamma\in \digraphs$ be a finite directed graph.  
Its geometric realization $|\Gamma|$ is equipped with the structure of a 1-dimensional CW complex.  
Regard $|\Gamma|$ as a stratified space, in which a 0-dimensional stratum is a 0-cell of this CW structure, and a 1-dimensional stratum is the interior of a 1-cell.
The direction of the edges of $\Gamma$ supplies this stratified space with the structure of a solid 1-framing.  
In this way, we regard $|\Gamma|$ as a solidly 1-framed stratified space: $|\Gamma| \in \Mfld^{\sofr}$.  
Furthermore, as each stratum is a Euclidean space, $|\Gamma| \in \Disk^{\sf sfr} \subset \Mfld^{\sofr}$.  
Consider the morphism $\Gamma \to \Gamma_{|\Gamma|}$ in $\digraphs$ given as follows.
It is the natural transformation $\Gamma([\bullet]) \to \Hom_{\Disk^{{\sf idl} , {\sf sfr}})}\bigl( |\Gamma| , \langle [ \bullet ] \rangle \bigr)$ whose value on a vertex $v\in \Gamma([0])$ is the canonical closed morphism $|\Gamma| \xra{\sf cls} \{v\} \cong \langle [0] \rangle$, and whose value on a (possibly degenerate) edge $e \in \Gamma([1])$ is the canonical idle morphism $|\Gamma| \xra{\sf idl} \ol{|\Gamma|_e} \xra{!} \DD^1 \cong \langle [1]\rangle$.
By inspection, this morphism $\Gamma \to \Gamma_{|\Gamma|}$ in $\digraphs$ is an isomorphism.
Therefore, the functor \Cref{e78} is surjective on objects.

\end{proof}

\begin{lemma}
\label{t77}
The diagram among $\infty$-categories
\[
\begin{tikzcd}[column sep=1.5cm]
(\Disk^{\sf cls.cr , \sofr})^{\op}
\arrow{d}[swap]{\Cref{e78}}
\arrow{r}{\rm inclusion}
&
(\Disk^{{\sofr}})^{\op}
\arrow{d}{\fC}
&
\\
\digraphs
\arrow{r}[swap]{\Free}
&
\Cat_{(\infty,1)}
\end{tikzcd}
\]
canonically commutes.
In particular, for $D \in \Disk^{\sofr}$, the value $\fC(D) \simeq \Free(\Gamma_D)$ is the free category on the finite directed graph $\Gamma_D$.  

\end{lemma}

\begin{proof}
Consider the natural transformation $\Cref{e78} \to \Bigl( \fC \circ {\rm inclusion} \Bigr)_{|\bDelta_{\leq 1}^{\op}}$ between functors from $(\Disk^{\sf cls.cr,\sofr})^{\op}$ to $\PShv(\bDelta_{\leq 1})$ given by the canonical monomorphism
\[
\Hom_{\Disk^{\sf cls.cr,\sofr}} 
\Bigl( - , \langle [\bullet] \rangle \Bigr)
~\hookrightarrow~
\Hom_{\Disk^{\sofr}} 
\Bigl(-,\langle [\bullet] \rangle \Bigr)
~.
\]
By definition of $\fC$, and using that $\Free$ is (the restriction of) a left adjoint, this supplies a natural transformation $\Free\circ \Cref{e78} \to \fC \circ {\rm inclusion}$.
We will show it is by equivalences in $\Cat_{(\infty,1)}$.

Let $D\in \Disk^{\sf cls.cr,\sofr}$.
Using that the functor $\Cat_{(\infty,1)} \hookrightarrow \PShv(\bDelta) \xra{\rm restriction} \PShv(\bDelta_{\leq 1})$ is conservative, we need only show the map between spaces
\begin{equation}
\label{e92}
\Hom_{\Cat_{(\infty,1)}}
\Bigl(
[p]
,
\Free(\Gamma_D)
\Bigr)
\longrightarrow
\Hom_{\Cat_{(\infty,1)}}
\Bigl(
[p]
,
\fC(D)
\Bigr)
:=
\Hom_{\Disk^{\sofr}} \Bigl(D, \langle [p] \rangle \Bigr)
\end{equation}
is an equivalence for $p=0,1$.  
Note that every morphism $D \to \DD^0=\langle [0]\rangle$ is a closed morphism, which is ${\sf Cylr}$ applied to an inclusion $\DD^0 \hookrightarrow D$ of a 0-dimensional stratum of $D$.
This, together with inspection of the space of objects of the values of $\Free$ from \Cref{t27}, reveals that this map~\Cref{e92} is an equivalence in the case that $p=0$.

It remains to show~\Cref{e92} is an equivalence in the case that $p=1$.
Note that $\langle [1] \rangle = \DD^1$ is the solidly 1-framed 1-disk.
Let $D \xra{f} \DD^1$ be a morphism in $\Disk^{\sofr}$.
Given the $p=0$ case above, we can assume this morphism does not factor through $\DD^0 \xra{!} \DD^1$, the unique morphism from the $0$-disk.
By inspection of the spaces of morphisms of the values of $\Free$ from \Cref{t27}, 
we must show the following:
\begin{itemize}
\item[{\bf Claim.}]
The morphism $f$ uniquely factors 
\[
D \xra{{\sf Cylr}(\pi)} L \xra{{\sf Cylo}(\gamma)} \DD^1
\] 
in $\Disk^{\sofr}$
in which $L \xra{\gamma} \DD^1$ is a solidly 1-framed refinement map between solidly 1-framed stratified spaces and $D \xla{\pi} L$ is a solidly 1-framed proper constructible bundle map between solidly 1-framed stratified spaces with the property that $\pi$ carries 1-dimensional strata to 1-dimensional strata.  
\end{itemize}
We first establish the existence of such a factorization.
By definition of the $\infty$-category $\Disk^{\sofr}$, the morphism $f$ is represented by a proper constructible bundle $X \to \Delta^1$ equipped with a fiberwise solid 1-framing, together with identifications as solidly 1-framed stratified spaces of the fibers: $D \cong X_{|\Delta^{\{0\}}}$ and $\DD^1 \cong X_{|\Delta^{\{1\}}}$.  
Consider the link $L':= {\sf Link}_{X_{|\Delta^{\{0\}}}}(X)$.
By construction, it fits into a span among stratified spaces
\[
D
\cong
X_{|\Delta^{\{0\}}}
\xla{~\pi'~}
L'
\xra{~\gamma'~}
X_{|\Delta^{\{1\}}}
\cong
\DD^1
~,
\]
in which $\gamma'$ is a refinement map and $\pi'$ is a proper constructible bundle map.  
Consider the solid 1-framing on $L'$ pulled back along $\gamma'$ from the given solid 1-framing of $\DD^1$.  
The fiberwise solid 1-framing on $X \to \Delta^1$ supplies a solid 1-framing structure on $\pi'$, which is to say an identification, for each 1-dimensional stratum $D_\alpha \subset D$, of the solid 1-framing pulled back along $\pi'$ from $D_\alpha$ to that on each (necessarily 1-dimensional) stratum of $L'$ over $D_\alpha$.
Next, consider the stratified space $L := L'{/\sim}$ obtained from $L'$ by collapsing each connected component of a preimage by $\pi'$ of a 0-dimensional stratum in $D$.  
By construction, $L$ inherits a solid 1-framing from that of $L'$.
Also, there remains a refinement morphism $L \xra{\gamma} \DD^1$ and the map $\pi'$ factors through $L$ as a proper constructible bundle: $\pi'\colon L' \xra{\rm quotient} L \xra{\pi} D$, both of which retain the structure of being solidly 1-framed.  
This supplies the existence of the sought factorization.

It remains to establish the uniqueness of such a factorization.
Let $D \xra{x} K \xra{y} \DD^1$ be another such factorization.
We must show there is a unique isomorphism $K \xra{h} L$ together with identifications $h \circ x \simeq {\sf Cylr}(\pi)$ and $y \simeq {\sf Cylo}(\gamma) \circ h$.  
Because $\bDelta^{\op} \xra{\langle - \rangle} \Disk^{\sofr}$ is fully faithful, and $\bDelta^{\op}$ is gaunt, if such an isomorphism $h$ exists then it is unique.  
Because $\Disk^{\sofr}$ is an ordinary category, if such identifications $h \circ x \simeq {\sf Cylr}(\pi)$ and $y \simeq {\sf Cylo}(\gamma) \circ h$ exist then they are unique.

Now, the origin $0\in \DD^1$ can be regarded as an object $0\in \Exit(\DD^1)  \simeq \Exit(X_{|\Delta^{\{1\}}}) \subset \Exit(X)$ in the exit-path $\infty$-category of the total space of the proper constructible bundle $X \to \Delta^1$ underlying the given morphism $f$.
With respect to the fully faithful functor $\Exit(D) \simeq \Exit(X_{|\Delta^{\{0\}}}) \hookrightarrow \Exit(X)$, consider the $\infty$-overcategory $\Exit(D)_{/0}$.
The conditions on each of the factorizations $D \xra{x} K \xra{y} \DD^1$ and $D \xra{{\sf Cylr}(\pi)} L \xra{{\sf Cylo}(\gamma)} \DD^1$ are such that each of the resulting canonical functors over $\Exit(D)$, 
\[
\Exit\bigl( \langle [q] \rangle \bigr)
~\simeq~
\Exit(K)
\longrightarrow
\Exit(D)_{/0}
\qquad
\text{ and }
\qquad
\Exit\bigl( \langle [p] \rangle \bigr)
~\simeq~
\Exit(L)
\longrightarrow
\Exit(D)_{/0}
~,
\]
are equivalences.  
Consequently, $p=q$.
Because the cellular realization functor $\bDelta^{\op} \xra{\langle - \rangle} \Disk^{\sofr}$ is fully faithful, there exists a unique isomorphism $K \xra{h} L$ in $\Disk^{\sofr}$, which lies over $\DD^1$ (ie, $y = {\sf Cylo}(\gamma) \circ h$).
Finally, because the equivalences $\Exit(K) \xra{\simeq} \Exit(D)_{/0} \xla{\simeq} \Exit(L)$ lie over $\Exit(D)$, it follows from \Cref{t77'} that $h\circ x = {\sf Cylr}(\pi)$.

\end{proof}

\begin{cor}
\label{t80}
The functor $(\Disk^{\sofr})^{\op} \xra{\fC} \fCat_{(\infty,1)}$ carries (the opposites of) closed covers to colimit diagrams.  

\end{cor}

\begin{proof}
Observe that the functor~\Cref{e78} carries (the opposites of) closed covers to pushouts in $\digraphs \subset \PShv(\bDelta_{\leq 1})$.  
The result follows from the fact that the functor $\Free$, being a left adjoint, preserves colimits.

\end{proof}

\begin{prop}
\label{t75}
There is a canonical identification 
\[
\Disk^{\sofr} \xra{~\simeq~} \Quiv^{\op}
\]
between $(\infty,1)$-categories under $\bDelta^{\op}$.

\end{prop}

\begin{proof}
By \Cref{d12}, the $\infty$-category $\Quiv \subset \Cat_{(\infty,1)}$ is a full $\infty$-subcategory on the image of the functor $\digraphs \xra{\Free} \Cat_{(\infty,1)}$.
By definition, the inclusion 
$
\Disk^{\sf cls.cr , \sofr}
\hookrightarrow
\Disk^{\sofr}
$
is surjective on objects.  
Therefore, \Cref{t77} implies $\Quiv$ is precisely the full $\infty$-subcategory of $\Cat_{(\infty,1)}$ on the image of the functor $(\Disk^{\sofr})^{\op} \xra{\fC} \Cat_{(\infty,1)}$.
This supplies a functor,
\[
\fC
\colon
\Disk^{\sofr}
\longrightarrow
\Quiv^{\op}
~,
\]
which is surjective on objects.
\Cref{t6} and commutativity of \Cref{t82} together imply this functor is canonically under $\bDelta^{\op}$.
So the proof is complete upon showing this functor $(\Disk^{\sofr})^{\op} \xra{\fC} \Cat_{(\infty,1)}$ is fully faithful.

So let $D,E \in \Disk^{\sofr}$.
We seek to show the map 
\begin{equation}
\label{e80}
\Hom_{\Disk^{\sofr}}( D , E) 
\xra{~\fC~} 
\Hom_{\Cat_{(\infty,1)}} \bigl( \fC(E) , \fC(D) \bigr)
\end{equation}
is an equivalence between spaces.
Using the limit diagram $\sE(E) \to (\Disk^{\sofr})^{E/}$ of \Cref{t79}, this map~\Cref{e80} fits into a commutative diagram among spaces:
\[
\begin{tikzcd}
\Hom_{\Disk^{\sofr}}( D , E)
\arrow{d}[swap]{\Cref{t79}}
\arrow{r}{\fC}
&
\Hom_{\Cat_{(\infty,1)}} \bigl( \fC(E) , \fC(D) \bigr)
\arrow{d}
\\
\underset{e \in \Exit(E)^{\op}}\lim
\Hom_{\Disk^{\sofr}}( D , \ol{E_e})
\arrow{r}[swap]{\fC}
&
\underset{e \in \Exit(E)^{\op}}\lim
\Hom_{\Cat_{(\infty,1)}} \bigl( \fC(\ol{E_e}) , \fC(D) \bigr)
\end{tikzcd}
~.
\]
Because closed covers in $\Disk^{\sofr}$ are, in particular, limit diagrams, the vertical map on the left is an equivalence.
Through \Cref{t80}, the fact that the diagram \Cref{t79} is a limit diagram also implies the vertical map on the right is an equivalence.
Therefore, because each value of the diagram $\sE(E) \to (\Disk^{\sofr})^{E/}$ of \Cref{t79} is $\langle [p] \rangle$ for $p=0$ or $p=1$, the horizontal map on the top is an equivalence provided it is in the case that $E \simeq \langle [p] \rangle$ for some $p\geq 0$.

So let $p \geq 0$ and assume $E = \langle [p] \rangle$.
In this case, the map~\Cref{e80} is identical with the map
\[
\fC(D)([p])
~:=~
\Hom_{\Disk^{\sofr}}( D , \langle [p] \rangle ) \xra{\fC} \Hom_{\Cat_{(\infty,1)}} \bigl( \fC( \langle [p] \rangle ) , \fC(D) \bigr) 
~.
\]
An instance of the commutativity of the diagram of \Cref{t82} implies this map is an equivalence, as desired.

\end{proof}

\begin{observation}
\label{t85}
Through the equivalence of \Cref{t75}, the opposite of the full $\infty$-subcategory $\w{\bLambda} \subset \Quiv$ of \Cref{t6'} is identified with the full $\infty$-subcategory of $\Disk^{\sofr}$ consisting of those objects whose underlying stratified space admits a refinement morphism to a circle.

\end{observation}

\subsection{Comparing $\Mfld^{\sofr}$ and $\M$}

Consider the full $\infty$-subcategories 
\[
\Disk^{{\sf con},{\sofr}} 
~\subset~ 
\Disk^{{\sofr}}
\qquad
\text{ and }
~\qquad~
\Mfld^{{\sf con},{\sofr}} \subset \Mfld^{{\sofr}}
\]
respectively consisting of those objects whose underlying stratified space is connected.
Recall the full $\infty$-subcategories $\Quiv^{\sf con} \subset \Quiv$ and $\M^{\sf con}\subset \M$ from \Cref{d9} and \Cref{d7}/\Cref{t49.2}.
Note the factorizations:
\[
\bDelta
\xra{~\rho~}
\Quiv^{\sf con}
\subset
\Quiv
\qquad
\text{ and }
\qquad
\bDelta^{\op}
\xra{~\langle-\rangle~}
\Disk^{\sf con,\sofr}
\subset
\Disk^{\sofr}
~.
\]
\Cref{t75} evidently restricts as an equivalence
\begin{equation}
\label{e82}
\Disk^{\sf con,\sofr}
\xra{~\simeq~}
(\Quiv^{\sf con})^{\op}
\end{equation}
between $\infty$-categories under $\bDelta^{\op}$.

Consider the full $\infty$-subcategory 
\[
\cA ~:=~\Ar^{\sf ref}(\Mfld^{\sf con,\sofr})^{|\Disk^{\sf con,\sofr}}_{|\fB \End_{\Mfld^{\sf con,\sofr}}(\SS^1)}
~\subset~
\Ar(\Mfld^{\sf con,\sofr})
\]
consisting of those arrows in $\Mfld^{\sf con,\sofr}$ that are refinements to a circle from an object in $\Disk^{\sf con,\sofr}\subset \Mfld^{\sf con,\sofr}$. 
Evaluation at source defines a functor 
\begin{equation}
\label{e86}
\cA 
\longrightarrow
\Disk^{\sf con,\sofr}
\end{equation}
to the full $\infty$-subcategory $\Disk^{\sf con,\sofr} \subset \Mfld^{\sf con,\sofr}$;
evaluation at target defines a functor
\begin{equation}
\label{e85}
\cA
\longrightarrow
\fB \End_{\Mfld^{\sf con,\sofr}}(\SS^1)
\end{equation}
to the full $\infty$-subcategory of $\Mfld^{\sf con,\sofr}$ consisting of the circle $\SS^1 \in \Mfld^{\sf con,\sofr}$.
The definition of the parametrized join as a colimit therefore supplies a canonical functor
\begin{equation}
\label{e87}
\Disk^{\sf con,\sofr}
\underset{
\cA 
}
\bigstar
\fB \End_{\Mfld^{\sf con,\sofr}}(\SS^1)
\longrightarrow
\Mfld^{\sf con,\sofr}
~.
\end{equation}

\begin{lemma}
\label{t87}
The functor~\Cref{e85} is a coCartesian fibration.

\end{lemma}

\begin{proof}
To see that the functor~\Cref{e85} is a locally coCartesian fibration, we must show each solid diagram in $\Mfld^{\sf con,\sofr}$,
\[
\begin{tikzcd}
D
\arrow[dashed]{r}
\arrow{d}[swap]{\sf ref}
&
E
\arrow[dashed]{d}{\sf ref}
\\
\SS^1
\arrow{r}[swap]{\pi}
&
\SS^1
\end{tikzcd}
~,
\]
admits an initial filler in which the vertical morphism on the right is a refinement, as indicated.  
Now, the endomorhpism $\pi$ is necessarily a idle morphism.  
Via Definition~1.24 of~\cite{AFR2}, the space of idle morphisms $\SS^1 \to \SS^1$ in $\Mfld^{\sf con,\sofr}$ is canonically identified as the space of oriented covering maps $\SS^1 \la \SS^1$ (in fact, which reverses the source and target, as indicated).  
A sought filler is then supplied via pullback of the downward disk-refinements along the covering map corresponding to $\pi$.
The universal property of pullbacks implies this filler is initial among all such fillers.
We conclude that~\Cref{e85} is a locally coCartesian fibration.
Finally, because taking pullback is associative, such locally coCartesian morphisms are closed under composition.
Therefore, the functor~\Cref{e85} is a coCartesian fibration, as asserted.

\end{proof}

\begin{lemma}
\label{t89}
The functor~\Cref{e87} is an equivalence between $\infty$-categories.

\end{lemma}

\begin{proof}

Observe that each object in $\Mfld^{\sf con,\sofr}$ is either equivalent with $\SS^1$ or an object in $\Disk^{\sf con,\sofr} \subset \Mfld^{\sf con,\sofr}$.
Therefore, the functor~\Cref{e87} is surjective on objects.

It remains to show~\Cref{e87} is fully faithful.
Clearly, the restriction of~\Cref{e87} to $\Disk^{\sf con,\sofr}$ and to $\fB \End_{\Mfld^{\sf con,\sofr}}(\SS^1)$ is fully faithful.
Next, observe that $\Hom_{\Mfld^{\sofr}}(\SS^1 , D) = \emptyset$ for each $D\in \Disk^{\sofr}\subset \Mfld^{\sofr}$.
Indeed, for $X \to \Delta^1$ a proper constructible bundle, if its fiber over $\Delta^{\{0\}} \subset \Delta^1$ has no 0-dimensional strata, then its fiber over any point in $\Delta^1$ has no 0-dimensional strata.
Therefore, since $\SS^1$ has no 0-dimensional strata while each object $D\in \Disk^{\sf con,\sofr}$ has at least one 0-dimensional stratum, there are no proper constructible bundles $X \to \Delta^1$ for which $X_{|\Delta^{\{0\}}} \cong \SS^1$ and $X_{|\Delta^{\{1\}}} \in \Disk^{\sf con,\sofr}$.
It remains to show that, for $D\in \Disk^{\sf con,\sofr}$, the map between spaces of morphisms induced by~\Cref{e87},
\begin{equation}
\label{e89}
\Hom_{
\Disk^{\sf con,\sofr}
\underset{
\cA 
}
\bigstar
\fB \End_{\Mfld^{\sf con,\sofr}}(\SS^1)
}
\Bigl(
D
,
\SS^1
\Bigr)
\longrightarrow
\Hom_{\Mfld^{\sf con,\sofr}}(D,\SS^1)
~,
\end{equation}
is an equivalence.
By construction of the parametrized join, the domain of this map~\Cref{e89} is the $\infty$-groupoid-completion of the $\infty$-overundercategory
$
\cA^{D/}_{/\SS^1}
$.
\Cref{t87} implies the canonical functor $(\cA_{|\SS^1})^{D/} \to \cA^{D/}_{/\SS^1}$ is a right adjoint.
Therefore, the map~\Cref{e89} is an equivalence provided the resulting composite functor
\[
(\cA_{|\SS^1})^{D/}
\longrightarrow
\cA^{D/}_{/\SS^1}
\longrightarrow
\Hom_{\Mfld^{\sf con,\sofr}}(D,\SS^1)
~,
\]
witnesses an $\infty$-groupoid-completion.
Equivalently, for each point $(D \xra{f} \SS^1)$ in the codomain of this functor,
we must verify that the fiber $\bigl( (\cA_{|\SS^1})^{D/} \bigr)_{|f}$ has contractible $\infty$-groupoid-completion.
Now, recognize this $\infty$-category $\bigl( (\cA_{|\SS^1})^{D/} \bigr)_{|f} \simeq \bcD{\sf isk}(\SS^1)^{f/}$ as the $\infty$-undercategory with respect to the canonical functor $\bcD{\sf isk}(\SS^1) \to {\Disk^{\sofr}}_{/\SS^1}$.
Through Quillen's Theorem~A, \Cref{t88} ensures that this $\infty$-undercategory indeed has contractible classifying space.

\end{proof}

\begin{lemma}
\label{t86}
The functor~\Cref{e86} is fully faithful.
Its image consists of those objects $D \in \Disk^{\sf con,\sofr}$ that admit a refinement morphism $D\to \SS^1$ in $\Mfld^{\sf con,\sofr}$.
\end{lemma}

\begin{proof}
Such objects are clearly all of the objects in the image of~\Cref{e86}.
It remains to show this functor is fully faithful.
We first show the induced map $\Obj(\cA) \to \Obj(\Disk^{\sofr})$ is a monomorphism.
\begin{itemize}
\item[]
Let $D \in \Disk^{\sofr}$ be in the image of~\Cref{e86}.
Using that refinement morphisms in $\Mfld^{\sofr}$ are, in particular, homeomorphisms between underlying stratified spaces, observe that the post-composition witnesses the space 
$\Hom_{\Mfld^{{\sf ref},\sofr}}(D,\SS^1)$ as an $\Aut_{\Mfld^{\sf con,\sofr}}(\SS^1)$-torsor.  
Therefore, there is a unique object $(D \xra{{\sf ref}}\hat{D}) \in \bigl( \fB \End_{\Mfld^{\sofr}}(\SS^1) \bigr)^{D/^{{\sf ref}}} \subset (\Mfld^{\sf con,\sofr})^{D/^{{\sf ref}}}$.  
It follows that $\Obj(\cA) \to \Obj(\Disk^{\sofr})$ is a monomorphism.
\end{itemize}
Now, let $D \xra{f} E$ be a morphism in $\Disk^{\sofr}$ between two objects that are in the image of~\Cref{e86}.
To show~\Cref{e86} is fully faithful, we must show there is a unique lift along~\Cref{e86} of this morphism to $\cA$.
Using that $\Obj(\cA) \to \Obj(\Disk^{\sofr})$ is a monomorphism, this is to establish a unique filler in the diagram in $\Mfld^{\sofr}$:
\begin{equation}
\label{e88}
\begin{tikzcd}
D
\arrow{r}{f}
\arrow{d}[swap]{{\sf ref}}
&
E
\arrow{d}{{\sf ref}}
\\
\hat{D}
\arrow[dashed]{r}
&
\hat{E}
\end{tikzcd}
~.
\end{equation}
Using that refinement morphisms in $\Mfld^{\sofr}$ are, in particular, homeomorphisms between underlying stratified spaces, any such filler must be unique.  
So we need only ensure such a filler exists.  
By definition of $\Mfld^{\sofr}$, the composite morphism $D \xra{f} E \xra{\sf ref} \hat{E}$ is implemented by a proper constructible bundle $X \to \Delta^1$ with a fiberwise solid 1-framing, together with identifications $X_{|\Delta^{\{0\}}} \simeq D$ and $X_{|\Delta^{\{1\}}} \simeq \hat{E}$ in $\Mfld^{\sofr}$.  
Consider the link $\sL := {\sf Link}_{X_{|\Delta^{\{0\}}}}(X)$.
By construction, it fits into a span among stratified spaces $X_{|\Delta^{\{0\}}} \xla{\pi} \sL \xra{\gamma} X_{|\Delta^{\{1\}}}$ in which $\pi$ is a proper constructible bundle and $\gamma$ is a refinement.  
Pulling back the solid 1-framing of $X_{|\Delta^{\{1\}}}\simeq \hat{E}$ along $\gamma$ supplies a solid 1-framing on $\sL$ for which the maps $\pi$ and $\gamma$ are solidly 1-framed, and as so they factor the composite $D \xra{f} E \xra{\sf ref} \hat{E}$ in $\Mfld^{\sofr}$:
\[
\begin{tikzcd}[column sep=1.5cm]
D
\arrow{r}{f}
\arrow{d}[swap]{{\sf Cylr}(\pi)}
&
E
\arrow{d}{{\sf ref}}
\\
\sL
\arrow{r}[swap]{{\sf Cylo}(\gamma)}
&
\hat{E}
\end{tikzcd}
~.
\]
Using that $\pi$ respects solid 1-framings, for each $0$-dimensional stratum $d \in D^{(0)}\subset D$, each connected component of the constructible closed subspace $\pi^{-1}(d) \subset \sL$ admits a framed refinement onto $\DD^0$ or $\DD^1$.  
Collapsing to a point each such component that refines onto $\DD^1$ results in a solidly 1-framed stratified space $\sL'$, still fitting into a span among solidly 1-framed stratified spaces $D \xla{\pi'} \sL' \xra{\gamma'} \hat{E}$ in which $\gamma'$ is a solidly 1-framed refinement and $\pi'$ is a solidly 1-framed proper constructible bundle with the feature that $i$-dimensional strata are carried to $i$-dimensional strata.
By construction, these data fit into a diagram in $\Mfld^{\sofr}$:
\[
\begin{tikzcd}[column sep=1.5cm]
D
\arrow{r}{f}
\arrow{d}[swap]{{\sf Cylr}(\pi')}
&
E
\arrow{d}{{\sf ref}}
\\
\sL'
\arrow{r}[swap]{{\sf Cylo}(\gamma')}
&
\hat{E}
\end{tikzcd}
~.
\]
Using that $\Obj(\cA) \to \Obj(\Disk^{\sofr})$ is a monomorphism, there is an identification $\hat{\sL'} \xra{\simeq} \hat{E}$ in $\Mfld^{\sofr}$ under $\sL'$ from the underlying framed 1-manifold of $\sL'$.
Now, being a constructible bundle, the map $\pi'$ restricts over each 1-dimensional stratum of $D$ as a framed covering space.
Because $\pi'$ is solidly 1-framed, it follows that it restricts over a neighborhood of each 0-dimensional stratum as a solidly 1-framed local homeomorphism.
Since $\pi'$ is proper, we conclude that $\pi'$ is a solidly 1-framed covering space. 
In other words, $\pi'$ is is the base-change of a framed covering space, necessarily from $\hat{\sL'}$ to $\hat{D}$:
there is a filler among solidly 1-framed stratified spaces
\[
\begin{tikzcd}
\sL'
\arrow{d}[swap]{\pi'}
\arrow{r}{\sf ref}
&
\hat{\sL'}
\arrow[dashed]{d}{\hat{\pi'}}
\\
D
\arrow{r}[swap]{\sf ref}
&
\hat{D}
\end{tikzcd}
\]
witnessing a pullback.
The filler in this diagram among solidly 1-framed stratified spaces determines a filler in the diagram in $\Mfld^{\sofr}$:
\[
\begin{tikzcd}[column sep=2cm]
D
\arrow{rr}{f}
\arrow{dd}[swap]{{\sf ref}}
\arrow{rd}[sloped]{{\sf Cylr}(\pi')}
&
&
E
\arrow{dd}{{\sf ref}}
\\
&
\sL'
\arrow{d}[swap]{\sf ref}
\arrow{rd}[sloped]{{\sf Cylo}(\gamma')}
\\
\hat{D}
\arrow[dashed]{r}[swap]{{\sf Cylr}(\hat{\pi'})}
&
\hat{\sL'}
\arrow{r}
&
\hat{E}
\end{tikzcd}
~.
\]
The desired filler in~\Cref{e88} is the bottom horizontal composite.

\end{proof}

\begin{lemma}
\label{t83}
The identification~\Cref{e82} extends as an identification 
\[
\Mfld^{\sf con,\sofr}
~\simeq~ 
\M^{\sf con} 
\]
between $(\infty,1)$-categories under $\bDelta^{\op}$.

\end{lemma}

\begin{proof}

We explain the diagram among $\infty$-categories:
\begin{equation}
\label{e91}
\begin{tikzcd}[column sep=1.5cm]
\Disk^{\sf con,\sofr}
&
\cA
\arrow{r}{\ev_t}
\arrow{l}[swap]{\ev_s}{\Cref{e86}}
&
\fB \End_{\Mfld^{\sf con,\sofr}}(\SS^1)
\\
\Disk^{\sf con,\sofr}
\arrow{u}[sloped, anchor=south]{=}
\arrow{d}[sloped, anchor=south]{\sim}[swap]{\Cref{e82}}
&
\cA_{|\BW^{\op}}
\arrow{d}[sloped, anchor=north]{\sim}
\arrow{u}[sloped, anchor=south]{\sim}
\arrow{r}
\arrow{l}
&
\BW^{\op}
\arrow{u}[sloped, anchor=south]{\sim}[swap]{\Cref{e90}}
\arrow{d}[sloped, anchor=south]{=}
\\
(\Quiv^{\sf con})^{\op}
&
\w{\bLambda}^{\op}
\arrow{r}{\Cref{paracyclic.cyclic.epicyclic}}[swap]{\rm localization}
\arrow{l}{\rm inclusion}
&
\BW^{\op}
\end{tikzcd}
~.
\end{equation}
The top-right square is defined as a pullback.
\Cref{t84} states that the functor~\Cref{e90} is an equivalence, which implies the upward functor $\cA_{|\BW^{\op}} \to \cA$ is an equivalence.
As noted earlier in the body, \Cref{t75} implies the functor $\Disk^{\sf con,\sofr} \xra{\Cref{e82}} (\Quiv^{\sf con})^{\op}$ is an equivalence.
\Cref{t86} thereafter implies the composite functor $\cA \to \Disk^{\sf con,\sofr} \xra{\Cref{e82}} (\Quiv^{\sf con})^{\op}$ is fully faithful.
\Cref{t86} implies the image of this fully faithful composite functor is $\w{\bLambda} \subset \Quiv^{\sf con}$ the cyclically finite directed graphs.  
This establishes the equivalence $\cA_{|\BW^{\op}} \xra{\simeq} (\w{\bLambda}^{\op})^{\lcone}$, as indicated.
Finally, inspecting the construction of the morphism $\NN^\times \ltimes \TT = \WW^{\op} \xra{\simeq} \End_{\Mfld^{\sf con,\sofr}}(\SS^1)$ over $\NN^\times$ and under $\TT$ reveals that the bottom right square canonically commutes.

Reading~\Cref{e91} as a (vertical) span by equivalences of (horizontal) spans, we obtain equivalences among parametrized joins as in this diagram among $\infty$-categories:
\[
\begin{tikzcd}[column sep=1.5cm, row sep=1.5cm]
\Disk^{\sf con,\sofr}
\underset{
\cA_{|\BW^{\op}}
}
\bigstar
\BW^{\op}
\arrow{r}{\sim}
\arrow{d}[sloped, anchor=north]{\sim}
&
\Disk^{\sf con,\sofr}
\underset{
\cA
}
\bigstar
\fB \End_{\Mfld^{\sf con,\sofr}}(\SS^1)
\arrow{r}{\Cref{e87}}[swap]{\simeq}
&
\Mfld^{\sf con,\sofr}
\arrow[dashed]{d}[sloped, anchor=south]{\sim}
\\
(\Quiv^{\sf con})^{\op}
\underset{
\w{\bLambda}^{\op}
}
\bigstar
\BW^{\op}
\arrow{rr}{\sim}[swap]{\Cref{f51}}
&
&
\M^{\sf con}
\end{tikzcd}
~.
\]
\Cref{t51} states that the functor~(\ref{f51}) is an equivalence.
\Cref{t89} states that the functor~\Cref{e87} is an equivalence.
We conclude a unique dashed filler, which is necessarily an equivalence.  

To finish the identification $\Mfld^{\sf con,\sofr} \simeq \M^{\sf con}$ indeed lies under $\bDelta^{\op}$ because, by construction, it extends the identification $\Disk^{\sf con,\sofr} \simeq (\Quiv^{\sf con})^{\op}$ of \Cref{t75}.

\end{proof}

\begin{prop}
\label{t76}
The identification of \Cref{t75} extends as an identification 
\[
\Mfld^{\sofr}
~\simeq~ 
\M
\]
between $(\infty,1)$-categories under $\bDelta^{\op}$.

\end{prop}

\begin{proof}

\Cref{t23} states that $\M$ is freely generated from its full $\infty$-subcategory $\M^{\sf con}\subset \M$ via categorical products.
Using that coproducts of underlying stratified spaces implements products in $\Mfld^{\sofr}$, the full $\infty$-subcategory $\Mfld^{{\sf con},{\sofr}} \subset \Mfld^{{\sofr}}$ freely generates via categorical products.  
The result is therefore implied by \Cref{t83}.

\end{proof}

\begin{notation}
Denote the equivalence of \Cref{t76} as 
\[
\ul{(-)} \colon \Mfld^{\sofr} 
\longrightarrow
\M
~,\qquad
M
\longmapsto
\ul{M}
~.
\]
\end{notation}

We can now articulate, and prove, the sense in which combinatorial factorization homology agrees with factorization homology.
\begin{theorem}
\label{t90}
Let $\cX$ be an $\infty$-category that admits finite limits and geometric realizations such that, for each $X\in \cX$, the functor $X\times - \colon \cX \to \cX$ preserves geometric realizations.
The diagram among $\infty$-categories,
\[
\begin{tikzcd}
&
\fCat_1[\cX]
\arrow{rd}[sloped]{\int}
\arrow{ld}[sloped]{\int'}
\\
\Fun(\Mfld^{\sofr},\cX)
\arrow{rr}[swap]{\ul{(-)}^\ast}{\sim}
&&
\Fun(\M , \cX)
\end{tikzcd}
~,
\]
canonically commutes.  
In particular, for each category-object $\cC$ in $\cX$, and each solidly 1-framed stratified space $M$, there is a canonical identification in $\cX$:
\[
\int'_M \cC
~\simeq~
\int_{\ul{M}} \cC
~.
\]

\end{theorem}

\begin{proof}

\Cref{t75} and \Cref{t76} supply a commutative diagram among $\infty$-categories:
\[
\begin{tikzcd}[column sep=1.5cm, row sep=1.5cm]
\bDelta^{\op}
\arrow{r}{\rho}
&
\Quiv^{\op}
\arrow{r}{\delta}
&
\M
\\
\bDelta^{\op}
\arrow{r}[swap]{\langle - \rangle}
\arrow{u}[sloped, anchor=north]{=}
&
\Disk^{\sofr}
\arrow{r}[swap]{\iota}
\arrow{u}[swap]{\ul{(-)}}
&
\Mfld^{\sofr}
\arrow{u}[swap]{\ul{(-)}}
\end{tikzcd}
~,
\]
in which the vertical functors are equivalences.
Applying $\Fun(-,\cX)$ to this diagram gives a commutative diagram among $\infty$-categories
\[
\begin{tikzcd}[column sep=1.5cm, row sep=1.5cm]
\Fun(\bDelta^{\op},\cX)
\arrow[leftarrow]{r}{\rho^*}
&
\Fun(\Quiv^{\op},\cX)
\arrow[leftarrow]{r}{\delta^*}
&
\Fun(\M,\cX)
\\
\Fun(\bDelta^{\op},\cX)
\arrow[leftarrow]{r}[swap]{\langle - \rangle^*}
\arrow[leftarrow]{u}[sloped, anchor=south]{=}
&
\Fun(\Disk^{\sofr},\cX)
\arrow[leftarrow]{r}[swap]{\iota^*}
\arrow[leftarrow]{u}[swap]{\ul{(-)}^*}
&
\Fun(\Mfld^{\sofr},\cX)
\arrow[leftarrow]{u}[swap]{\ul{(-)}^*}
\end{tikzcd}
\]
in which the vertical functors are equivalences.  
Taking right adjoints of the left horizontal functors, and left adjoints of the right horizontal functors, results in a commutative diagram among $\infty$-categories
\[
\begin{tikzcd}[column sep=1.5cm, row sep=1.5cm]
\Fun(\bDelta^{\op},\cX)
\arrow{r}{\rho_*}
&
\Fun(\Quiv^{\op},\cX)
\arrow{r}{\delta_!}
&
\Fun(\M,\cX)
\\
\Fun(\bDelta^{\op},\cX)
\arrow{r}[swap]{\langle - \rangle_*}
\arrow[leftarrow]{u}[sloped, anchor=south]{=}
&
\Fun(\Disk^{\sofr},\cX)
\arrow{r}[swap]{\iota_!}
\arrow[leftarrow]{u}[swap]{\ul{(-)}^*}
&
\Fun(\Mfld^{\sofr},\cX)
\arrow[leftarrow]{u}[swap]{\ul{(-)}^*}
\end{tikzcd}
\]
in which the vertical functors are equivalences.  
The result follows by \Cref{d21} of the functor $\int$ and by the definition of the functor $\int'$.  

\end{proof}

Recall from \Cref{d111} the notion of an excision site.  
The next explains how excision for factorization homology works intrinsic to $\Mfld^{\sofr}$.
\begin{prop}
\label{s100}
Through \Cref{t76}, an excision site $(\w{\Gamma},S,\varphi , M')$ corresponds to a pair $(M,S)$ in which $M \in \Mfld^{\sofr}$ and $S \subset M^{(1)}$ is a finite subset of its 1-dimensional strata.
Furthermore, such a pair $(M,S)$ determines a simplicial object 
\[
\bDelta^{\op}
\xra{~M_\bullet~}
\Mfld^{\sofr}
\]
whose colimit $|M_\bullet| \simeq M$, and such that each canonical morphism $M_p \to M$ is a refinement.
Moreover, for $\cC \in \fCat_1[\cX]$ a category-object in $\cX$, the canonical morphism in $\cX$,
\[
\left|
\int_{M_\bullet} \cC
\right|
\xra{~\simeq~}
\int_M \cC
~,
\]
is an equivalence.

\end{prop}

\begin{proof}
Let $(\w{\Gamma},S,\varphi, M')$ be an excision site in the sense of \Cref{d111}.
Take $M := \colim\left( \bDelta^{\op} \xra{M_\bullet} \M \right) \in \M$ as in \Cref{t4.5}.
Through \Cref{t76}, this object in $\M$ corresponds to an object $M\in \Mfld^{\sofr}$, which we give the same notation.  
The object $\Gamma_0 \sqcup M'\in \M$ corresponds to an object $\w{M} \in \Mfld^{\sofr}$, which is equipped with a refinement $\w{M} \to M$.  
The site of this refinement is a finite subset $S \subset M$.  
In this way, an excision site in the sense of \Cref{d111} determines a pair $(M,S)$ as in the statement of the proposition.

Next, let $(M,S)$ be as in the statement of the proposition.
Consider the coarsest refinement $\w{M} \xra{~r~} M$ such that there is containment 
$r^{-1}S \subset \w{M}^{(0)}$ in the subset of 0-dimensional strata.
Because refinements are homeomorphisms between underlying topological spaces, the map $r^{-1}S \xra{r_|} S$ is a bijection.
So let us simply denote $r^{-1}S = S$.
Now, consider the blow-up ${\sf Bl}_S(\w{M})$ of $\w{M}$ at $S$ -- it is a stratified space with boundary ${\sf Link}_S(\w{M})\subset {\sf Bl}_S(\w{M})$, which fits into a pushout diagram among stratified spaces,
\[
\xymatrix{
{\sf Link}_S(\w{M})
\ar[rr]
\ar[d]
&&
{\sf Bl}_S(\w{M})
\ar[d]
\\
S
\ar[rr]
&&
\w{M}
,
}
\]
in which the downward maps are proper constructible bundles. 
As so, ${\sf Bl}_S(\w{M})$ inherits a solid 1-framing.
Observe ${\sf Link}_S(\w{M}) \cong (\SS^0)^{\amalg S}$ is a disjoint union of $\SS^0$, one cofactor for each element in $S$.
Let $D \subset {\sf Bl}_S(\w{M})$ be the union of those connected components that contain ${\sf Link}_S(\w{M})$.
By design, each stratum of $D$ contains a 0-dimensional stratum in its closure.  
Therefore $D\in \Disk^{\sofr}$.  
Let $\w{\Gamma} \in \Quiv$ correspond to $D\in \Disk^{\sofr}$ through \Cref{t75}.
Let $(\SS^0)^{\amalg S} \xra{\varphi} \w{\Gamma}$ correspond through \Cref{t74} to $(\SS^0)^{\amalg S} \cong {\sf Link}_S(\w{M}) \hookrightarrow {\sf Bl}_S(\w{M})$.
Let $M'\in \M$ correspond to ${\sf Bl}_S(\w{M}) \setminus D \in \Mfld^{\sofr}$ through \Cref{t76}.
Then $(\w{\Gamma},S,\varphi,M')$ is an excision site.  

Observe that the composite association $(M,S) \mapsto (\w{\Gamma},S,\varphi,M') \mapsto (M,S)$ is the identity.  
Observe that the composite association $(\w{\Gamma},S,\varphi, M')\mapsto (M,S) \mapsto (\w{\Gamma}' , S , \varphi , M'')$ is not necessarily the identity.
However, the simplicial objects $\bDelta^{\op} \to \M$ determined by $(\w{\Gamma},S,\varphi,M')$ is identical with that determined by $(\w{\Gamma}',S,\varphi,M'')$.
This justifies the first statement of the proposition.

The second statement follows through \Cref{t76} from \Cref{t4.5}.
The third statement follows through \Cref{t90} from \Cref{t4.4}.

\end{proof}

\bibliographystyle{amsalpha}
\bibliography{circle-refs}{}

\providecommand{\bysame}{\leavevmode\hbox to3em{\hrulefill}\thinspace}
\providecommand{\MR}{\relax\ifhmode\unskip\space\fi MR }
\providecommand{\MRhref}[2]{%
  \href{http://www.ams.org/mathscinet-getitem?mr=#1}{#2}
}
\providecommand{\href}[2]{#2}
\begin{thebibliography}{AFMGR}

\bibitem[AF18]{flagged}
David Ayala and John Francis, \emph{Flagged higher categories}, Topology and
  quantum theory in interaction, Contemp. Math., vol. 718, Amer. Math. Soc.,
  2018, pp.~137--173.

\bibitem[AF20]{fibrations}
\bysame, \emph{Fibrations of {$\infty$}-categories}, High. Struct. \textbf{4}
  (2020), no.~1, 168--265.

\bibitem[AFMGR]{enriched1}
David Ayala, John Francis, Aaron Mazel-Gee, and Nick Rozenblyum,
  \emph{Factorization homology of enriched $\infty$-categories}, available at
  \texttt{arXiv:1710.06414}.

\bibitem[AFR18]{AFR2}
David Ayala, John Francis, and Nick Rozenblyum, \emph{Factorization homology
  {I}: {H}igher categories}, Adv. Math. \textbf{333} (2018), 1042--1177.

\bibitem[AFR19]{AFR1}
\bysame, \emph{A stratified homotopy hypothesis}, J. Eur. Math. Soc. (JEMS)
  \textbf{21} (2019), no.~4, 1071--1178.

\bibitem[AMGRa]{trace}
David Ayala, Aaron Mazel-Gee, and Nick Rozenblyum, \emph{The geometry of the
  cyclotomic trace}, available at \texttt{arXiv:1710.06409}.

\bibitem[AMGRb]{cyclo}
\bysame, \emph{A naive approach to genuine {$G$}-spectra and cyclotomic
  spectra}, available at \texttt{arXiv:1710.06416}.

\bibitem[Bat98]{Batanin}
M.~A. Batanin, \emph{Computads for finitary monads on globular sets}, Higher
  category theory ({E}vanston, {IL}, 1997), Contemp. Math., vol. 230, Amer.
  Math. Soc., Providence, RI, 1998, pp.~37--57.

\bibitem[BFG94]{BFG-epi}
D.~Burghelea, Z.~Fiedorowicz, and W.~Gajda, \emph{Power maps and epicyclic
  spaces}, J. Pure Appl. Algebra \textbf{96} (1994), no.~1, 1--14.

\bibitem[BM15]{BluMan}
Andrew~J. Blumberg and Michael~A. Mandell, \emph{The homotopy theory of
  cyclotomic spectra}, Geom. Topol. \textbf{19} (2015), no.~6, 3105--3147.

\bibitem[Con83]{connes}
Alain Connes, \emph{Cohomologie cyclique et foncteurs {${\rm Ext}^n$}}, C. R.
  Acad. Sci. Paris S\'er. I Math. \textbf{296} (1983), no.~23, 953--958.

\bibitem[GJ93]{GetzlerJones-paracyclic}
Ezra Getzler and John D.~S. Jones, \emph{The cyclic homology of crossed product
  algebras}, J. Reine Angew. Math. \textbf{445} (1993), 161--174.

\bibitem[HM97]{HessMad-Witt}
Lars Hesselholt and Ib~Madsen, \emph{On the {$K$}-theory of finite algebras
  over {W}itt vectors of perfect fields}, Topology \textbf{36} (1997), no.~1,
  29--101.

\bibitem[Lod92]{Loday}
Jean-Louis Loday, \emph{Cyclic homology}, Grundlehren der mathematischen
  Wissenschaften [Fundamental Principles of Mathematical Sciences], vol. 301,
  Springer-Verlag, Berlin, 1992, Appendix E by Mar\'{\i}a O. Ronco.

\bibitem[Lura]{LurieHA}
Jacob Lurie, \emph{Higher algebra}, available from the author's website.

\bibitem[Lurb]{LurieRot}
\bysame, \emph{Rotation invariance in algebraic {K}-theory}, available from the
  author's website.

\bibitem[Lur09a]{LurieHTT}
\bysame, \emph{Higher topos theory}, Annals of Mathematics Studies, vol. 170,
  Princeton University Press, Princeton, NJ, 2009.

\bibitem[Lur09b]{Lurie-cobordism}
\bysame, \emph{On the classification of topological field theories}, Current
  developments in mathematics, 2008, Int. Press, Somerville, MA, 2009,
  pp.~129--280.

\bibitem[Mad94]{Madsen-trace}
Ib~Madsen, \emph{Algebraic {$K$}-theory and traces}, Current developments in
  mathematics, 1995 ({C}ambridge, {MA}), Int. Press, Cambridge, MA, 1994,
  pp.~191--321.

\bibitem[Rez01]{Rezk1}
Charles Rezk, \emph{A model for the homotopy theory of homotopy theory}, Trans.
  Amer. Math. Soc. \textbf{353} (2001), no.~3, 973--1007.

\bibitem[Str76]{Street-2computads}
Ross Street, \emph{Limits indexed by category-valued {$2$}-functors}, J. Pure
  Appl. Algebra \textbf{8} (1976), no.~2, 149--181.

\bibitem[Tre09]{Treumann}
David Treumann, \emph{Exit paths and constructible stacks}, Compos. Math.
  \textbf{145} (2009), no.~6, 1504--1532.

\end{thebibliography}

\end{document}